\documentclass[reqno,a4paper]{amsart}
\usepackage[english,activeacute]{babel}
\usepackage[dvipsnames]{xcolor}
\usepackage{amssymb,amsmath,amsthm,amsfonts,mathrsfs,hyperref,mathtools}
\usepackage[font=footnotesize]{caption}
\usepackage{tikz}
\usetikzlibrary{intersections,shapes,trees,arrows,spy,positioning,decorations,arrows.meta,decorations.pathreplacing,patterns,calc}
\usepackage{tikz-cd}
\usepackage{tkz-euclide}
\usepackage{pgfplots} 
\usepackage[numbers,square,sort]{natbib}
\usepackage{tabularx}
\usepackage{enumitem}  %Proper enumerate 
\usepackage[T1]{fontenc}
\usepackage[foot]{amsaddr}
\usepackage{url}  
\usepackage{mathabx}
\usepackage{graphicx}

\theoremstyle{plain}
\newtheorem{theorem}{Theorem}[section]
\newtheorem{lem}[theorem]{Lemma}
\newtheorem{lemma}[theorem]{Lemma}
\newtheorem{proposition}[theorem]{Proposition}

\newtheorem{corollary}[theorem]{Corollary}
\theoremstyle{definition}

\newtheorem{definition}[theorem]{Definition}

\theoremstyle{remark}
\newtheorem{remark}[theorem]{Remark}

\numberwithin{equation}{section}

\newcommand{\Cb}  {{\mathbb C}}
\newcommand{\Db}  {{\mathbb D}}
\newcommand{\Nb}  {{\mathbb N}}
\newcommand{\PP}  {{\mathbb P}}

\newcommand{\As} {{\mathcal A}}

\newcommand{\Cs} {{\mathcal C}}
\newcommand{\Ds} {{\mathcal D}}

\newcommand{\Ls} {{\mathcal L}}

\newcommand{\Ps} {{\mathcal P}}

\newcommand{\Rs} {{\mathcal R}}

\newcommand{\Ts} {{\mathcal T}}

\newcommand{\Ae} {{\mathscr{A}}}
\newcommand{\Ce} {{\mathscr{C}}}

\newcommand{\dd}{\, \mathrm{d}}
\newcommand{\ee}{\mathrm{e}}
\newcommand{\trans}{\raisebox{1pt}{$\scriptscriptstyle\mathsf{T}$}}

\newcommand{\stirlingii}{\genfrac{\{}{\}}{0pt}{}}

\renewcommand{\P}{\mathbb P}

\newcommand{\E}{\mathbb E}
\newcommand{\R}{\mathbb R}

\newcommand{\N}{\mathbb N}
\newcommand{\ind}{1\!\kern-1pt \mathrm{I}}
\newcommand\1{{1\!\kern-1pt \mathrm{I}}}
\newcommand\cA{{\mathcal A}}

\newcommand\Ys{{\mathcal Y}}

\newcommand{\scO}{{\scriptstyle\mathcal{O}}}
\newcommand{\sscO}{{\scriptscriptstyle\mathcal{O}}}

\newcommand{\defeq}{\mathrel{\mathop:}=}
\newcommand{\eqdef}{=\mathrel{\mathop:}}

\definecolor{darkgreen}{rgb}{0.0, 0.2, 0.13}
\definecolor{lightbrown}{rgb}{0.71, 0.4, 0.11}

\def \globalscale {1.000000}

\pgfdeclarepatternformonly{soft crosshatch}{\pgfqpoint{-1pt}{-1pt}}{\pgfqpoint{4pt}{4pt}}{\pgfqpoint{3pt}{3pt}}%
{
	\pgfsetstrokeopacity{1}
	\pgfsetlinewidth{0.4pt}
	\pgfpathmoveto{\pgfqpoint{3.1pt}{0pt}}
	\pgfpathlineto{\pgfqpoint{0pt}{3.1pt}}
	\pgfpathmoveto{\pgfqpoint{0pt}{0pt}}
	\pgfpathlineto{\pgfqpoint{3.1pt}{3.1pt}}
	\pgfusepath{stroke}
}

\newcommand*{\faf}[2]{
	{#1}^{\mspace{2mu}\underline{\mspace{-2mu}\smash{#2}\mspace{-2mu}}\mspace{2mu}}%
}

\newcommand*{\exdesc}[4]{
	M^{#1}_{{#2},{#3},{#4}}
}

\newcommand{\imu}{\mathfrak{i}}
\newcommand{\narrow}{\scalebox{1.3}{\begin{tikzpicture}\draw[-{Latex[length=1.5mm,fill=black]}] (0,0) -- (.2,0);
		\end{tikzpicture}}}

\newcommand{\sarrow}{\scalebox{1.3}{\begin{tikzpicture}
		\draw[-{Latex[length=1.5mm,fill=white]}] (0,0) -- (.2,0);
		\end{tikzpicture}}}

\title{Lines of descent in a Moran model with frequency-dependent selection and mutation}

\author{E. Baake$^{1}$}

\address{$^1$Faculty of Technology, Bielefeld University, Postbox 100131, 33501 Bielefeld, Germany}
\email{ebaake@techfak.uni-bielefeld.de}

\author{L. Esercito$^1$}
\email{lesercito@techfak.uni-bielefeld.de}

\author{S. Hummel$^{1,2}$}
\email{shummel@berkeley.edu}
\address{$^2$Department of Statistics, University of California, 367 Evans Hall, Berkeley, CA 94720-3860, U.S.A.}

\date{\today}%
\begin{document}

\begin{abstract}
We study ancestral structures for the two-type Moran model with  mutation and frequency-dependent selection under the nonlinear dominance or fittest-type-wins scheme. Under appropriate conditions, both  lead, in distribution, to the same type-frequency process. Reasoning through the mutations  on the ancestral selection graph (ASG), we develop the corresponding killed and pruned lookdown ASG and  use them to determine the present and ancestral type distributions. To this end, we establish factorial moment dualities to the Moran model and a relative. We extend the results to the diffusion limit and present applications for finite population size as well as   moderate and  weak selection. 
\end{abstract}
\maketitle
\bigskip { \footnotesize
\noindent{\slshape\bfseries MSC 2020.} Primary:\, 60K35, 92D15  \ Secondary:\, 60J25, 60J27 

\medskip 
\noindent{\slshape\bfseries Keywords.} duality, frequency-dependent selection, Moran model, Wright--Fisher diffusion,  ancestral selection graph,  descendant process, ancestral type distribution}
\setcounter{tocdepth}{1}
%\tableofcontents

\medskip

\noindent{\slshape\bfseries Declarations of interest:}  none

\section{Introduction}
The \emph{ancestral selection graph (ASG)} is a branching-coalescing random graph and  a classical tool to describe ancestries and genealogies in population-genetic models under selection. It goes back to Neuhauser and Krone \cite{NeKro97} and  has mostly been used in the context of haploid populations (where each individual carries one copy of the genetic information), or diploid populations (where every individual is composed of two copies, the \emph{gametes}) under so-called genic selection. The latter means that each copy of the genetic information contributes to the reproduction rate of the individual in an independent, additive fashion, which implies that one may, to a good approximation, work with the population of haploid gametes and ignore their pairing into diploid individuals.

\smallskip

If, however, the contribution of the gametes to the reproduction rate of a diploid individual is not additive, or if the contribution depends on the genetic composition of the entire population (that is, if there is \emph{frequency-dependent selection}), the standard version of the ASG does not suffice. The extension to  frequency-dependent selection has been sketched by Neuhauser \cite{Ne99}; however, it is difficult to handle, and obtaining explicit results is conceptually and technically difficult. Recent examples of such endeavours are \cite{casanova2018duality,cordero2019general,MachSturmSwart20,BCH2018interaction}. \citep{casanova2018duality} describes the ancestry of a sample from the present population in a (discrete-time) $\Lambda$-Wright--Fisher model and its diffusion limit, under a specific kind of frequency-dependent (so-called \emph{fittest-type-wins}) selection, and \citep{cordero2019general} analyses frequency-dependent selection described via a general polynomial drift vanishing  at the boundary in the corresponding stochastic differential equation; but both works exclude mutation. \cite{MachSturmSwart20} and \cite{BCH2018interaction}  consider models with mutation and determine the ancestry of a single individual in a  law of large numbers regime, where the type-frequency process satisfies an ordinary differential equation, and the ASG reduces to a tree due to the absence of coalescence events.  \cite{MachSturmSwart20}  presents a general formalism to treat frequency-dependent selection and mutation, whereas  \cite{BCH2018interaction}  works out in detail the  \emph{recessive} case (where  the contributions of two gametes to the reproduction rate of a diploid individual is subadditive) including mutation, and computes many quantities of interest explicitly. 

\smallskip

In the finite-population case and the diffusion limit (and in contrast to the law of large numbers regime), the graph contains coalescence events and no longer reduces to a tree. In the recessive setting including mutation, this seems to render the approach of \cite{BCH2018interaction} infeasible. In this paper, we tackle \emph{nonlinear dominance} and \emph{fittest-type-wins} selection schemes, again with mutation. We take \emph{nonlinear dominance} to mean a  kind of frequency-dependent selection where an individual of the beneficial type reproduces  when a   random number  of uniformly chosen partner individuals are of the deleterious type; it is a generalisation of the dominant case in a diploid population, where the contributions of the two gametes to the reproduction rate of an individual are superadditive. In contrast, \emph{fittest-type-wins} selection means that a group of individuals jointly try to produce an offspring; reproduction  occurs if at least one of the individuals in the group is fit. Both schemes will turn out to be equivalent in distribution in a certain parameter regime. The models can be analysed via the ASG, which is constructed on the basis of the graphical representation of the Moran model. Mutation events on the ASG contain information that then allow us to reduce the graph to the parts that are informative for the type distribution of an individual at present --- the resulting object goes under the name of \emph{killed ASG}.

\smallskip

The type of an individual's ancestor can differ from the individual's type because of mutations. It is a challenging problem to find a tractable representation of its distribution. In the Moran model, all individuals at present share a common ancestor in the sufficiently distant past. 
The type distribution of this common ancestor was first described and analysed in the diffusion limit  in the case of genic selection by~\citet{fearnhead2002common}, and later for general frequency-dependent selection by~\citet{Taylor2007}. In \cite{Kluth2013},  similar results were derived for the (finite-population) Moran model with genic selection. All these investigations mainly relied on analytical methods. A probabilistic approach to  the common ancestor type distribution under genic selection  was introduced by~\citet{LKBW15} based on the ASG. To recover \citeauthor{fearnhead2002common}'s results, they construct the \emph{pruned lookdown ASG} in the diffusion limit. Its name appeals to the underlying idea of pruning certain lines upon mutations, and of ordering them in a way inspired by the lookdown constructions of~\citet{DK99}. The approach was later extended to the $\Lambda$-Wright--Fisher process with genic selection~\citep{BLW16}, to the Wright--Fisher diffusion with selection in random environments~\citep{fern2019moran}, and to the mutation--selection differential equation with a specific form of pairwise interaction~\citep{BCH18jmath,BCH2018interaction}. But in the finite Moran model the only known results  assume genic selection~\citep{Cordero17DL,Hummel2019}. Here, we consider the ancestral type distribution for the Moran model with nonlinear dominance and fittest-type-wins selection schemes, and we augment the analysis by a forward-in-time approach based on a \emph{descendant process}~\citep{Kluth2013}. Moreover, working with a finite population size all the way through allows to take various scaling limits at a late stage, which we exploit in a \emph{moderate selection} setting, and in a diffusion limit under \emph{weak selection}.

\smallskip

The paper is organised as follows. Our main results along with the intuition for our constructions are presented in Section~\ref{sec:mainresult}. Technical details and proofs are deferred to subsequent sections. Section~\ref{sec:models} contains details of the two selection models, their representations as interacting particle systems, and the connection between them. The definition of the ASG and the rigorous definition of the ancestral and common ancestor type distributions is in~Section~\ref{sec:ASG}. Details of the construction of the killed ASG and proofs of associated results may be found in Section~\ref{sec:kASG}. Section~\ref{sec:siegmund} provides details of a new look on the Siegmund dual for our Moran model. The results associated with the backward and forward perspectives on the ancestral type distribution  are in Sections~\ref{sec:ancestraltype} and ~\ref{sec:absorbingmarkov}, respectively. Section~\ref{sec:diffusionlimit} contains the proofs of the results in the diffusion limit, and Section~\ref{sec:applications} is devoted to the proof of the fixation probability of a single beneficial mutant in the moderate-selection regime.

%%%%%%%%%%%%%%%%%%%%%%%%%%%%%%%%%%%%%%%%%%%%%%%%%%%%%%%%%%%%%%%%%%%%%%%%%%%%%%%%%%%%%%%%%%%%%%%%%%%%%%%%%%%%%%%%%%%%%%%%%%%%%%%%%%%%%%%%%%%%%%%%%%%%%%%%%%%%%
%%%%%%%%%%%%%%%%%%%%%%%%%%%%%%%%%%%%%%%%%%%%%%%%%%%%%%%%%%%%%%%%Section MAIN RESULTS%%%%%%%%%%%%%%%%%%%%%%%%%%%%%%%%%%%%%%%%%%%%%%%%%%%%%%%%%%%%%%%%%%%%%%%%%%%%%%%%%%%%%
%%%%%%%%%%%%%%%%%%%%%%%%%%%%%%%%%%%%%%%%%%%%%%%%%%%%%%%%%%%%%%%%%%%%%%%%%%%%%%%%%%%%%%%%%%%u%%%%%%%%%%%%%%%%%%%%%%%%%%%%%%%%%%%%%%%%%%%%%%%%%%%%%%%%%%%%%%%

\section{Main results}\label{sec:mainresult}
A Moran model (MoMo) is composed of $N\in \N$ haploid individuals, each characterised by a type from~$\{0,1\}$. We refer to type $1$ as \emph{unfit} or \emph{deleterious}, and to type $0$ as \emph{fit} or \emph{beneficial}. The population is panmictic --- that is, there is no spatial structure --- and evolves in continuous time via mutation and neutral and selective reproduction. More specifically, each individual mutates at rate~$u$, with the resulting type  being $0$ and $1$, respectively, with probability $\nu_0\in(0,1)$ and $\nu_1=1-\nu_0$. Independently at rate $1$, each individual, independently of its type, produces a single offspring that inherits the parental type and replaces a uniformly chosen individual so that the population size remains constant; this is the neutral reproduction part. Selection can be incorporated into a MoMo in various ways. We consider two versions, one with \emph{nonlinear dominance} (DOM) and one with \emph{fittest-type-wins} (FTW) selection. Both mechanisms lead to a \emph{selective advantage} of type~$0$ over its counterparts, and both are parametrised by a non-negative real-valued sequence as follows.

\smallskip

Under the DOM scheme, each type-$0$ individual independently at rate~$\widehat s_m\in \R_+$ checks the type of $m-1$   individuals chosen uniformly with replacement, $m\in \N$. If all the checked individuals have type~$1$,  the checking individual produces a single fit offspring that replaces a uniformly chosen individual.
Under the FTW scheme, each individual  is independently \emph{affected} at rate $s_m\in \R_+$ by a selective event of order~$m$, $m\in \N$. In such an event, $m$ individuals, chosen uniformly without replacement, form a group. The first fit group member produces a fit offspring that replaces the affected individual. If none of the group members is fit, nothing happens. (Later, we will first consider DOM, but, for reasons that will become apparent, eventually work with the FTW model. This is why we use the hats for notation related to DOM.)

\smallskip

Let $ \widehat Y^{t_0}_t$ and ${Y}^{t_0}_t$ be the counts of type-$1$ individuals at time $t_0+t$ in the MoMo with DOM and FTW, respectively. Here, $t\in \R_+$ is the  time increment relative to some reference time $t_0\in \R$; the reason for this notation will become clear soon. Then $\widehat Y^{t_0} = (\widehat Y^{t_0}_t)_{t \geq 0}$ and $Y^{t_0} = (Y^{t_0}_t)_{t \geq 0}$ are birth-death processes on $[N]_0\coloneqq\{0,1,2,\ldots,N\}$; whenever $t_0=0$, we just write $\widehat Y_t$ instead of $\widehat Y_t^0$, and likewise for ${Y}_t$ and ${Y}_t^0$. The respective generators $\cA_{Y}$ and $\cA_{\widehat Y}$ act on functions $f: [N]_0 \to \R$ and are given by $\mathcal{A}_{\widehat Y} =\mathcal{A}_{ Y}^{u}+\mathcal{A}_{Y}^{\rm{n}}+\sum_{m>0} \mathcal{A}_{\widehat Y}^{\widehat s_m}$ and $\mathcal{A}_{Y} =\mathcal{A}_{ Y}^{u}+\mathcal{A}_{Y}^{\rm{n}}+\sum_{m>0} \mathcal{A}_{Y}^{s_m}$,  where 
\begin{equation}\label{eq:generatorhatY}
	\begin{split}
		\cA_{Y}^{\rm{n}}  f(k)& \defeq k \frac{N-k}{N} [ f(k+1)-f(k)] + k \frac{N-k}{N}  [ f(k-1)-f(k)], \\
		\mathcal{A}_{Y}^{u}   f(k) & \defeq (N-k) u \nu^{}_1 [ f(k+1)-f(k)] + k u \nu^{}_0 [ f(k-1)-f(k)],\\
		\mathcal{A}_{\widehat Y}^{\widehat s_m}   f(k) &  \defeq  \widehat s_m  (N-k) \Big (\frac{k}{N} \Big )^m  [f(k-1)- f(k)], \quad \mathcal{A}_Y^{s_m}   f(k)  \defeq s^{}_m  k \Big( 1- \Big(\frac{k}{N}\Big)^m\Big)  [f(k-1)- f(k)], 	\end{split}
\end{equation}
under the convention that $f(-1) \defeq 0 \eqdef f(N+1)$.
We assume throughout that $0 < \sum_{m>0}  \widehat s_m m < \infty$ and $0 < \sum_{m>0}  s_m m < \infty$ to avoid the trivial case of neutral evolution (that is, without selection), and the degenerate case of an infinite rate of offspring production. 

\begin{figure}[b]
	
	\scalebox{0.7}{\hspace{-.6cm}\begin{tikzpicture}
			
			%Horizontal ASG
			\draw[semithick ] (14.5,4) -- (0,4);
			\draw[semithick ] (12.9,1) -- (1,1);
			\draw[semithick ] (8.5,2) -- (3.6,2);
			\draw[semithick ] (10.4,1) -- (1,1);
			\draw[semithick ] (8.5,3) -- (0,3);
			\draw[semithick ] (2.9,0) -- (0,0);
			
			%Vertical ASG
			
			%Frame
			\draw[dashed] (0,-0.5) --(0,4.5);
			%\draw[dashed] (14.5,-0.5) --(14.5,4.5);    
			\node [below] at (0,-0.5) {\scalebox{1.3}{$t_0$}};
			%\node [right] at (14.5,-0.5) {$t$};
			\draw[-{angle 60[scale=5]},line width=1.3] (5.25,-0.6) -- (9.25,-0.6) node[text=black, pos=.5, yshift=-6pt]{\scalebox{1.3}{$t$}};
			%horizontal lines
			\draw[opacity=1,semithick] (0,0) -- (14.5,0);
			\draw[opacity=1,semithick] (0,1) -- (14.5,1);
			\draw[opacity=1,semithick] (14.5,2) -- (0,2);
			\draw[opacity=1,semithick] (14.5,3) -- (0,3);
			\draw[opacity=1,semithick] (0,4) -- (14.5,4);
			%neutral arrows
			\draw[-{triangle 45[scale=5]},semithick,opacity=1] (10,4) -- (10,3);	\draw[-{triangle 45[scale=5]},semithick,opacity=1] (11.5,1) -- (11.5,0);
			\draw[-{triangle 45[scale=5]},semithick,opacity=1] (3.6,1) -- (3.6,2);
			%\draw[-{triangle 45[scale=5]},semithick,opacity=1] (1,3) -- (1,1);		\draw[-{triangle 45[scale=5]},semithick,opacity=1] (1,3) .. controls (.8,2.5) and (0.8,1.5) .. (1,1);

			%selective arrows
			\draw[-{open triangle 45[scale=5]},semithick,opacity=1] (12.9,1) .. controls (12.7,2) and (12.7,3) .. (12.9,4);
			\draw[-{open triangle 45[scale=5]},semithick,opacity=1] (7,2) .. controls (6.8,1.3) and (6.8,0.7) .. (7,0);
			\draw[-{open triangle 45[scale=5]},semithick,opacity=1] (6.2,4) .. controls (6,2.5) and (6,2.4) .. (6.2,1);
			\draw[-{open triangle 45[scale=5]},semithick,opacity=1] (2.9,0) .. controls (2.7,1) and (2.7,3) .. (2.9,4);
			\draw[] (12.9,1) circle (.6mm)  [fill=black!100];
			\draw[] (7,2) circle (.6mm)  [fill=black!100];
			\draw[] (6.2,4) circle (.6mm)  [fill=black!100];
			\draw[] (2.9,0) circle (.6mm)  [fill=black!100];

			%interactive arrows 
			\draw[] (8.5,3) circle (.6mm)  [fill=black!100];
			\draw[semithick] (8.5,2) .. controls (8.5,2.5) and (8.65,2.5) .. (8.65,3);
			\draw[semithick] (8.65,3) .. controls (8.65,3.5) and (8.5,3.5) .. (8.5,4);
			\filldraw[white, draw=black] (8.5,2.1) -- (8.6,2) -- (8.5,1.9) -- (8.4,2) -- (8.5,2.1);
			\draw[-{triangle 45[fill=white]},opacity=1,semithick] (8.5,3) -- (8.5,4);
			
			\draw[opacity=1] (12.4,2) circle (.6mm)  [fill=black!100];
			\draw[semithick, opacity=1] (12.4,3.9) .. controls (12.55,3.5) and (12.55,3.5) .. (12.4,3);
			\filldraw[semithick,opacity=1,white, draw=black] (12.4,4.1) -- (12.5,4) -- (12.4,3.9) -- (12.3,4) -- (12.4,4.1);
			\draw[semithick, opacity=1] (12.4,0.1) .. controls (12.55,0.5) and (12.55,0.5) .. (12.55,1);
			\draw[opacity=1,semithick] (12.55,1) -- (12.55,2);
			\draw[semithick, opacity=1] (12.55,2) .. controls (12.55,2.5) and (12.35,2.5) .. (12.4,3);
			\filldraw[semithick,opacity=1,white, draw=black] (12.4,0.1) -- (12.5,0) -- (12.4,-0.1) -- (12.3,0) -- (12.4,0.1);
			\draw[-{triangle 45[fill=white]},semithick,opacity=1] (12.4,2) -- (12.4,3);
			%Mutations deleterious
			\node[opacity=1] at (11,3) {\scalebox{1.5}{$\times$}} ;
			\node[opacity=1] at (13.3,1) {\scalebox{1.5}{$\times$}};
			\node[opacity=1] at (12,0) {\scalebox{1.5}{$\times$}};
			\node[opacity=1] at (8,4) {\scalebox{1.5}{$\times$}};
			\node[opacity=1] at (2,2) {\scalebox{1.5}{$\times$}};
			
			\node[opacity=1,left] at (-.2,0) {\scalebox{1.5}{$N$}};   
			\node[opacity=1,left] at (-.2,4) {\scalebox{1.5}{$1$}};  
			\node[opacity=1,left] at (-.2,1.3) {\scalebox{1.2}{$\bullet$}}; 
			\node[opacity=1,left] at (-.2,2) {\scalebox{1.2}{$\bullet$}};   
			\node[opacity=1,left] at (-.2,2.7) {\scalebox{1.2}{$\bullet$}};
			
			%beneficial mutations
			\draw (2.2,0)[opacity=1] circle (1.5mm)  [fill=white!100];    
			%\draw (3.2,4)[opacity=1] circle (1.5mm)  [fill=white!100];
			\draw (7.5,2)[opacity=1] circle (1.5mm)  [fill=white!100];
			
	\end{tikzpicture} }
	\caption{An untyped realisation of the  Moran interacting particle system with nonlinear dominance; $t$ is the time increment.}
	\label{fig:untypedmomo}
\end{figure}
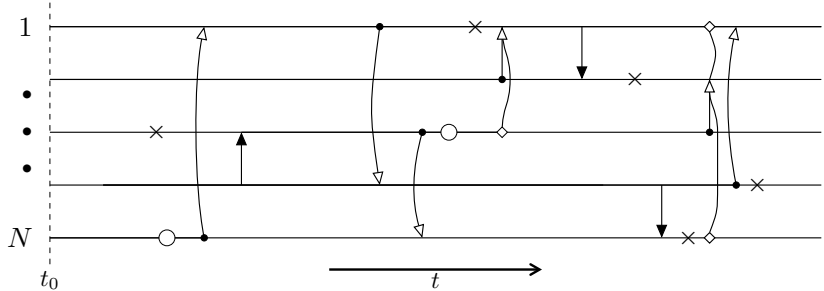

\smallskip

The \emph{graphical representation} underlying the MoMo provides intuition for the associated backward processes and is shown in   Fig.~\ref{fig:untypedmomo}. Each individual corresponds to a horizontal line segment, with the forward direction of time being left to right; the events described above are represented by graphical elements juxtaposed to this picture. We first describe the \emph{untyped} version and include the types later. There are $N$ lines, with labels in $[N]\coloneqq\{1,2,\ldots,N\}$ that we will sometimes refer to as \emph{sites} --- but this does not imply a spatial structure. Mutation events are depicted by circles and crosses on the lines. A circle (cross) indicates a mutation
to type 0 (type 1), which means that the type on the line is 0 (is 1) after the mutation. This occurs
at rate~$u \nu_0$ (at rate~$u \nu_1$) on every line, by way of independent Poisson point processes. (Potential) reproduction events are depicted by arrows between the lines, with the (potential) parent at the tail and the offspring at the tip. If a parent places offspring via the arrow, the offspring inherits the parent's type and replaces
the individual at the tip. We then say that the (parent) individual \emph{uses} the arrow. We decompose
reproduction events into neutral ones and selective ones of various orders, corresponding to the $\widehat s_m$ and $s_m$, respectively, for $m \in \mathbb{N}$. \emph{Neutral arrows} (with solid arrowheads) appear at rate $1/N$ independently per ordered pair of  lines; if head and tail are identical, the arrow points to itself, hence it is irrelevant and may be ignored. Neutral arrows are always used.

\smallskip

\emph{Selective arrows} (hollow arrowheads, black bullet at tail) of order $m$ appear at rate $\widehat s_m/N$  (at rate $s_m/N$) independently per ordered pair of  lines; again, an arrow pointing to itself is void. Every selective event of order $m$ consists of a selective  arrow  accompanied by $0$ up to $m-1$ \emph{checking arrows} (diamond at tail), whose tips share the tip of the selective arrow, but whose tails are connected to those lines that receive at least one mark when marking $m-1$ lines  chosen uniformly, independently, and with replacement. More precisely, for an $m$-tuple of lines  $J=(J_1, \ldots, J_m)\in [N]^m$, an event $ E_{J,i}$ defines the tip $i$ of the selective arrow, the tail $J_1$ of the selective arrow, and the set $\{J_2, \ldots, J_m\}$ of tails of checking arrows and occurs at rate $s_m/N^{m}$  per  $m+1$-tuple of lines; note that if there are duplicate marks, then $\lvert \{J_2, \ldots, J_m\}\rvert<m-1$. Line~$i$ is called the \emph{continuing line},  line $J_1$ is the  \emph{incoming} line, and the lines in $\{J_2, \ldots, J_m\}$ are the \emph{checking lines}, see Fig.~\ref{fig:asgelements}.  

\smallskip

Under the DOM scheme, fit individuals use selective arrows if there is an unfit
individual at the tail of each of the associated checking arrows. In particular, a selective arrow  will never be used if the incoming line is also a checking line.
In the FTW scheme, selective and checking arrows are equivalent; we therefore refer to them just as selective arrows in this context (and replace the diamonds at their tails by bullets). Likewise, the lines at their tails will all be referred to as incoming. In a selective event of order~$m$, that is $E_{J,i}$ for some $i \in [N]$ and $J \in [N]^m$, the first selective arrow with a fit individual at its tail is used (the order is prescribed by the indices of~$J$). None of the selective arrows is used if there is no such fit individual. In particular, the descendant on line~$i$ is fit if and only if at least one of the potential parents is fit, where the set of \emph{potential parents} is $\{J_1, \ldots, J_m\} \cup \{i\}$; hence the name \emph{fittest-type-wins}.

\smallskip

Given a realisation of the untyped system and some initial time~$t_0$,  we can  turn it into a \emph{typed} one by assigning a type to each line at time $t_0$ and propagating the types forward in time according to the rules described above (see Figs.~\ref{fig:asgelements} and~\ref{fig:typedmomo}). 
This way, we can read off the types for any time $t_0+t$ for $t>0$. 

\smallskip

Under certain conditions on the selection parameters, the two selection schemes are equivalent in distribution as is shown in the following result, which we prove in Section~\ref{sec:models}.

\begin{figure}[t]
	\begin{tikzpicture}
		\draw[opacity=0,line width=1.3mm] (0,-.2) -- (1,2);
		\draw[red,line width=1.3mm] (0,0) -- (1,0);
		\draw[red,line width=1.3mm] (0,1.4) -- (1.6,1.4);
		\draw[red,line width=1.3mm] (1,0) -- (1,1.4);
		\draw[semithick] (0,0) -- (1,0);
		\draw[semithick] (0,1.4) -- (1.6,1.4);
		
		%\draw[-{open triangle 45[scale=5]},line width=.8mm,red] (1,0) -- (1,.7);
		\draw[-{open triangle 45[scale=5]},semithick,opacity=1] (1,0) -- (1,1.4);
		\node[left] at (0,0) {\text{i}};
		\node[left] at (0,1.4) {\text{co}};
		\node[right] at (1.6,1.4) {\text{d}};
		\draw[opacity=1] (1,0) circle (.8mm)  [fill=black!100];
		
	\end{tikzpicture}\hspace{1cm}\begin{tikzpicture}
		\draw[opacity=0,line width=1.3mm] (0,-.2) -- (1,2);
		\draw[red,line width=1.3mm] (0,1.4) -- (1.6,1.4);
		\draw[line width=1.3mm,red] (0,.7) -- (1,.7);
		\draw[line width=1.3mm,red] (0,0) -- (1,0);
		\draw[line width=1.3mm,red] (1,0.7) -- (1,1.3);
		\draw[semithick] (0,1.4) -- (1.6,1.4);
		\draw[semithick] (0,.7) -- (1,.7);
		\draw[semithick] (0,0) -- (1,0);
		%\draw[semithick] (1,0) -- (1,1.3);
		
		\draw[line width=1mm,red] (1,0) .. controls (1,.2) and (1.15,.5) .. (1.15,.7);
		\draw[line width=1mm,red] (1.15,.7) .. controls (1.15,.9) and (1,1.05) .. (1,1.35);
		\draw[semithick] (1,0) .. controls (1,.2) and (1.15,.5) .. (1.15,.7);
		\draw[semithick] (1.15,.7) .. controls (1.15,.9) and (1,1.05) .. (1,1.17);

		%\fill[white, draw=black] (.9,1.3) rectangle (1.1,1.5); 
		\draw[-{open triangle 45[fill=white]},semithick,opacity=1] (1,.7) -- (1,1.4)  node [right] {};
		
		\filldraw[semithick,opacity=1,white, draw=black] (1,.8) -- (1.1,.7) -- (1,.6) -- (.9,.7) -- (1,.8);
		\draw[opacity=1] (1,0) circle (.8mm)  [fill=black!100];
		\node[left] at (0,0) {\text{i}};
		\node[left] at (0,.7) {\text{ch}};
		\node[left] at (0,1.4) {\text{co}};
		\node[right] at (1.6,1.4) {\text{d}};
	\end{tikzpicture}
	\hspace{1cm}
	\begin{tikzpicture}
		\draw[opacity=0,line width=1.3mm] (0,-.2) -- (1,2);
		\draw[semithick] (0,1.4) -- (1.6,1.4);
		\draw[semithick] (0,.7) -- (1,.7);
		\draw[semithick] (0,0) -- (1,0);
		%\draw[semithick] (1,0) -- (1,1.3);

		\draw[semithick] (1,0) .. controls (1,.2) and (1.15,.5) .. (1.15,.7);
		\draw[semithick] (1.15,.7) .. controls (1.15,.9) and (1,1.05) .. (1,1.17);
		
		%\fill[white, draw=black] (.9,1.3) rectangle (1.1,1.5);
		\draw[-{open triangle 45[fill=white]},semithick,opacity=1] (1,.7) -- (1,1.4)  node [right] {};
		
		\filldraw[semithick,opacity=1,white, draw=black] (1,.8) -- (1.1,.7) -- (1,.6) -- (.9,.7) -- (1,.8);
		\draw[opacity=1] (1,0) circle (.8mm)  [fill=black!100];
		
		\fill[lightbrown, draw=black] (1.45,1.25) rectangle (1.75,1.55);
		\fill[lightbrown, draw=black] (.15,1.25) rectangle (-.15,1.55);
		\fill[lightbrown, draw=black] (.15,-.15) rectangle (-.15,.15);
		\fill[lightbrown, draw=black] (-.15,.85) rectangle (.15,.55);
		
		\fill[lightbrown, draw=black] (-.3,1.25) rectangle (-.6,1.55);
		\fill[darkgreen, draw=black] (-.3,-.15) rectangle (-.6,.15);
		\fill[darkgreen, draw=black] (-.3,.85) rectangle (-.6,.55);
		
		\fill[lightbrown, draw=black] (-.75,1.25) rectangle (-1.05,1.55);
		\fill[lightbrown, draw=black] (-.75,-.15) rectangle (-1.05,.15);
		\fill[darkgreen, draw=black] (-.75,.85) rectangle (-1.05,.55);
		
	\end{tikzpicture}
	\caption{A selective event of order~$1$ (left) and of order~$2$ (center) in the DOM model. The continuing line is indicated by co, the checking line by ch, the incoming line by i, and the descendant  line by d. On the right, all the possible configurations in a selective event of order $2$ that lead to an unfit descendant are depicted; type~0 in dark green, type~1 in light brown.}
	\label{fig:asgelements}
\end{figure}
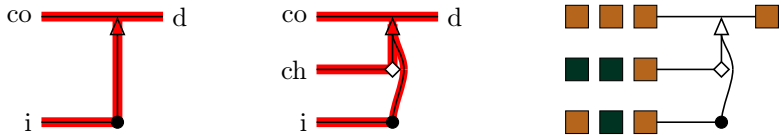

\begin{lem}\label{lem:DF}
	Let $(\widehat{s}_m)_{m>0}$ and $(s_m)_{m>0}$ be two sequences in $\R_+$ satisfying $0 < \sum_{m>0}  \widehat{s}_m m < \infty$, $0 < \sum_{m>0}  s_m m < \infty$, and $(\widehat{s}_m)_{m>0}$ is non-increasing.
	Let $\widehat Y$ and $Y$ be the MoMo processes with DOM and FTW scheme and with  selection parameters  $(\widehat{s}_m)_{m>0}$ and $(s_m)_{m>0}$, respectively. If $\widehat Y_0=Y_0$ and $s_m = \widehat s_m - \widehat s_{m+1}$ for $m>0$  (or, equivalently, $\widehat s_m = \sum_{n \geqslant m} s_n$), then $\widehat{Y}$ and $Y$ are identical in distribution.
\end{lem}
\begin{figure}[t]
	
	\scalebox{0.7}{\hspace{-.6cm}\begin{tikzpicture}
			
			%1 line
			\draw[darkgreen,opacity=1,line width=0.9mm] (0,0)--(11.5,0);
			\draw[darkgreen,opacity=1,line width=0.9mm] (11.5,0)--(12,0);
			\draw[lightbrown,opacity=1,line width=0.9mm] (12,0)--(14.5,0);
			
			%2 line
			\draw[darkgreen,opacity=1,line width=0.9mm] (0,1)--(1,1);
			\draw[lightbrown,opacity=1,line width=0.9mm] (1,1)--(6.2,1);
			\draw[darkgreen,opacity=1,line width=0.9mm] (6.2,1)--(13.3,1);
			\draw[lightbrown,opacity=1,line width=0.9mm] (13.3,1)--(14.5,1);
			
			%3 line
			\draw[darkgreen,opacity=1,line width=0.9mm] (0,2)--(2,2);
			\draw[lightbrown,opacity=1,line width=0.9mm] (2,2)--(9.4,2);
			\draw[darkgreen,opacity=1,line width=0.9mm] (7.5,2)--(14.5,2);
			
			%4 line
			\draw[lightbrown,opacity=1,line width=0.9mm] (0,3)--(12.4,3);
			\draw[darkgreen,opacity=1,line width=0.9mm] (12.4,3)--(14.5,3);
			
			%5 line
			\draw[lightbrown,opacity=1,line width=0.9mm] (0,4)--(2.9,4);
			\draw[darkgreen,opacity=1,line width=0.9mm] (2.9,4)--(8,4);
			\draw[lightbrown,opacity=1,line width=0.9mm] (8,4)--(12.9,4);
			\draw[darkgreen,opacity=1,line width=0.9mm] (12.9,4)--(14.5,4);
			
			%starting distribution
			\fill [darkgreen] (-.15,-.15) rectangle (0.15,0.15);
			\fill [darkgreen] (-.15,0.85) rectangle (0.15,1.15);
			\fill [darkgreen] (-.15,1.85) rectangle (0.15,2.15);
			\fill [lightbrown] (-.15,2.85) rectangle (0.15,3.15);
			\fill [lightbrown] (-.15,3.85) rectangle (0.15,4.15);
			
			%		%ending distribution
			%		\fill [lightbrown] (14.35,-.15) rectangle (14.65,0.15);
			%		\fill [lightbrown] (14.35,0.85) rectangle (14.65,1.15);
			%		\fill [darkgreen] (14.35,1.85) rectangle (14.65,2.15);
			%		\fill [darkgreen] (14.35,2.85) rectangle (14.65,3.15);
			%		\fill [darkgreen] (14.35,3.85) rectangle (14.65,4.15);
			
			%Horizontal ASG
			\draw[semithick ] (14.5,4) -- (0,4);
			\draw[semithick ] (12.9,1) -- (1,1);
			\draw[semithick ] (8.5,2) -- (3.6,2);
			\draw[semithick ] (10.4,1) -- (1,1);
			\draw[semithick ] (8.5,3) -- (0,3);
			\draw[semithick ] (2.9,0) -- (0,0);
			
			%Vertical ASG
			
			%Frame
			\draw[dashed] (0,-0.5) --(0,4.5);
			%\draw[dashed] (14.5,-0.5) --(14.5,4.5);    
			\node [below] at (0,-0.5) {\scalebox{1.3}{$t_0$}};
			%\node [right] at (14.5,-0.5) {\scalebox{1.3}{$t$}};
			\draw[-{angle 60[scale=5]},line width=1.3] (5.25,-0.6) -- (9.25,-0.6) node[text=black, pos=.5, yshift=-6pt]{\scalebox{1.3}{$t$}};
			%horizontal lines
			\draw[opacity=1,semithick] (0,0) -- (14.5,0);
			\draw[opacity=1,semithick] (0,1) -- (14.5,1);
			\draw[opacity=1,semithick] (14.5,2) -- (0,2);
			\draw[opacity=1,semithick] (14.5,3) -- (0,3);
			\draw[opacity=1,semithick] (0,4) -- (14.5,4);
			%neutral arrows
			\draw[-{triangle 45[scale=5]},semithick,opacity=1] (10,4) -- (10,3);
			\draw[-{triangle 45[scale=5]},semithick,opacity=1] (11.5,1) -- (11.5,0);
			\draw[-{triangle 45[scale=5]},semithick,opacity=1] (3.6,1) -- (3.6,2);
			\draw[-{triangle 45[scale=5]},semithick,opacity=1] (1,3) .. controls (.8,2.5) and (0.8,1.5) .. (1,1);
			%selective arrows
			\draw[-{open triangle 45[scale=5]},semithick,opacity=1] (12.9,1) .. controls (12.7,2) and (12.7,3) .. (12.9,4);
			\draw[-{open triangle 45[scale=5]},semithick,opacity=1] (7,2) .. controls (6.8,1.3) and (6.8,0.7) .. (7,0);
			\draw[-{open triangle 45[scale=5]},semithick,opacity=1] (6.2,4) .. controls (6,2.5) and (6,2.4) .. (6.2,1);
			\draw[-{open triangle 45[scale=5]},semithick,opacity=1] (2.9,0) .. controls (2.7,1) and (2.7,3) .. (2.9,4);
			\draw[] (12.9,1) circle (.6mm)  [fill=black!100];
			\draw[] (7,2) circle (.6mm)  [fill=black!100];
			\draw[] (6.2,4) circle (.6mm)  [fill=black!100];
			\draw[] (2.9,0) circle (.6mm)  [fill=black!100];
			
			%interactive arrows 
			\draw[] (8.5,3) circle (.6mm)  [fill=black!100];
			\draw[-{open triangle 45[scale=5]},opacity=1,semithick] (8.5,3) -- (8.5,4);
			\draw[semithick] (8.5,2) .. controls (8.5,2.5) and (8.65,2.5) .. (8.65,3);
			\draw[semithick] (8.65,3) .. controls (8.65,3.5) and (8.5,3.5) .. (8.5,4);
			\filldraw[white, draw=black] (8.5,2.1) -- (8.6,2) -- (8.5,1.9) -- (8.4,2) -- (8.5,2.1);
			
			\draw[-{open triangle 45[scale=5]},semithick,opacity=1] (12.4,2) -- (12.4,3);
			\draw[opacity=1] (12.4,2) circle (.6mm)  [fill=black!100];
			\draw[semithick, opacity=1] (12.4,3.9) .. controls (12.55,3.5) and (12.55,3.5) .. (12.4,3);
			\filldraw[semithick,opacity=1,white, draw=black] (12.4,4.1) -- (12.5,4) -- (12.4,3.9) -- (12.3,4) -- (12.4,4.1);
			\draw[semithick, opacity=1] (12.4,0.1) .. controls (12.55,0.5) and (12.55,0.5) .. (12.55,1);
			\draw[opacity=1,semithick] (12.55,1) -- (12.55,2);
			\draw[semithick, opacity=1] (12.55,2) .. controls (12.55,2.5) and (12.35,2.5) .. (12.4,3);
			\filldraw[semithick,opacity=1,white, draw=black] (12.4,0.1) -- (12.5,0) -- (12.4,-0.1) -- (12.3,0) -- (12.4,0.1);
			
			%Mutations deleterious
			\node[opacity=1] at (11,3) {\scalebox{1.5}{$\times$}} ;
			\node[opacity=1] at (13.3,1) {\scalebox{1.5}{$\times$}};
			\node[opacity=1] at (12,0) {\scalebox{1.5}{$\times$}};
			\node[opacity=1] at (8,4) {\scalebox{1.5}{$\times$}};
			\node[opacity=1] at (2,2) {\scalebox{1.5}{$\times$}};
			
			\node[opacity=1,left] at (-.2,0) {\scalebox{1.5}{$N$}};   
			\node[opacity=1,left] at (-.2,4) {\scalebox{1.5}{$1$}};  
			\node[opacity=1,left] at (-.2,1.3) {\scalebox{1.2}{$\bullet$}}; 
			\node[opacity=1,left] at (-.2,2) {\scalebox{1.2}{$\bullet$}};   
			\node[opacity=1,left] at (-.2,2.7) {\scalebox{1.2}{$\bullet$}};
			
			%beneficial mutations
			\draw (2.2,0)[opacity=1] circle (1.5mm)  [fill=white!100];    
			%\draw (3.2,4)[opacity=1] circle (1.5mm)  [fill=white!100];
			\draw (7.5,2)[opacity=1] circle (1.5mm)  [fill=white!100];
			
	\end{tikzpicture} }
	
	\caption{The realisation of Fig.~\ref{fig:untypedmomo}, but now with types (type~0 in dark green, type~1 in light brown).  }
	\label{fig:typedmomo}
	
\end{figure}
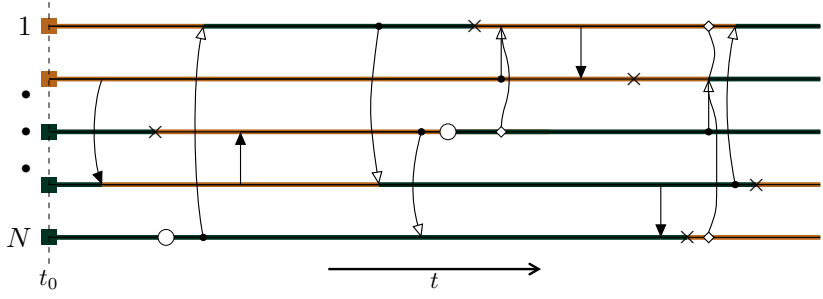

\subsection{The ancestral selection graph} \label{sec:mainresukt:subsec:asg}
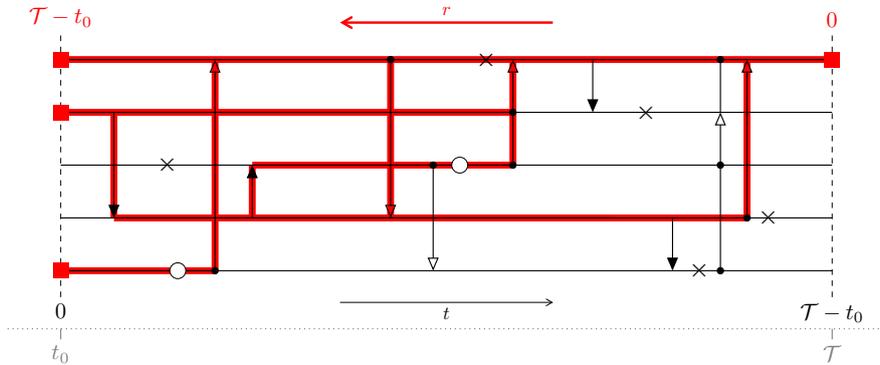
\begin{figure}[b]
	\scalebox{0.7}{\begin{tikzpicture}
			\draw[red,opacity=1,line width=1.3mm] (0,4)--(14.5,4);
			\draw[red,opacity=1,line width=1.3mm] (12.9,1)--(1,1);		
			\draw[red,opacity=1,line width=1.3mm] (3.6,2)--(8.5,2);
			\draw[red,opacity=1,line width=1.3mm] (0,3)--(8.5,3);
			\draw[red,opacity=1,line width=1.3mm] (0,0)--(2.9,0);
			\draw[red,opacity=1,line width=1.3mm] (12.9,1)--(12.9,4);
			\draw[red,opacity=1,line width=1.3mm] (8.5,2)--(8.5,4);
			\draw[red,opacity=1,line width=1.3mm] (6.2,1)--(6.2,4);
			\draw[red,opacity=1,line width=1.3mm] (2.9,0)--(2.9,4);
			\draw[red,opacity=1,line width=1.3mm] (3.6,1)--(3.6,2);
			\draw[red,opacity=1,line width=1.3mm] (1,1)--(1,3);

			%Horizontal ASG
			\draw[semithick ] (14.5,4) -- (0,4);
			\draw[semithick ] (12.9,1) -- (1,1);
			\draw[semithick ] (8.5,2) -- (3.6,2);
			\draw[semithick ] (10.4,1) -- (1,1);
			\draw[semithick ] (8.5,3) -- (0,3);
			%		\draw[semithick ] (4.25,3) -- (0,3);
			\draw[semithick ] (2.9,0) -- (0,0);
			%		%Vertical ASG

			%Frame
			\draw[dashed] (0,-0.5) --(0,4.5);
			\draw[dashed] (14.5,-0.5) --(14.5,4.5);
			%\node[right] at (0, 4.5) {\color{red} $t$};  
			%\node[right] at (14.5, 4.5) {\color{red} $0$};  
			\node [below] at (0,-0.5) {\scalebox{1.2}{$0$}};
			
			\draw[opacity=1,semithick, gray, dotted] (-1,-1.1) -- (15.5,-1.1);
			\draw[opacity=1,semithick, gray] (0,-1.1) -- (0,-1.3);
			\draw[opacity=1,semithick, gray] (14.5,-1.1) -- (14.5,-1.3);
			\node [below] at (0,-1.3) {{\color{gray}\scalebox{1.2}{$t_0$}}};
			\node [below] at (14.5,-1.3) {{\color{gray}\scalebox{1.2}{$\Ts$}}};

			\node [below] at (14.5,-0.5) {\scalebox{1.2}{$\Ts- t_0$}};
			\node [above] at (0,4.5) {\scalebox{1.2}{\color{red} $\Ts-t_0$}};
			\node [above] at (14.5,4.5) {\scalebox{1.2}{\color{red}$0$}};
			\draw[-{angle 60[scale=5]},red,line width=1.3] (9.25,4.7) -- (5.25,4.7) node[text=red, pos=.5, yshift=6pt]{$r$};
			\draw[-{angle 60[scale=5]}] (5.25,-0.6) -- (9.25,-0.6) node[text=black, pos=.5, yshift=-6pt]{$t$};
			%horizontal lines
			\draw[opacity=1,semithick] (0,0) -- (14.5,0);
			\draw[opacity=1,semithick] (0,1) -- (14.5,1);
			\draw[opacity=1,semithick] (14.5,2) -- (0,2);
			\draw[opacity=1,semithick] (14.5,3) -- (0,3);
			\draw[opacity=1,semithick] (0,4) -- (14.5,4);
			%neutral arrows
			\draw[-{triangle 45[scale=5]},semithick,opacity=1] (10,4) -- (10,3);	\draw[-{triangle 45[scale=5]},semithick,opacity=1] (11.5,1) -- (11.5,0);
			\draw[-{triangle 45[scale=5]},semithick,opacity=1] (3.6,1) -- (3.6,2);
			\draw[-{triangle 45[scale=5]},semithick,opacity=1] (1,3) -- (1,1);
			%selective arrows
			\draw[-{open triangle 45[scale=5]},semithick,opacity=1] (12.9,1) -- (12.9,4);
			\draw[-{open triangle 45[scale=5]},semithick,opacity=1] (2.9,0) -- (2.9,4);
			\draw[-{open triangle 45[scale=5]},semithick,opacity=1] (6.2,4) -- (6.2,1);
			\draw[-{open triangle 45[scale=5]},semithick,opacity=1] (7,2) -- (7,0);
			\draw[] (7,2) circle (.6mm)  [fill=black!100];
			\draw[] (6.2,4) circle (.6mm)  [fill=black!100];
			\draw[] (2.9,0) circle (.6mm)  [fill=black!100];
			\draw[] (12.9,1) circle (.6mm)  [fill=black!100];

			%interactive arrows 
			
			\draw[] (8.5,3) circle (.6mm)  [fill=black!100];
			\draw[] (8.5,2) circle (.6mm)  [fill=black!100];
			\draw[-{open triangle 45[scale=5]},opacity=1,semithick] (8.5,3) -- (8.5,4);
			\draw[opacity=1,semithick] (8.5,2) -- (8.5,3);
			
			\draw[-{open triangle 45[scale=5]},semithick,opacity=1] (12.4,2) -- (12.4,3);
			\draw[opacity=1] (12.4,2) circle (.6mm)  [fill=black!100];
			\draw[opacity=1] (12.4,0) circle (.6mm)  [fill=black!100];
			\draw[opacity=1] (12.4,4) circle (.6mm)  [fill=black!100];
			\draw[opacity=1,semithick] (12.4,4) -- (12.4,3);
			\draw[opacity=1,semithick] (12.4,0) -- (12.4,2);

			%Mutations deleterious
			\node[opacity=1] at (11,3) {\scalebox{1.5}{$\times$}} ;
			\node[opacity=1] at (13.3,1) {\scalebox{1.5}{$\times$}};
			\node[opacity=1] at (12,0) {\scalebox{1.5}{$\times$}};
			\node[opacity=1] at (8,4) {\scalebox{1.5}{$\times$}};
			\node[opacity=1] at (2,2) {\scalebox{1.5}{$\times$}};
			%\node[opacity=1] at (5,0) {\scalebox{1.5}{$\times$}};    
			%beneficial mutations
			\draw (2.2,0)[opacity=1] circle (1.5mm)  [fill=white!100];    
			%\draw (3.2,4)[opacity=1] circle (1.5mm)  [fill=white!100];
			%\draw (10.4,1)[opacity=1] circle (1.5mm)  [fill=white!100];
			\draw (7.5,2)[opacity=1] circle (1.5mm)  [fill=white!100];

			\fill [red] (14.35,3.85) rectangle (14.65,4.15);
			\fill [red] (-.15,3.85) rectangle (.15,4.15);
			\fill [red] (-.15,2.85) rectangle (.15,3.15);
			\fill [red] (-.15,-.15) rectangle (.15,0.15);

	\end{tikzpicture}}
	\caption{In red, the ASG for one of the individuals in Fig.~\ref{fig:untypedmomo}, but now for the FTW model. Notice that arrows that start from  lines in the ASG and hit individuals not in the current graph are not relevant for the types in our initial sample. Grey dotted line, black arrow, and red arrow indicate absolute time, forward time increment, and backward time increment, respectively.}
	\label{fig:ASG} 
\end{figure}  
Our analysis of the MoMo is based on the ASG. It arises by tracing back in the graphical representation all lines that may carry information about the ancestry of a sample, that is, those lines that may influence the types in the sample when mutations are ignored. We call the corresponding individuals \emph{potential influencers} (a term borrowed from \citep[Sect.~8.1]{Donnelly1999a}). The collection of these influencer lines  as a function of time makes up the ASG; lines that do not belong to the collection at a given time are said to be \emph{outside the graph}. The true ancestry is only determined after assigning the types to all lines in the ancestral graph at some initial time and propagating them forward through the untyped ASG according to the propagation rules for the given model variant.

\smallskip

Specifying this approach to the DOM model becomes quickly intractable because of the asymmetric role of selective and checking arrows. This is why from now on, unless stated otherwise,  \emph{we require} that 
\begin{equation}\label{eq:non-increasing}
(\widehat s_m)_{m>0} \; \text{is non-increasing},
\end{equation}
so that both models become distributionally equivalent by Lemma~\ref{lem:DF}. In particular, we may (and will) equivalently work with the FTW model and take advantage of its higher symmetry, due to the symmetric role of arrows.

\smallskip

To construct the ASG for the FTW model (see Fig.~\ref{fig:ASG}), fix absolute times $t_0<\Ts$, where $\Ts$ is referred to as the present. We  use the variables $t$ and $r$ for increments (relative to $t_0$ and $\Ts$, respectively) forward and backward in time, so that for $t\in [0,\Ts-t_0]$, we have $r=\Ts-t_0-t$, that is, $r=0$  corresponds to time~$\Ts$. Start the graph from a collection of lines at time~$\Ts$ and call these lines the \emph{sample}. When tracing back their ancestry, each  line is hit at rate $s_m$ (for $m\in \N$) by selective arrows associated with an event of order~$m$. This causes a number $\leqslant m$ of new lines to branch off, which may or may not be  part of the graph yet; if at least one of these lines goes to the outside, we speak of a (binary or multiple) \emph{branching event}. Moreover, neutral arrows hit every line in the graph at rate~$1$. If one out of currently~$n$ lines in the graph is hit by a neutral arrow that comes from one of the $n-1$ remaining  potential influencers, we have  a \emph{coalescence event}, that is, the two lines merge into a single one; such events occur at rate $(n-1)/N$ per line in the graph.  If a neutral arrow comes from outside the current set of lines in the graph (rate $(N-n)/N$ per line in the graph), this causes a \emph{relocation event}, which leaves the number of lines in the graph unchanged. Deleterious and beneficial mutations appear on every line at rates $u \nu_1$ and $u \nu_0$, respectively. The resulting process takes values in the set of (line) labels, see Definition~\ref{def:ASG} for details.

\smallskip

Denote by $h_r(k)$ the conditional probability that the ancestor of a uniformly chosen individual is unfit at backward time~$r$ given the population at this time consists of $k$ type-$1$ individuals. Put differently, the (conditional) \emph{ancestral type at time $r$} (given~$k$) is Bernoulli distributed with parameter $h_r(k)$. The conditional \emph{common ancestor type} (given $k$) is Bernoulli distributed with parameter $h_{\infty}(k)\coloneqq \lim_{r\to\infty}h_r(k)$ if the limit exists.  Definition~\ref{def:ancestraltypedistribution} contains the precise formulation. The name \emph{common ancestor}  is motivated by the fact that,
in the MoMo, all individuals at present share a common ancestor in the sufficiently distant past, as illustrated in Fig.~\ref{ancline}.  To see this, note that the genealogy of the entire population is embedded in the ASG started from the entire population. The number of true ancestors never increases and is dominated by the number of potential influencers, that is, the line-counting process of the ASG. Since the latter is irreducible on the finite state space $[N]$, it is recurrent and  reaches~$1$ in finite time almost surely; and from this point onwards, the number of true ancestors is always~$1$. 
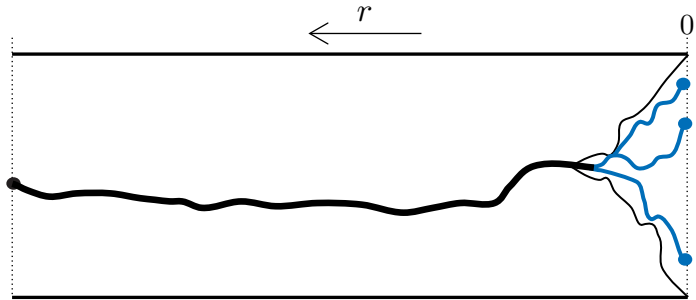
\begin{figure}[t]
	\scalebox{2.5}{\begin{tikzpicture}[y=0.80pt, x=0.80pt, yscale=-\globalscale, xscale=\globalscale, inner sep=0pt, outer sep=0pt]
			\begin{scope}[shift={(-30.09534,-91.13568)}]
				\begin{scope}[cm={{0.26458,0.0,0.0,0.26458,(-5.71556,82.08761)}}]
					
					%%%% Dotted Time lines
					\path[draw=black,dash pattern=on \pgflinewidth off .5pt,miter limit=1.00,line width=0.200pt] (140.5001,23) -- (140.5001,208.1513);
					
					\path[draw=black,dash pattern=on \pgflinewidth off .5pt,miter limit=1.00,line width=0.200pt] (618.0098,23) -- (618.0098,208.1513);
					
					\node [above] at (618,20.7) {\scalebox{0.44}{$0$}};
					%				\node [above] at (140,20.7) {\scalebox{0.44}{$-r$}};
					%%%% Frame
					\path[draw=black,line width=0.5pt] (140.3922,208.1513) -- (618.3922,208.1513);
					\path[draw=black,line width=0.5pt] (140.3927,34.6978) -- (618.3927,34.6978);
					
					%%% Time arrow
					\draw[-{angle 60[scale=2]},line width=0.2] (430,20) -- (350,20) node[text=black, pos=.5, yshift=3pt]{\scalebox{0.5}{$r$}};

					%%%%%Right side coalescent
					\path[draw=black,line join=miter,line cap=butt,line width=0.3pt] (536.9933,113.9372) .. controls (536.9933,113.9372) and (544.7100,109.5737) .. (548.8566,108.0517) .. controls (551.9150,106.9292) and (555.0864,105.6975) .. (558.3473,105.6975) .. controls (559.9777,105.6975) and (561.5626,107.4336) .. (563.0926,106.8747) .. controls (566.2849,105.7084) and (567.7422,101.7613) .. (569.0243,98.6349) .. controls (570.9846,93.8547) and (568.3331,87.5030) .. (571.3969,83.3326) .. controls (573.4804,80.4967) and (577.9410,80.5655) .. (580.8876,78.6242) .. controls (583.9249,76.6233) and (586.6993,74.2015) .. (589.1919,71.5616) .. controls (594.5993,65.8350) and (598.4623,58.8358) .. (603.4279,52.7280) .. controls (608.3562,46.6662) and (618.8503,35.0715) .. (618.8503,35.0715);

					\path[draw=black,line join=miter,line cap=butt,line width=0.3pt] (535.5691,114.6337) .. controls (535.5691,114.6337) and (543.0822,119.7903) .. (547.1193,121.5889) .. controls (550.0970,122.9155) and (553.1847,124.3710) .. (556.3595,124.3710) .. controls (557.9470,124.3710) and (559.4900,122.3195) .. (560.9796,122.9800) .. controls (564.0876,124.3581) and (564.4144,129.7546) .. (566.7547,132.7173) .. controls (570.0744,136.9198) and (574.9017,139.7615) .. (578.3050,143.8972) .. controls (581.4537,147.7236) and (582.4071,157.6012) .. (586.3901,156.3651) .. controls (597.0836,153.0464) and (595.5641,160.6069) .. (597.9403,164.7114) .. controls (601.6809,171.1723) and (597.2455,180.1367) .. (600.2504,186.9681) .. controls (603.9873,195.4637) and (618.7308,207.8337) .. (618.7308,207.8337);
					
					% Right Circles
					\path[xscale=1.000,yscale=-1.000,draw=NavyBlue,fill=NavyBlue,line width=0.400pt] (616.5671,-84.2876) ellipse (0.1150cm and 0.0986cm);
					
					\path[xscale=1.000,yscale=-1.000,draw=NavyBlue,fill=NavyBlue,line width=0.400pt] (615.2415,-55.9569) ellipse (0.1150cm and 0.0986cm);
					
					\path[xscale=1.000,yscale=-1.000,draw=NavyBlue,fill=NavyBlue,line width=0.400pt] (616.5673,-181.0941) ellipse (0.1150cm and 0.0986cm);		
					%Left Circle
					\path[xscale=1.000,yscale=-1.000,draw=Black,fill=Black,line width=0.400pt] (141.3886,-126.9441) ellipse (0.1150cm and 0.0986cm);
					
					%Blue lines
					\path[draw=NavyBlue,line width=0.600pt] (617.4450,181.1225) .. controls (612.0613,181.4137) and (611.1171,165.8070) .. (606.3329,159.1312) .. controls (603.9864,155.8568) and (601.7799,151.4750) .. (597.8705,150.5035) .. controls (595.9949,150.0374) and (594.0792,153.1142) .. (592.3452,152.2606) .. controls (588.1372,150.1892) and (589.2501,143.4034) .. (587.5467,139.0334) .. controls (585.7778,134.4954) and (586.2983,127.6815) .. (581.9195,125.5488) .. controls (575.4391,122.3924) and (566.3570,119.9258) .. (558.3509,117.8923) .. controls (555.2502,117.1048) and (559.1540,118.3328) .. (548.9144,116.1425);
					
					\path[draw=NavyBlue,line width=0.600pt] (618.0819,84.6076) .. controls (612.6983,84.3164) and (613.5287,101.5531) .. (606.9698,106.5989) .. controls (604.3665,108.6017) and (600.1367,107.0668) .. (597.2861,108.7129) .. controls (595.4244,109.7879) and (594.8145,112.3470) .. (592.9821,113.4695) .. controls (590.1394,115.2108) and (586.5666,116.8352) .. (583.2984,116.1120) .. controls (579.8258,115.3435) and (578.2327,111.1171) .. (575.2285,109.2414) .. controls (573.7379,108.3108) and (572.1272,107.4222) .. (570.3866,107.1274) .. controls (568.0825,106.7372) and (565.5171,106.6961) .. (563.3928,107.6559) .. controls (560.8583,108.8011) and (559.9978,112.2999) .. (557.4749,113.4695) .. controls (554.8506,114.6860) and (552.2486,115.2940) .. (548.8100,114.2926);
					
					\path[draw=NavyBlue,line width=0.600pt] (617.4450,56.4920) .. controls (612.0613,56.2008) and (611.5890,65.5563) .. (607.1471,68.3058) .. controls (604.3110,70.0614) and (600.0609,68.7345) .. (597.4633,70.8269) .. controls (593.5907,73.9465) and (595.1460,81.8718) .. (590.7168,84.1326) .. controls (588.7848,85.1188) and (586.2860,82.2622) .. (584.2899,83.1112) .. controls (580.6173,84.6731) and (579.8734,89.7676) .. (577.4414,92.9319) .. controls (575.4051,95.5812) and (573.1823,98.0832) .. (570.9710,100.5884) .. controls (569.2248,102.5667) and (567.4697,106.5253) .. (565.6056,106.4092);
					
					\path[draw=black,line width=1pt] (140.5837,127.1467) .. controls (144.4739,128.1579) and (147.3609,130.3969) .. (151.0258,131.7227) .. controls (156.0313,133.5336) and (161.0160,135.8522) .. (166.6889,136.2988) .. controls (173.1287,136.8057) and (179.3730,134.3936) .. (185.8326,134.0107) .. controls (193.9214,133.5313) and (202.1367,133.3584) .. (210.1974,134.0107) .. controls (217.3193,134.5871) and (224.0562,136.4823) .. (231.0815,137.4428) .. controls (234.5362,137.9151) and (238.0428,138.2054) .. (241.5236,138.5868) .. controls (245.0042,138.9681) and (248.4489,139.5395) .. (251.9656,139.7308) .. controls (254.8563,139.8880) and (257.8443,139.2927) .. (260.6673,139.7308) .. controls (265.7172,140.5144) and (269.4323,143.9346) .. (274.5900,144.3068) .. controls (284.1004,144.9931) and (292.8694,139.8373) .. (302.4355,139.7308) .. controls (310.7403,139.6383) and (318.5035,142.9063) .. (326.8003,143.1628) .. controls (336.7192,143.4694) and (346.4608,141.0758) .. (356.3860,140.8748) .. controls (366.8392,140.6630) and (377.4062,140.8505) .. (387.7122,142.0188) .. controls (397.8423,143.1672) and (407.0205,147.6001) .. (417.2980,147.7389) .. controls (424.4721,147.8358) and (431.2396,145.4994) .. (438.1821,144.3068) .. controls (444.5815,143.2075) and (450.7701,141.4519) .. (457.3258,140.8748) .. controls (465.3756,140.1662) and (474.6232,143.5052) .. (481.6906,140.8748) .. controls (487.5777,138.6837) and (488.3236,133.1099) .. (492.1327,129.4347) .. controls (497.4897,124.2659) and (500.9159,117.2100) .. (509.5361,114.5626) .. controls (522.3718,110.6205) and (538.2507,114.2282) .. (552.0028,115.7066);
				\end{scope}
			\end{scope}
	\end{tikzpicture}}
	\caption{All individuals in the population share a common ancestor in the sufficiently distant past.}
	\label{ancline}
\end{figure}

\smallskip

We aim at answering two questions: how can we determine via the ASG 1) the type distribution at present, and 2) the ancestral type distribution? If there are no mutations ($u=0$),  the answers are the same and tied to the line-counting process of the ASG starting with a sample of size one in a simple way: iterating the FTW rule, it becomes clear that the sample at time $r=0$ is unfit if and only if all lines in the graph are associated with the unfit type at any given backward time $r>0$;  likewise, all true ancestors of the sample are unfit in precisely this case.  If there are mutations, however, the two questions have different answers in general and require their own construction each, namely 1) the killed ASG with multiple branching and 2) the pruned lookdown ASG with multiple branching, both derived from the ASG by exploiting the information inherent in the mutation events. Both are generalisations of the corresponding processes  developed in \cite{LKBW15,BCH18jmath,baake2018lines} for the situation with binary branching in the law of large numbers and the diffusion limit, respectively.

\subsection{Type distribution via a killed ASG with multiple branching}\label{subsec:kASG}
The probability that all individuals in a sample from the present population are unfit can be deduced via two elementary but crucial insights that give rise to a modified ASG. First, the type of an individual at present is determined by the most recent mutation along its ancestral line. In particular, if the most recent mutation on a line of a potential influencer is of type~$1$, the individual is beneficial if and only if one of the remaining potential influencers is of type~$0$. Hence, we need not trace the line with the mutation any further and may instead prune it, that is remove it from the graph.  Second,  due to the FTW rule, type $0$ has priority at every branching event; if the most recent mutation on any line (that has not been pruned) is beneficial,  this means that a potential influencer of an individual in the sample is of type $0$. Due to the type propagation, at least one individual in the sample then has  type $0$ as well, thus  "killing" our chances of a completely unfit sample. We  therefore \emph{kill} the process, that is, send it to a \emph{cemetery state} $\Delta$. The resulting process is called the \emph{killed ASG} (kASG) and we write $K_r$ for the collection of its line labels at backward time~$r$.

\smallskip

In what follows, we will not need the full complexity of the kASG, but only its \emph{(generalised) line-counting process} $R=(R_r)_{r\geq 0}$, where $R_r \defeq \lvert K_r\rvert$ (with the convention $R_r=\Delta$ if $K_r=\Delta$, and $R_r=0$ if $K_r=\varnothing$). It will indeed turn out that~$R$ suffices to determine the type distribution at backward time 0 (for any given  exchangeable type distribution at backward time $r$). The following proposition summarises the transition rates of this line-counting process.

\begin{proposition}\label{prop:lckASG}
	The generalised line-counting process $R$ is a continuous-time Markov chain on~$[N]_{0,\Delta} \defeq [N]\cup \{0,\Delta\}$. The corresponding infinitesimal generator acts on functions $\tilde f: [N]_{0,\Delta} \to \R$ and is given by $\mathcal{A}_R  = \mathcal{A}_R^{\rm{n}} + \sum_{m>0} \mathcal{A}_R^{s_m} + \mathcal{A}_R^{u}$ with the building blocks defined for $n\in [N]_0$, $m \in \N$, via 
	\begin{align}
		\mathcal{A}_R^{\rm{n}} \tilde f(n) & \defeq n \frac{n-1}{N} [ \tilde f (n-1) - \tilde f(n)], \label{eq:generatorRcoalescence}\\
		\mathcal{A}_R^{s_m} \tilde f(n) & \defeq s^{}_m \frac{n}{N^m} \sum_{j=1}^m  (N-n)^{\underline{j}} \, C^n_{mj} [ \tilde f (n+j) - \tilde f(n)], \label{eq:generatorRselection} \\
		\mathcal{A}_R^{u} \tilde f(n) & \defeq n u \nu^{}_1  [ \tilde f (n-1) - \tilde f(n)] + n u \nu^{}_0  [ \tilde f (\Delta) - \tilde f(n)], \label{eq:generatorRmutation}
	\end{align}
	where $y^{\underline{j}} \defeq y(y-1)(y-2)\dots(y-j+1)$ is the falling factorial,
	\[
	C^n_{mj} \defeq  \sum_{\ell=j}^m \binom{m}{\ell}\stirlingii{\ell}{j}  n^{m-\ell},
	\] and the $\stirlingii{\ell}{j}$ are the Stirling numbers of the second kind. Moreover, $\mathcal{A}_R^{\rm{n}}\tilde f(\Delta)  =  \mathcal{A}_R^{s_m} \tilde f(\Delta) = \mathcal{A}_R^{u} \tilde f(\Delta) \defeq 0$.
\end{proposition}
The proof of the proposition essentially boils down to a combinatorial argument that~\eqref{eq:generatorRselection} indeed corresponds to the rate at which the number of lines in the kASG increases; \eqref{eq:generatorRcoalescence} and \eqref{eq:generatorRmutation} are immediate. For details see Section~\ref{sec:kASG}, where we formally construct the kASG as a set-valued process.

The next theorem establishes that the line-counting process of the kASG carries enough information to determine the factorial moments of the FTW MoMo.

\begin{theorem}[Factorial moment duality]\label{thm:duality_R}
	Let $Y$ be the frequency process of the unfit individuals of the FTW MoMo and~$R$ the line-counting process of the kASG. Then, for all $t \geq 0$, $n \in [N]_{0,\Delta}$, and $k  \in [N]_0$, 
	\begin{equation}\label{eq:factmomentdualityYR}
		\mathbb{E} \Big [\frac{{ Y_t}^{\underline{n}}}{N^{\underline{n}}} \mid { Y_0} = {k} \Big ] = 	\E \Big [\frac{{ k}^{\underline{R_t}}}{N^{\underline{R_t}}} \mid { R_0} = { n} \Big ],
	\end{equation}
	where $k^{\underline{\Delta}}/N^{\underline{\Delta}} \defeq 0 \; \forall k$. 
	So ${Y}$ and ${R}$ are factorial moment (or hypergeometric) duals, that is,  dual w.r.t. the duality function 
	\[
	H_F (k, n) \defeq \frac{k^{\underline{n}}}{N^{\underline{n}}}.
	\]
\end{theorem}		
The formal proof is based on generator calculations and can be found in Section~\ref{sec:kASG}. Here we provide  a plausibility argument that appeals to the intuition gained by the graphical construction. Consider a population with $Y_0 = k$, then let the process $Y$ evolve for a time $t$ and sample $n$ individuals from the population at time $t$ (with $Y_t$ unfit individuals) without replacement. The left-hand side of~\eqref{eq:factmomentdualityYR} is the probability that all individuals in the sample are  unfit. On the other hand, starting from a number $n$ of lines in the ASG, let $R$ run for  time $t$, then  sample without replacement $R_t$ individuals from the initial population with $Y_0=k$ type-1 individuals. If $R_t \notin \{0, \Delta\}$, then  $k^{\underline{R}_t}/N^{\underline{R}_t}$ is the probability that all lines that have not been pruned have type $1$ in the past, and hence all individuals in the sample are unfit. If $R_t=0$, all lines have been pruned by deleterious mutations, so the individual is unfit with probability 1; and if  $R_t=\Delta$, the process has been killed, so at least one individual is fit.

\smallskip

Let us mention that factorial moment dualities  have a long history in the context of fixed-size population genetic models. In the neutral case, they already appear in papers by Cannings \cite{Cannings74} and Gladstien \cite{Gladstien77a,Gladstien77b,Gladstien78} in the 1970's, at a time where neither the coalescent process nor the concept of dualities for Markov chains had been formulated yet. Rather, the dualities appear in terms of algebraic identities between matrices, and are used  to calculate the eigenvalues of the Markov transition matrix via a similarity transform; the connection with the backward point of view is at most implicit. In 1999, M\"ohle \cite{Moehle99} established factorial moment dualities in neutral genetic models, explicitly and in terms of backward processes.

\smallskip 

There are two distinct regimes depending on the presence of mutation. On the one hand, if $u = 0$, $Y$ absorbs in 0 or $N$ (so one of the two types dies out), while $R$ is positive recurrent and  therefore converges to  a unique stationary distribution~$\pi_R$~\citep[Thm. 3.5.3, Thm. 3.6.2]{norris1998markov}. On the other hand, if $u>0$, it is $Y$ that is positive recurrent and converges  to a unique stationary distribution~$\pi_Y$, while $R$ absorbs in~$0$ or~$\Delta$. We use this connection to derive a representation of the absorption probabilities of~$Y$ and $R$ if $u=0$ and $u>0$, respectively. Denote by $Y_\infty$ and $R_\infty$ a random variable on~$[N]_{0} \defeq [N] \cup \{0\}$ and $[N]_{0,\Delta}$, respectively, with distribution~$\pi_Y$ and~$\pi_R$.

\begin{corollary}[Representation absorption probabilities]\label{cor:repr.absorpt.prob.finite}
	Suppose $u=0$. For $k\in [N]_0$, $$\P(\lim_{t\to\infty} Y_t=N\mid Y_0=k)=\E \Big [\frac{{ k}^{\underline{R_\infty}}}{N^{\underline{R_\infty}}}\Big ].$$
	Suppose $u>0$. For $n\in [N]_{0,\Delta} $, $$\P(\lim_{t\to\infty} R_t=0\mid R_0=n)=\E \Big [\frac{Y_{\infty}^{\underline{n} } }{N^{\underline{n}}}\Big].$$
\end{corollary}
\begin{proof}
	If $u=0$, setting $n=1$ in the duality in Theorem~\ref{thm:duality_R} and taking $t\to\infty$ yields the first result. If $u>0$, setting $k=1$ in Theorem~\ref{thm:duality_R} and taking $t\to\infty$ yields the second one.
\end{proof}

\subsection{Siegmund duality}
The type distribution in the MoMo may also be expressed via the  Siegmund dual process.  Siegmund duality has been observed and applied in many contexts, such as birth-death processes (e.g.\ \cite{DetteFillPitmanStudden97}), ruin problems (e.g.\ \cite[Ch.~XIV.5]{Asmussen2003applied}), interacting particle systems (e.g.\ \cite{CliffordSudbury85}), and population genetics (e.g.\ \cite{BLW16}). We focus here on an interpretation in terms of the graphical representation of the MoMo and establish a connection to a functional of the ASG that, to the best of our knowledge, has not appeared in the literature so far. Let us first recall Siegmund duality on $[N]_0$. The following result  is a corollary of~\citet[Thm.~3]{siegmund1976equivalence} (see also \citep{vanDoorn80} and \citep[Sec.~2]{CoxRoesler83}). We provide a short proof in our finite context in Section~\ref{sec:siegmund} based on generator arguments.

\begin{lemma}[Siegmund duality,{\citep[Thm.~3]{siegmund1976equivalence}}] \label{lem:bdduality}
	Let $X$ be a continuous-time birth-death process on $[N]_0$ with birth rates $\{\lambda_x \}_{\{0 \leq x <  N\}}$ and death rates $\{\mu_x\}_{\{0 < x\leq N\}}$, complemented by $\lambda_N=\mu^{}_0=0$. Denote as $X^S$ the continuous-time birth-death process on $[N+1]_0$ with birth rates $\lambda^*_x$ and death rates $\mu^*_x$, respectively, as given by 
	\begin{align*}
		\mu_x^{\ast} &\defeq \lambda_{x-1} \; \text{ for } x \in [N+1],\quad \, \lambda_x^{\ast} \defeq \mu_{x} \; \text{ for } x \in [N]_0.
	\end{align*}  
	Then, $X$ and $X^S$ are Siegmund duals, that is,
	\begin{equation}
		\mathbb{P} (X_t \geq x^\ast \mid X_0 = x) = \mathbb{P}(x \geq X_t^{S} \mid X^{S}_0 = x^\ast )  \quad \text{for } x \in [N]_0,\, x^\ast \in [N+1]_0,\, t\geq 0.
		\label{eq:sigmundduality}
	\end{equation}
	Put differently,  the processes are dual with respect to the duality function $H_S(x,x^\ast) \defeq \ind_{\{x \geq x^\ast \}}.$
\end{lemma}

\begin{remark} \label{rem:dualequiv}
If $X$ is irreducible on $[N]_0$, then (the only) absorbing states of $X^S$ are $0$ and $N+1$; on the other hand, if $X$ absorbs in $0$ and $N$, then $0$ and $N+1$ are isolated states for $X^S$, and its restriction to $[N]$ is irreducible and also Siegmund dual to $X$. A direct consequence of the lemma, which was also the motivation for the original setting, is the equivalence between absorption probabilities of one process and the stationary distribution of its dual \cite{siegmund1976equivalence,vanDoorn80,CoxRoesler83}. 
\end{remark}

The following corollary is an immediate consequence of the lemma.

\begin{corollary}
	Let $Y$ be the type-$1$ frequency process in an FTW MoMo and let $Y^S$ be the birth-death process on $[N]$ with birth rates $\lambda^\ast_k$ and  and death rates $\mu^\ast_k $ given by
	\[\lambda^\ast_k \defeq k \Big(\frac{N-k}{N}+u\nu_0+\sum_{m>0}s_m\Big(1-\Big(\frac{k}{N}\Big)^m\Big)\Big) \quad \text{and } \; \mu^\ast_k \defeq (N-k+1)\Big(\frac{k-1}{N}+u\nu_1\Big), \quad k\in [N].
	\]
 Let $k_0\in [N]_0$ and $k_0^\ast\in[N].$ Then, for $t\geq 0$, $$\P(Y_t\geq k_0^\ast\mid Y_0=k_0)=\P(k_0\geq Y^S_t\mid Y_0^S=k_0^\ast).$$
%	In particular, if $u=0$, $$\lim_{t\to\infty} \P(Y_t=0\mid Y_0=k_0)=\frac{\sum_{\ell=k_0+1}^N \eta(\ell)}{\sum_{\ell=1}^N \eta(\ell)},$$ where $\eta(\ell)=\prod_{i=1}^{\ell-1}\big(1+\sum_{m>0}s_m\frac{1-(i/N)^m}{1-i/N} \big)$.
\end{corollary}

\begin{figure}[t]
	\scalebox{0.7}{\begin{tikzpicture}
			\draw[red,opacity=1,line width=1.3mm] (0,4)--(12.9,4);
			%		\draw[red,opacity=1,line width=1.3mm] (12.9,1)--(2.9,1);	
			\draw[red,opacity=1,line width=1.3mm] (2.9,1)--(1,1);		
			\draw[red,opacity=1,line width=1.3mm] (3.6,2)--(8.5,2);
			\draw[red,opacity=1,line width=1.3mm] (0,3)--(8.5,3);
			%		\draw[red,opacity=1,line width=1.3mm] (0,0)--(2.9,0);
			%		\draw[red,opacity=1,line width=1.3mm] (14.5,4)--(12.9,4);
			\draw[red,opacity=1,line width=1.3mm] (12.9,1)--(12.9,4);
			\draw[red,opacity=1,line width=1.3mm] (8.5,2)--(8.5,4);
			\draw[red,opacity=1,line width=1.3mm] (6.2,1)--(6.2,4);
			\draw[red,opacity=1,line width=1.3mm] (2.9,0)--(2.9,4);
			\draw[red,opacity=1,line width=1.3mm] (3.6,1)--(3.6,2);
			\draw[red,opacity=1,line width=1.3mm] (1,1)--(1,3);
			
			\draw[cyan,opacity=1,line width=1.3mm] (0,0)--(2.9,0);
			%\draw[blue,opacity=1,line width=1.1mm] (2.9,0)--(2.9,1);
			\draw[cyan,opacity=1,line width=1.3mm] (2.9,1)--(12.9,1);
			%\draw[blue,opacity=1,line width=1.1mm] (12.9,1)--(12.9,4);
			\draw[cyan,opacity=1,line width=1.3mm] (12.9,4)--(14.5,4);

			%Horizontal ASG
			\draw[semithick ] (14.5,4) -- (0,4);
			\draw[semithick ] (12.9,1) -- (1,1);
			\draw[semithick ] (8.5,2) -- (3.6,2);
			\draw[semithick ] (10.4,1) -- (1,1);
			\draw[semithick ] (8.5,3) -- (0,3);
			%		\draw[semithick ] (4.25,3) -- (0,3);
			\draw[semithick ] (2.9,0) -- (0,0);
			%		%Vertical ASG

			%Frame
			\draw[dashed] (0,-0.5) --(0,4.5);
			\draw[dashed] (14.5,-0.5) --(14.5,4.5);    
			\node[above] at (0, 4.5) {\color{red} $\Ts-t_0$};  
			\node[above] at (14.5, 4.5) {\color{red} $0$};
			\node [below] at (0,-0.5) {$0$};
			\node [below] at (14.5,-0.5) {$\Ts-t_0$};
			\draw[-{angle 60[scale=5]},red,line width=1.3] (9.25,4.7) -- (5.25,4.7) node[text=red, pos=.5, yshift=6pt]{$r$};
			\draw[-{angle 60[scale=5]}] (5.25,-0.6) -- (9.25,-0.6) node[text=black, pos=.5, yshift=-6pt]{$t$};
			%horizontal lines
			\draw[opacity=1,semithick] (0,0) -- (14.5,0);
			\draw[opacity=1,semithick] (0,1) -- (14.5,1);
			\draw[opacity=1,semithick] (14.5,2) -- (0,2);
			\draw[opacity=1,semithick] (14.5,3) -- (0,3);
			\draw[opacity=1,semithick] (0,4) -- (14.5,4);
			%line numbering
			\node [left] at (-0.2,0) {$5$};
			\node [left] at (-0.2,1) {$4$};
			\node [left] at (-0.2,2) {$3$};
			\node [left] at (-0.2,3) {$2$};
			\node [left] at (-0.2,4) {$1$};
			%neutral arrows
			%\draw[-{triangle 45[scale=5]},semithick,opacity=1, visible on=<1-3>] (10,4) -- (10,3);
			%\draw[-{triangle 45[scale=5]},semithick,opacity=1, visible on=<1-3>] (11.5,1) -- (11.5,0);
			\draw[-{triangle 45[scale=5]},semithick,opacity=1] (3.6,1) -- (3.6,2);
			\draw[-{triangle 45[scale=5]},semithick,opacity=1] (1,3) -- (1,1);
			%selective arrows
			\draw[-{open triangle 45[scale=5]},semithick,opacity=1] (12.9,1) -- (12.9,4);
			\draw[-{open triangle 45[scale=5]},semithick,opacity=1] (2.9,0) -- (2.9,4);
			\draw[-{open triangle 45[scale=5]},semithick,opacity=1] (6.2,4) -- (6.2,1);
			%\draw[-{open triangle 45[scale=5]},semithick,opacity=1, visible on=<1-3>] (7,2) -- (7,0);
			\draw[] (12.9,1) circle (.6mm)  [fill=black!100];
			\draw[] (2.9,0) circle (.6mm)  [fill=black!100];
			\draw[] (6.2,4) circle (.6mm)  [fill=black!100];
			
			%interactive arrows 
			
			\draw[] (8.5,3) circle (.6mm)  [fill=black!100];
			\draw[] (8.5,2) circle (.6mm)  [fill=black!100];
			\draw[-{open triangle 45[scale=5]},opacity=1,semithick] (8.5,3) -- (8.5,4);
			\draw[opacity=1,semithick] (8.5,2) -- (8.5,3);
			
			\fill [red] (14.35,3.85) rectangle (14.65,4.15);
			\fill [red] (-.15,3.85) rectangle (.15,4.15);
			\fill [red] (-.15,2.85) rectangle (.15,3.15);
			\fill [red] (-.15,-.15) rectangle (.15,0.15);

	\end{tikzpicture}}
	\caption{The same ASG as in Fig.~\ref{fig:ASG}, but without mutations. In light blue, the maximal line in the ASG.}
	\label{fig:Mprocess}
\end{figure}
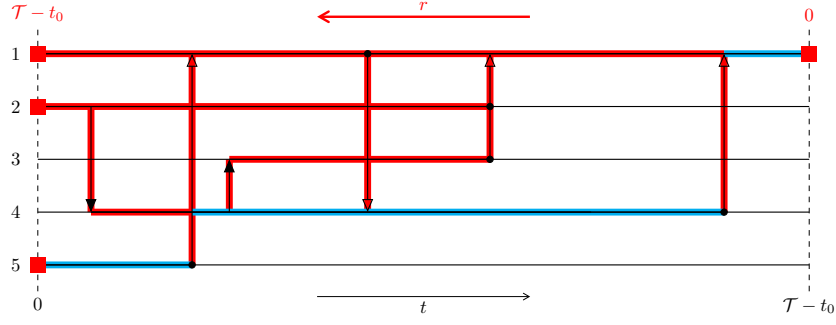

Consider now the kASG and define the \emph{maximal influencer line process} (or max-line process for short) $M=(M_r)_{r \geq 0}$ on $[N]_{0,\Delta}$, where $M_r \defeq \max  K_r$ with the convention that $M_r\defeq \Delta$ if $K_r=\Delta$, and $M_r\defeq 0$ if $K_r=\varnothing$. Fig. \ref{fig:Mprocess} illustrates this process. Let $p_r$ and $q_r$ be the distributions (understood as row vectors) of $R_r$ and $M_r$, respectively.  When starting from an exchangeable distribution for the lines contained in the kASG, the kASG stays exchangeable for all times, since its transitions do not depend on the line labels. Given this exchangeable setting, for any time $r$, $p_r$ and $q_r$ satisfy $q_r = p_r T$, where $T=(T(j,k)\in \R:j,k\in [N]_{0,\Delta})$ is the (matrix representation of the) linear transformation defined by $T(0,0)\defeq 1 \eqdef T(\Delta,\Delta)$, namely,
\begin{equation}\label{eq:T}
	T(j,k) \defeq 
	\frac{\binom{k-1}{j-1}}{\binom{N}{j}} \quad  \forall\, k,j\in[N],
\end{equation}
and all other entries are~$0$. Note that $$T(j,k)=\PP(M_r=k \mid R_r=j).$$ Indeed, for $k,j\in [N]$, place $j$ indistinguishable balls (the lines of the ASG) into $k$ out of~$N$ distinguishable boxes (the sites) so that each of the~$k$ boxes receives at most one ball; there are $\binom{N}{j}$ possibilities altogether, of which $\binom{k-1}{j-1}$ choices place $j-1$ of the balls in boxes $1, \dots, k-1$ and the remaining ball in box $k$. The remaining cases follow from the definition of $M$.

\medskip

Since $\frac{\dd p_r}{\dd r} = p_r Q_R$ with $Q_R$ the generator matrix of $R$, we have
\begin{equation}\label{tildeQM}
	\frac{\dd q_r}{\dd r}  = q_r \widetilde{Q}_M \text{ with } \widetilde{Q}_M \defeq T^{-1}Q_R T,
\end{equation}
where $T^{-1}$ is obtained via the block diagonal property of~$T$ and binomial inversion \cite[Cor.~3.38]{Aigner79} as
\begin{equation}\label{Tinv}
	(T^{-1})(j,k) = (-1)^{j + k} \binom{N}{k} \binom{k-1}{j-1},\qquad \forall \, j,k\in[N],
\end{equation}
complemented by $T^{-1}(0,0)=T^{-1}(\Delta,\Delta)=1$ and~$0$ otherwise. While $M$ is not Markov, $\widetilde Q_M$ is a Markov generator. Moreover, $\widetilde Q_M$ turns out to be the generator of the Siegmund dual~$Y^S$ of~$Y$ when we identify the state~$\Delta$ with $N+1$. This is a consequence of the following result, which will be proved in Section~\ref{sec:siegmund}.
\begin{theorem}\label{thm:factsieg}
	Consider a Markov process $X$ on $[N]_{0}$. 
	\begin{enumerate}
		\item If $X$ admits a factorial moment dual $X^F$ on $[N+1]_0$ with generator matrix $Q_F$, then $X$ admits a Siegmund dual $X^{S}$ on $[N+1]_0$ if and only if $T^{-1}Q_F T$ is a generator matrix, in which case it is the generator matrix of $X^{S}$; \label{factsieg}
		\item If $X$ admits a Siegmund dual $X^{S}$ on $[N+1]_0$ with generator matrix $Q_S$, then $X$ has a factorial moment dual $X^F$ on $[N]$ if and only if $TQ_S T^{-1}$ is a generator matrix, in which case it is the generator matrix of $X^F$. \label{siegfact}
	\end{enumerate}
\end{theorem}
Thus, $M$ has the same finite-dimensional distributions as $Y^S$, given an initial type assignment that is exchangeable. The benefit of this result is that the Siegmund dual has been given a meaning in terms of the ancestral process, via the max-line process and relation~\eqref{eq:T}.

\subsection{Ancestral type distribution}\label{mr:atd}
Our analysis of the conditional ancestral type distribution illustrates the power and versatility of the genealogical approach. More precisely, we exploit that a suitable ordering and pruning of the ASG leads to a tractable process that allows to determine the common ancestor and its type distribution. These ideas led to the representation of the common ancestor type distribution in the case of genic selection~\citep{Cordero17DL}, the only finite-population result available so far (see \citep{Kluth2013,LKBW15} for large population limits). Here, we extend this approach to the FTW case.

\smallskip

The overall idea is as follows. Assume the initial sample consists of a single individual. First, we \emph{prune} away those lines in the ASG that, due to mutations, are not potential ancestors of  the sampled individual.  Second, we \emph{order} the potential ancestors in a  manner bearing elements of the lookdown construction \cite{DK99}. More precisely, we assign to each potential ancestor a (finite) \emph{level} that reflects the priority (or \emph{pecking order}) to be the ancestor of the initial sample when the line is fit (see Prop.~\ref{prop:levelancestralsite} for the precise statement); this leads to an enumeration of the potential ancestors. The ordering is visualised  by placing the lines on top of each other according to their level, starting at level~$1$.
Third, we keep track of one distinguished line out of the potentially ancestral ones, which will turn out to be ancestral if a site colouring assigns only unfit types (see again Prop.~\ref{prop:levelancestralsite} for details). We call this distinguished line  \emph{immune} (following~\citep{LKBW15}). Lines in the ASG that are not potential ancestors are assigned the level~$\infty$, with the convention that $\infty=\infty-1$. There will always be at least one potential ancestor (i.e. with a finite level). The resulting process is called the \emph{pruned lookdown ASG} (pLD-ASG). In our figures, we only include lines that are potential ancestors.

\smallskip

Here, we construct the pLD-ASG informally appealing to Fig.~\ref{fig:transitionpLDASG} (see also Fig.~\ref{fig:pldASGfinitepopulations}); the rigorous formulation can be found in Definition~\ref{def:pLDASG}.  The events in the graphical representation affect the level ordering in the following way as we go backward in time. 
\begin{enumerate}
	\item[(1.a)] (Coalescence) 
	If two lines in the ASG coalesce, the line at the tail of the arrow (that is, the line remaining in the ASG and hence in the pLD-ASG) takes the lower level of the two coalescing lines. To fill up the (level) gap left by the removed line (the one at the tip), all finite levels are rearranged such that the set of finite levels remains an enumeration and the relative order between the potential parents is preserved. In particular, if both coalescing lines have level $\infty$, nothing happens. If the line at the tail has level $\infty$ but the line at the tip has a finite one, the levels are relabelled, but the event is invisible in our figures of the pLD-ASG.
	\item[(1.b)] (Relocation)
	If a line in the ASG  is relocated to a site that was outside the ASG, the line at the tail of the arrow takes the level of the line at the tip, while the latter is removed. In any case, these events are invisible in the graphical representation of the pLD-ASG.
	\item[(2)] (Selection)
	If a selective event hits a line in the ASG,  those incoming lines that are already in the graph at levels below the continuing line will stay where they are. The  remaining incoming lines  are assigned consecutive levels in the order given by $J$, starting with the level of the continuing line. The levels of all  other lines (that is, at levels at or above the continuing line and not belonging to $J$) are shifted above the incoming lines  such that the relative order among them is preserved. In particular, if the continuing line has level~$\infty$, all new lines have level~$\infty$ too; this is invisible in our figures.
	\item[(3.a)] (Deleterious mutation on immune line) If the immune line receives a deleterious mutation,  we relocate the line to the currently-highest finite level. All unaffected levels are reordered to fill the space, such that the relative order is preserved. In particular, if the mutation happens at level $\infty$, nothing happens.
	\item[(3.b)] (Deleterious mutation not on immune line) If any other line obtains a deleterious mutation, it ceases to be a potential parent and thus moves to level $\infty$. 
	Again, all unaffected levels are reordered to fill the space, such that the relative order is preserved. In particular, if the mutation happens at level~$\infty$, nothing happens.
	\item[(4)] (Beneficial mutation)
	A beneficial mutation means that all lines above it are not potential parents any more and thus moves them to level $\infty$. If the mutation happens at level $\infty$, nothing happens.
\end{enumerate}
A beneficial mutation on a line with finite level makes the (level of the) \emph{immune line} move to the level of the mutation (which is, by construction, the highest finite level after the event). In a coalescence event, the immune line moves to the (new) level together with its (new) site. In all other events, the site of the immune line  is unaffected, but  inherits its  (new) level. The assignment of types and the propagation of types and ancestry  also translate to the levels in the pLD-ASG. We say that a level has type~$0$ (or type~$1$) at (backward) time~$r$ if the line associated with that level has type~$0$ (type~$1$) at time~$r$  in the original ASG. \\

Let $L=(L_r)_{r\geq 0}$ be the line-counting process of the pLD-ASG, where $L_r$ counts the lines with finite level in the pLD-ASG at backward time~$r$. The next result provides the rates of the process and will be proved in Section~\ref{sec:ancestraltype}.
\begin{proposition}\label{prop:eqdistL}
	The line-counting process $L$ of a pLD-ASG is a continuous-time Markov chain on~$[N]$. 
	The corresponding infinitesimal generator acts on functions~$\tilde{f}:[N]\to\R$ and is given by $\As_{L}=\As_{R}^{\rm{n}}+\sum_{m>0}\As_{R}^{s_m}+\As_{L}^{\nu_0}+\As_{L}^{\nu_1}$, with~$\As_{R}^{\rm{n}}$ and~$\As_{R}^{s_m}$ of \eqref{eq:generatorRcoalescence} and \eqref{eq:generatorRselection}, respectively, and for $n\in [N]$, \begin{align}
		\As_{{L}}^{\nu_0}\tilde{f}(n)\defeq u\nu_0\sum_{j=1}^{n-1}\,[\tilde{f}(j)-\tilde{f}(n)],\qquad 
		\As_{L}^{\nu_1}\tilde{f}(n)\defeq u\nu_1\,(n-1)\,[\tilde{f}(n-1)-\tilde{f}(n)].\label{eq:generatorLmut}\end{align}
\end{proposition}
Since $\sum_{m>0}s_m>0$, $L$ is irreducible and converges in distribution to its stationary measure. We denote by~$L_{\infty}$ a random variable distributed according to this measure.

\begin{figure}[t]
	\centering
	\begin{minipage}{.2\textwidth}
		\begin{center}
			\scalebox{0.75}{\begin{tikzpicture}
				%%%%%Coalescence
				\draw[opacity=0,line width=1.3mm] (0,0) -- (0,3.6);
				
				%%%%%%Immune line
				\draw[gray,line width=1mm,opacity=.7] (0,0) -- (1,0);
				\draw[gray,line width=1mm,opacity=.7] (1,1.4) -- (2.5,1.4);
				\draw[gray,line width=1mm,opacity=.7] (1,0) .. controls (0.8,0.4) and (0.8,1) .. (1,1.4);
				%%%%%%%%%%%%%%%%%%%%%%%%%%
				
				\draw (0,1.4) -- (0.7,1.4);
				\draw (0,.7) -- (1,.7);
				\draw (0,0) -- (1,0);
				
				\draw[-{triangle 45[scale=5]},opacity=1] (1,0) .. controls (0.8,0.4) and (0.8,1) .. (1,1.4);
				
				\draw (0.7,1.4) -- (1,2.1);
				
				\draw (1,2.1) -- (2.5,2.1);
				\draw (1,1.4) -- (2.5,1.4);
				\draw (1,.7) -- (2.5,.7);
				\draw (1,0) -- (2.5,0);
			\end{tikzpicture}}\\
			(1.a) coalescence \\ ~~
		\end{center}
	\end{minipage}\begin{minipage}{.2\textwidth}
		\begin{center}
			\scalebox{0.75}{\begin{tikzpicture}
				\draw[opacity=0,line width=1.3mm] (0,0) -- (0,3.6);
				
				%%%%%%Immune line
				\draw[gray,line width=1mm,opacity=.7] (0,2.1) -- (0.9,2.1);
				\draw[gray,line width=1mm,opacity=.7] (1.2,0.7) -- (2.5,0.7);
				\draw[gray,line width=1mm,opacity=.7] (0.9,2.1) -- (1.2,0.7);
				%%%%%%%%%%%%%%%%%%%%%%%%%%
				
				\draw (0,3.5) -- (0.9,3.5);
				\draw (0,2.8) -- (0.9,2.8);
				\draw (0,2.1) -- (0.9,2.1);
				\draw (0,1.4) -- (0.9,1.4);
				\draw (0,.7) -- (0.9,.7);
				\draw (0,0) -- (0.9,0);
				
				\draw[] (0.9,2.8) circle (.4mm)  [fill=black!100];
				\draw[] (0.9,2.1) circle (.4mm)  [fill=black!100];
				\draw[] (0.9,1.4) circle (.4mm)  [fill=black!100];
				\draw[] (0.9,0.7) circle (.4mm)  [fill=black!100];
				\draw[] (0.9,0) circle (.4mm)  [fill=black!100];
				
				%\draw[-{open triangle 45[scale=5]}] (1,0) .. controls (0.8,0.4) and (0.8,1) .. (1,1.4);
				
				\draw[-{open triangle 45[scale=5]}] (0.9,0.7) -- (1.15,0.7);
				
				%			\draw[-{open triangle 45[scale=5]}] (0.7,0.7) -- (1,.7);
				%\draw[-{open triangle 45[scale=5]}] (0.7,1.4) -- (1,0.7);
				%\draw[-{open triangle 45[scale=5]}] (0.7,2.1) -- (1,0.7);
				\draw (0.9,3.5) -- (1.2,2.8);
				\draw (0.9,2.1) -- (1.2,0.7);
				\draw (0.9,1.4) -- (1.2,2.1);
				\draw (0.9,2.8) -- (1.2,1.4);
				\draw (0.9,0) -- (1.2,0);
				\draw (0.9,2.8) .. controls (0.6,2) and (0.6,1.2) .. (0.9,0.7);
				\draw (0.9,2.1) .. controls (0.7,1.7) and (0.7,1) .. (0.9,0.7);
				\draw (0.9,1.4) .. controls (0.8,1.1) and (0.8,1) .. (0.9,0.7);
				\draw (0.9,0) .. controls (0.8,0.15) and (0.8,0.55) .. (0.9,0.7);

				\draw (1.2,2.8) -- (2.5,2.8);
				\draw (1.2,2.1) -- (2.5,2.1);
				\draw (1.2,1.4) -- (2.5,1.4);
				\draw (1.2,.7) -- (2.5,.7);
				\draw (1.2,0) -- (2.5,0);
			\end{tikzpicture}}\\
			(2) selection \\ ~~
		\end{center}
	\end{minipage}\begin{minipage}{.2\textwidth}
		\begin{center}
			\scalebox{0.75}{\begin{tikzpicture}
				\draw[opacity=0,line width=1.3mm] (0,0) -- (0,3.6);
				
				%%%%%%Immune line
				\draw[gray,line width=1mm,opacity=.7] (0,2.1) -- (0.9,2.1);
				\draw[gray,line width=1mm,opacity=.7] (1.2,0.7) -- (2.5,0.7);
				\draw[gray,line width=1mm,opacity=.7] (0.9,2.1) -- (1.2,0.7);
				%%%%%%%%%%%%%%%%%%%%%%%%%%
				
				\draw (0,2.1) -- (0.9,2.1);
				\draw (0,1.4) -- (0.9,1.4);
				\draw (0,.7) -- (0.9,.7);
				\draw (0,0) -- (1.2,0);
				
				\draw (0.9,2.1) -- (1.2,0.7);
				\draw (0.9,1.4) -- (1.2,2.1);
				\draw (0.9,0.7) -- (1.2,1.4);
				\draw (0.8,0) -- (1,0);

				\draw (1.2,2.1) -- (2.5,2.1);
				\draw (1.2,1.4) -- (2.5,1.4);
				\draw (1.2,.7) -- (2.5,.7);
				\draw (1.2,0) -- (2.5,0);
				
				\node[opacity=1] at (0.9,2.1) {\scalebox{1.5}{$\times$}};
			\end{tikzpicture}}\\
			(3.a) del. mutation \\ on immune line 
		\end{center}
	\end{minipage}\begin{minipage}{.2\textwidth}
		\begin{center}
			\scalebox{0.75}{\begin{tikzpicture}
				\draw[opacity=0,line width=1.3mm] (0,0) -- (0,3.6);
				
				%%%%%%Immune line
				\draw[gray,line width=1mm,opacity=.7] (0,0.7) -- (0.75,0.7);
				\draw[gray,line width=1mm,opacity=.7] (0.7,0.7) -- (1,1.4);
				\draw[gray,line width=1mm,opacity=.7] (1,1.4) -- (2.5,1.4);
				%%%%%%%%%%%%%%%%%%%%%%%%%%
				
				\draw (0,1.4) -- (0.7,1.4);
				\draw (0,.7) -- (0.7,.7);
				\draw (0,0) -- (1,0);
				
				\draw (0.7,1.4) -- (1,2.1);
				\draw (0.7,0.7) -- (1,1.4);
				\draw (0.8,0) -- (1,0);

				\draw (1,2.1) -- (2.5,2.1);
				\draw (1,1.4) -- (2.5,1.4);
				\draw (.9,.7) -- (2.5,.7);
				\draw (1,0) -- (2.5,0);
				
				\node[opacity=1] at (.9,0.7) {\scalebox{1.5}{$\times$}};
			\end{tikzpicture}}\\
			(3.b) del. mutation \\ not on immune line 
		\end{center}
	\end{minipage}\begin{minipage}{.2\textwidth}
		\begin{center}
			\scalebox{0.75}{\begin{tikzpicture}
				\draw[opacity=0,line width=1.3mm] (0,0) -- (0,3.6);
				
				%%%%%%Immune line
				\draw[gray,line width=1mm,opacity=.7] (0,1.4) -- (1,1.4);
				\draw[gray,line width=1mm,opacity=.7] (1,0.7) -- (2.5,0.7);
				%%%%%%%%%%%%%%%%%%%%%%%%%%
				
				\draw (0,1.4) -- (1,1.4);
				\draw (0,.7) -- (1,.7);
				\draw (0,0) -- (1,0);

				\draw (1,2.8) -- (2.5,2.8);
				\draw (1,2.1) -- (2.5,2.1);
				\draw (1,1.4) -- (2.5,1.4);
				\draw (1,.7) -- (2.5,.7);
				\draw (1,0) -- (2.5,0);
				
				%\draw[dashed] (1,1.4) --(1,3.3);
				
				\draw (1,1.4)[opacity=1] circle (1.5mm)  [fill=white!100];
			\end{tikzpicture}}\\
			(4) ben. mutation \\ ~~
		\end{center}
	\end{minipage}
	\caption[Transitions pLD-ASG]{Transitions of the pLD-ASG. Lines are ordered according to their levels and the immune line is pictured in bold grey. Only  lines at finite levels are shown.  Note that the direction of the arrow in the coalescence event is meaningless (either direction of the arrow in the original particle system will lead to an upward arrow in the pLD-ASG). Note also that, in a selection event, the hollow arrowhead and the bullet indicating the corresponding tail are contracted at the level the arrow is targeting. Relocation events are invisible in the pLD-ASG.  }
	\label{fig:transitionpLDASG}
\end{figure}
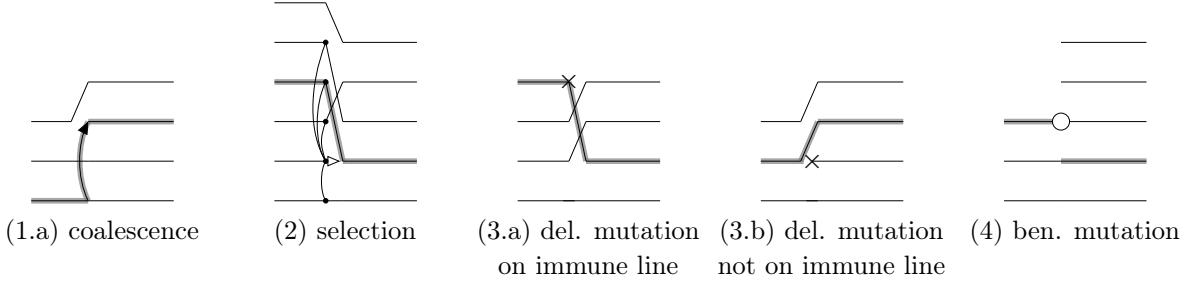

\smallskip

The pLD-ASG facilitates to identify the ancestral line for a given type assignment in a crucial way. More precisely, in Section~\ref{sec:ancestraltype}, we will establish in Proposition~\ref{prop:levelancestralsite}  that the ancestor of an individual is unfit at time $r$ if and only if all  lines at finite levels in its pLD-ASG are unfit at time $r$. This will lead to one of our main results. 

\begin{theorem}[Representation of ancestral type distribution] \label{thm:atyperepresentation}
	We have \begin{equation}
		h_r(k)=\E_1\Bigg[\frac{k^{\underline{L_{r} } }}{N^{\underline{L_{r} } }}\Bigg],\qquad k\in [N]_0.\label{eq:atypedist}
	\end{equation}
	Furthermore, $h_{\infty}(k)=\lim_{r\to\infty}h_r(k)$ exists and is given by \begin{equation}
		h_{\infty}(k)=\E \Bigg[\frac{k^{\underline{L_{\infty} } }}{N^{\underline{L_{\infty} } }}\Bigg],\qquad k\in [N]_0.\label{eq:catypedist}
	\end{equation}
	
\end{theorem}

There is also a complementary, forward-in-time representation that we now want to lay out. In the diffusion limit of the MoMo, the ancestral type distribution has been expressed as the absorption probability of a jump-diffusion process (\citet[Prop.~2.5]{Taylor2007}, see also \cite[Sec.~7]{LKBW15}). Our next result establishes an analogous representation in the finite-population setting. To this end, let $\widetilde{Y}$ be a process on $[N]_0$ that is coupled to~$Y$ on the basis of the same graphical representation until the first time~$\widetilde{Y}$ hits the boundary. More precisely, $\widetilde{Y}$ moves with~$Y$. Additionally, at every deleterious and beneficial type-changing mutation in the (typed) graphical representation that governs~$Y$,  perform a Bernoulli experiment with success parameter $1/(k+1)$ and $1/(N-k+1)$, respectively. In case of a success, $\widetilde{Y}$  jumps to~$0$ and~$1$, respectively, where it absorbs. In case of failure, $\widetilde{Y}$ continues to move with~$Y$ (if not yet absorbed at the boundary). The following result establishes the connection to the pLD-ASG.

\begin{theorem}[Factorial moment duality] \label{thm:finite.dual.mut}
	The processes~$\widetilde{Y}$ and~$L$ are dual with respect to the duality function~$H_F$ of Theorem~\ref{thm:duality_R}, that is, for~$t\geq0$,
	\begin{equation}
		\E_{k}\bigg[\frac{\faf{\widetilde{Y}_t}{n}}{\faf{N}{n}}\bigg]=\E_n\bigg[\frac{\faf{k}{L_t}}{\faf{N}{L_t}}\bigg]  \qquad \forall {k}\in[N]_0, \ n\in [N].\label{eq:dualitymarkovjumpchain}\end{equation}
\end{theorem}
\begin{corollary}[Forward-in-time representation of ancestral type distribution] \label{coro:jumpprocess}
	For~${k}\in [N]_0$, one has $h_t(k)=\E_k[\widetilde{Y}_t/N]$ and  
	\begin{equation}
		h_{\infty}({k})=\P(\lim_{t\to\infty}\widetilde{Y}_t=N\mid \widetilde{Y}_0={k}).\label{eq:characerizationhinffinite}
	\end{equation}
	In particular,~$h_{\infty}$ is the unique solution of the difference equation
	\begin{equation} \begin{split}
			\Bigg[2\bigg(&1-\frac{k}{N}\bigg)  +\frac{N-k}{{k}}u\nu_1+u\nu_0+\sum_{m=1}^\infty s_m \bigg(1-\bigg(\frac{k}{N}\bigg)^m\bigg)\Bigg]h_{\infty}({k}) \\
			& = \Bigg[1-\frac{k}{N}+\frac{N-k}{{k}+1}u\nu_1\Bigg]h_{\infty}({k}+1)+\Bigg[1-\frac{k}{N}+\frac{N-k}{N-{k}+1}u\nu_0+\sum_{m=1}^\infty s_m\bigg(1-\bigg(\frac{k}{N}\bigg)^m\bigg) \Bigg]h_{\infty}({k}-1) \\ & \hphantom{=} +\frac{u\nu_0}{N-{k}+1}
		\end{split}
		\label{eq:boundaryvalueproblem}
	\end{equation}
	for ${k}\in [N-1]$; complemented by~$h_{\infty}(0)=0$ and~$h_{\infty}(N)=1$.
\end{corollary}
Even though $\widetilde{Y}$ is constructed via the graphical representation, a biological interpretation does not seem straightforward. However, we are able to relate the mean of $\widetilde Y$ to the so-called \emph{descendant process} of \citep{Kluth2013}. More specifically, consider a sample taken at forward time~$0$ and let $D_t$ count their type-$1$ \emph{descendants} at time~$t$; analogously, let $B_t$ count the sample's type-$0$ descendants at time~$t$. The triple $(Y,B,D)=(Y_t,B_t,D_t)_{t\geq 0}$ then forms the descendant process. 
\begin{proposition}\label{prop:ytildedescendantprocess}
	Consider $\widetilde{Y}$ and the descendant process $(Y, D, B)$. Then, for all $t \geq 0$ and $k \in [N]_0$, 
	\begin{equation}\label{eq:ytildedescendantprocess}
		\E [D_t + B_t \mid Y_0 = D_0 = k, B_0 = 0]	= \E [\widetilde{Y}_t \mid \widetilde{Y}_0 = k].	
	\end{equation}
\end{proposition}
We prove Theorem~\ref{thm:finite.dual.mut}, its corollary, and Proposition~\ref{prop:ytildedescendantprocess} in Section~\ref{sec:absorbingmarkov}, where we also provide more detail on the definition of~$\widetilde{Y}$.

\subsection{Diffusion limit}\label{mr:difflimit}
The type frequency process of the MoMo converges to the Wright--Fisher diffusion if the population size tends to infinity, time is appropriately rescaled, and  mutation and selection are \emph{weak}. Most of our results in the finite population setting translate to that limit, as will be worked out in Section~\ref{sec:diffusionlimit}. We start out here by recalling the classic convergence result for the type-frequency process. To this end, let $u^{(N)}$, and $s_m^{(N)}$, $m\in \N$, be the mutation and selection rate in a population of size~$N$. We assume  \begin{equation}\label{eq:assweaklimits}
	N u^{(N)}\xrightarrow{N\to\infty}\theta, \qquad Ns_m^{(N)}\xrightarrow{N\to\infty} \sigma_m\quad (\forall m\in\N),\quad  \text{and}\quad \sum_{m>0} Ns_m^{(N)}\xrightarrow{N\to\infty}\sum_{m>0} \sigma_m,
\end{equation}
where $\theta\geq 0$ and $(\sigma_m)_{m=1}^\infty$ is a sequence in~$\R_+$ with $\sum_{m=1}^\infty \sigma_m m <\infty$. 

\smallskip

Denote by $\Ys = (\Ys_t)_{t\geq 0}$ the Wright-Fisher diffusion on $[0,1]$ with mutation and FTW selection, i.e. the process with generator \begin{equation}
	\mathcal{A}_\Ys= \mathcal{A}_{\Ys}^{\mathrm{n}} + \sum_{m>0} \mathcal{A}_{\Ys}^{\sigma_m} + \mathcal{A}_{\Ys}^{\nu_0} + \mathcal{A}_{\Ys}^{\nu_1} \label{eq:generatorWF}
\end{equation} acting on $f\in C^2([0,1])$, where 
\begin{align*}
	\mathcal{A}_{\Ys}^n f(y) &\defeq y(1-y) f''(y), &&\mathcal{A}_{\Ys}^{\sigma_m} f(y) \defeq -\sigma_m y(1-y^m) f'(y), \\
	\mathcal{A}_{\Ys}^{\nu_0}f(y) & \defeq -y\theta \nu_0 f'(y),  &&\mathcal{A}_{\Ys}^{\nu_1} f(y) \defeq (1-y) \theta \nu_1 f'(y).
\end{align*} 
Because the drift term is Lipschitz continuous, it follows from~\citep[Thm. 8.2.8]{Ku86} that the closure of $\mathcal{A}_{\Ys}$ generates a Feller semigroup on $C([0,1])$. The following result connects the Wright--Fisher diffusion of the present section with the MoMo with mutation and FTW.

\begin{proposition}[Convergence MoMo to Wright--Fisher diffusion]\label{prop:convergenceWF}
	For $N \in \mathbb{N}$, let $\bar{Y} \defeq (Y^{(N)}_{Nt}/N)_{t \geq 0}$, where  $Y^{(N)}$ is the Moran model with population size~$N$. Suppose $\lim_{N \to \infty} \bar{Y}^{(N)}_0 = \Ys_0$ in distribution. Then, $\bar{Y}^{(N)} \xRightarrow{N \to \infty} \Ys$ in distribution. 
\end{proposition}

The proof will be given in Section~\ref{sec:diffusionlimit}. There, 
we establish that the processes encoding the ancestral structures  also converge weakly. In particular, the duality between forward and backward process is preserved in the limit; only that moment dualities take the place of factorial moment dualities. 

\subsubsection{The kASG in the diffusion limit}
The following process is a natural candidate to be the limit of the kASG under assumption~\eqref{eq:assweaklimits}. Define $\mathcal{R} = (\mathcal{R}_r)_{r \geq 0}$ as the continuous-time Markov chain on $\N_{0,\Delta}\coloneqq\mathbb{N}_{0} \cup \{\Delta\}$ with generator $\mathcal{A}_{\mathcal{R}} = \mathcal{A}^{\rm{n}}_{\mathcal{R}} + \sum_{m>0} \mathcal{A}^{\sigma_m}_{\mathcal{R}} + \mathcal{A}^{\theta}_{\mathcal{R}}$, where 
\begin{equation}\label{eq:R} \begin{split}
		\mathcal{A}_{\Rs}^{\rm{n}} f(n) & \defeq n (n-1) [f (n-1) - f(n)],\\ 
		\mathcal{A}_{\Rs}^{\sigma_m}  f(n)  &\defeq n\sigma^{}_m [f (n+m) -  f(n)], \\
		\mathcal{A}_{\Rs}^{\theta}  f(n) & \defeq n\theta \nu^{}_0  [  f (\Delta) -  f(n)] + n \theta \nu_1 [f(n-1) - f(n)], 
	\end{split}
\end{equation}

and $f$ is a function on $\N_{0,\Delta}$ vanishing at infinity. Here, $\N_{0,\Delta}$ is to be equipped with the discrete topology; in particular, letting $\Delta$ sit in the spot of, for instance, $-1$, makes it a normed space with the usual norm and the notion of "vanishing at infinity" then corresponds to the usual one. It is not difficult to argue that~\eqref{eq:R} gives rise to a unique Markov process. The following result establishes that $\Rs$ is indeed the correct limit process.

\begin{proposition}\label{prop:convergencekASG}
	For $N\in \N$, let $\bar{R}^{(N)}\coloneqq (R^{(N)}_{Nr})_{r\geq 0}$, where $\bar{R}^{(N)}$ is the line-counting process of the kASG in a MoMo of size~$N$. Assume $\bar{R}^{(N)}_0\xrightarrow[N\to\infty]{(d)} \Rs_0$. Then, $\bar{R}^{(N)}\xRightarrow[]{N\to\infty} \Rs$ in distribution.
\end{proposition}

The factorial moment duality between $R^{(N)}$ and $Y^{(N)}$ turns into a moment duality in  the diffusion limit. This is the content of the next result, which we prove in Section~\ref{sec:kASGdiff} using that $R^{(N)}$ and $Y^{(N)}$ converge weakly as $N\to\infty$. 

\begin{theorem}[Moment duality]\label{thm:momentduality}
	Let $\Ys$ be the Wright--Fisher diffusion with mutation and FTW selection and $\Rs$ the line-counting process of the kASG in the diffusion limit. Then, for $y\in [0,1]$, $n\in \N_{0, \Delta}$, and $t\geq 0$,
	\begin{equation}\label{eq:momentdualityYR}
		\E[\Ys_t^{n} \mid \Ys_0 = y]=\E[y^{\Rs_t} \mid \Rs_0 = n],
	\end{equation}
	where $y^{\Delta} \defeq 0, \, \forall y$. That is, $\Ys$ and $\Rs$ are dual w.r.t. the duality function $\mathcal{H} (y, n) \defeq y^n.$
\end{theorem}
This moment duality is an extension of the case without selection and mutation,  where $\Rs$ is the block-counting process of Kingman's coalescent, which is the moment dual of $\Ys$; see, for example, \cite[Thm.~2.7]{B09}. 

\smallskip

As in the MoMo, two different regimes appear for the long-term behaviour of the Wright--Fisher diffusion and the ancestral process depending on the presence of mutation. On the one hand, if $\theta = 0$, then one of the two types dies out, so $\Ys$ is absorbing; and~$\Rs$ is positive recurrent and hence converges to a unique stationary distribution, which we denote by~$\pi_{\Rs}$. 
On the other hand, if $\theta > 0$, then~$\Ys$ converges to a unique stationary distribution~$\pi_{\Ys}$, while $\Rs$ is absorbing. (Let us note in passing that, as an easy consequence of \cite[Eq.~(2)]{Taylor2007}, $\pi_{\Ys}$ has density 
\begin{equation}\label{piY}
   \pi_{\Ys}(y) = C \exp \Big ( -\sum_{m > 0} \sigma_m \sum_{k=1}^m \frac{y^k}{k!} \Big ) (1-y)^{\theta \nu_0-1} y^{\theta \nu_1 -1}
\end{equation}
on $[0,1]$, where $C$ is a normalising constant.) 

The moment duality leads to a representation of the absorption probabilities of~$\Ys$ in terms of the stationary distribution of~$\Rs$ and vice versa; thus providing the analogue of Corollary~\ref{cor:repr.absorpt.prob.finite} in the large population setting. Denote by $\Ys_\infty$ and $\Rs_\infty$ a random variable on~$[0,1]$ and $\N_{0,\Delta}$, respectively, with distribution $\pi_\Ys$ and $\pi_\Rs$. 

\begin{corollary}[Representation absorption probabilities]\label{cor:repr.absorpt.prob.diff}
	Assume $\theta=0$. For $y \in [0,1]$, $$\P(\lim_{t\to\infty} \Ys_t=1\mid \Ys_0=y)=\E [ y^{\Rs_\infty} ].$$
	Assume $\theta>0$. For $n\in \N_{0,\Delta}  $, $$\P(\lim_{t\to\infty} \Rs_t=0\mid \Rs_0=n)=\E [\Ys_\infty^n ],$$
	with the convention that $y^{\Delta} \defeq 0$ for every $y$.
\end{corollary}

The proofs of Propositions~\ref{prop:convergenceWF} and \ref{prop:convergencekASG}, Theorem~\ref{thm:momentduality}, as well as Corollary \ref{cor:repr.absorpt.prob.diff} may be found in Section~\ref{sec:kASGdiff}.

\subsubsection{The pLD-ASG in the diffusion limit}\label{mr:pldASGdifflimit}
To derive the ancestral type distribution in the diffusion limit, we determine the limiting behaviour of $L$ and then combine it with the representation of the ancestral type distribution in Theorem~\ref{thm:atyperepresentation}. First, we establish the limiting process.

\smallskip

The limit candidate is the continuous-time Markov chain $\Ls=(\Ls_r)_{r\geq 0}$ on $\N$ with generator 
\[
\mathcal{A}_{\Ls}  = \mathcal{A}_{\Rs}^{\rm{n}} + \sum_{m>0} \mathcal{A}_{\Rs}^{\sigma_m} + \mathcal{A}_{\Ls}^{\nu_0}+\mathcal{A}_{\Ls}^{\nu_1},\]
where $\mathcal{A}_{\Rs}^{\rm{n}}$ and $\mathcal{A}_{\Rs}^{\sigma_m}$ are defined in~\eqref{eq:R}, whereas 
\[
	\mathcal{A}_{\Ls}^{\nu_0}  f(n)  \defeq \theta \nu^{}_0  \sum_{j=1}^{n-1}[  f (j) -  f(n)] \quad \text{and } \; \mathcal{A}_{\Ls}^{\nu_1}  f(n)  \defeq (n-1) \theta \nu^{}_1  [  f (n-1) -  f(n)]
\]
for~$f:\N\to\R$ vanishing at infinity. \begin{proposition}\label{prop:convergencepLD}
	For $N\in \N$, let $\bar{L}^{(N)}\coloneqq (L^{(N)}_{Nt})_{t\geq 0}$, where $\bar{L}^{(N)}$ is the line-counting process of the pLD-ASG in a MoMo of size~$N$. Assume $\bar{L}^{(N)}_0\xrightarrow[N\to\infty]{(d)} \Ls_0$. Then,  $\bar{L}^{(N)}\xRightarrow[]{N\to\infty} \Ls$ in distribution.
\end{proposition}

The proof will be given in Section~\ref{sec:pldASGdifflimit}.

\smallskip

If $\theta=0$, then~$\Ls$ and $\Rs$ agree and admit a unique stationary distribution. If $\theta>0$, the return time to~$1$ is dominated by an exponential random variable with parameter~$\theta\nu_0$ and $\Ls$ is then       positive recurrent. In both cases we denote the stationary distribution of $\Ls$ by~$\pi_{\Ls}$.

\smallskip

In the diffusion limit, the probability that  the ancestral type is $1$ at backward time~$r$, given the type-$1$ frequency at that time is~$y$, is \begin{equation}
	\mathfrak{h}_r(y)\coloneqq\E[y^{\Ls_r}]\label{eq:diffancestraltyp}.\end{equation}
Moreover, the common ancestor type conditional on the type-$1$ frequency~$y\in [0,1]$ is Bernoulli distributed with parameter
\begin{equation}
	\mathfrak{h}_{\infty}(y)\coloneqq \lim_{r\to\infty}\mathfrak{h}_r(y),\label{eq:diffcomancestraltyp}
\end{equation}
which is well defined because $\Ls$ converges to its unique stationary distribution~$\pi_{\Ls}$. Both definitions are motivated by Theorem~\ref{thm:atyperepresentation} and the convergence of~$L^{(N)}$ to~$\Ls$.
\smallskip 

Finally, we provide an alternative representation in terms of the diffusion-limit analogue of $\widetilde{Y}$ from Section~\ref{mr:atd}. Consider $\widetilde{\Ys}_t$ on $[0,1]$ with generator $\As_{\widetilde{\Ys}}f(y)=\As_{\Ys}f(y)+\As_{\widetilde{\Ys}}^{\mathrm{j}}f(y)$, where $\As_{\Ys}$ is given in~\eqref{eq:generatorWF} and \begin{equation}\label{eq:tildeYlimitgen}
	\As_{\widetilde{\Ys}}^{\mathrm{j}}f(y)=\frac{1-y}{y}\theta\nu_1 [f(0)-f(y)]+\frac{y}{1-y} \theta \nu_0[f(1)-f(y)] 
\end{equation}
with domain $\Ds(\As_{\widetilde{\Ys}})=\{f\in \Cs^2([0,1]):\lim_{y\to 1} \As_{\widetilde{\Ys}}f(y)=\lim_{y\to 0}\As_{\widetilde{\Ys}}f(y)=0\}$.
That~$\As_{\widetilde{\Ys}}$ indeed generates a well-defined Markov semigroup was proved in \citep[Eq. (11) ff.]{Taylor2007}. This process follows a Wright--Fisher diffusion with mutation and selection until a random time, when it jumps to one of the boundary points, where it is absorbed. The crucial feature is that the jump rates diverge at the boundary, leading to a jump to $\{0,1\}$ before the process can diffusively access~$\{0,1\}$. Moreover, one can show that~$\widetilde{\Ys}$ arises as the large population limit of $\widetilde{Y}^{(N)}$ (where the superscript indicates again the population size) in the diffusion limit setting of~\eqref{eq:assweaklimits}. We sketch the key steps of the proof that establishes existence of~$\widetilde{\Ys}$ and convergence of $\widetilde{Y}^{(N)}\to \widetilde{\Ys}$ as $N\to\infty$ in Section~\ref{sec:pldASGdifflimit}.

\smallskip 
The factorial moment duality between~$L^{(N)}$ and $\widetilde{Y}^{(N)}$ translates to the diffusion limit as a moment duality. This strengthens the result of \citet[Eq.~(11)]{Taylor2007} insofar as we establish a connection between the jump diffusion and an ancestral process.
\begin{theorem}[Moment duality]\label{thm:tildemomentduality}
	Let $\widetilde{\Ys}$ be the Markov process corresponding to $\As_{\widetilde{\Ys}}$ and $\Ls$ the line-counting process of the pLD-ASG in the diffusion limit. Then, for $y\in [0,1]$, $n\in \N$, and $t\geq 0$,
	\begin{equation}
		\E_y[\widetilde{\Ys}_t^{n}]=\E_n[y^{\Ls_t}].
	\end{equation}
	That is, $\Ys$ and $\Rs$ are dual w.r.t. the duality function $\mathcal{H} (y, n)= y^n$.
\end{theorem}
The proof is based on generator arguments, see Section~\ref{sec:pldASGdifflimit}.
The result leads to a representation of the ancestral type distribution in terms of $\widetilde{\Ys}$. Moreover, the common ancestor type distribution admits a representation as a hitting probability, which was derived in~\citep[Prop.~2.5]{Taylor2007} for more general forms of selection.
\begin{corollary}\label{coro:hittingprobab}
	For $y\in [0,1]$, $\mathfrak{h}_r(y)=\E[\widetilde{\Ys}_r]$. Moreover, $$\mathfrak{h}_\infty(y)=\P(\widetilde{T}_1<\widetilde{T}_0\mid \widetilde{\Ys}_0=y),$$ where $\widetilde{T}_i\coloneqq \inf\{r\geq 0:\widetilde{\Ys}_r=i\}$ for $i \in \{0,1\}$.
\end{corollary}

\subsection{Applications}
\subsubsection{Fixation probability under moderate selection, and the expected number of potential influencers}
Recently, the regime of moderate selection, where the strength of selection scales with $1/N^\alpha$ for some $0<\alpha<1$, has received increased attention. After all, it covers a large range of possible scalings, in contrast to the diffusion limit, which requires scaling with precisely $1/N$. Specifically,  the classical  \emph{Haldane's formula} for the fixation probability of a single beneficial mutant  has been extended to the MoMo and the class of \emph{Cannings models}, both  under  moderate genic selection \cite{Boenkostetal21,Boenkostetal22}. We  now generalise this MoMo result to FTW (and hence DOM) selection.

\begin{proposition}[Haldane's formula]\label{prop:haldane}
In the FTW model with $u=0$, $s_m=\sigma/N^\alpha$ for some $m>0$, $\sigma>0$, and $0<\alpha<1$, as well as $s_j=0$ for all $ j \neq m$, the fixation probability of a single fit individual in an otherwise unfit population is 
\[
  \P(\lim_{t\to\infty} Y_t = 0 \mid Y_0=N-1)  =  \frac{m \sigma}{N^\alpha} \big (1+\scO(1) \big ) \quad \text{as} \; N \to \infty.
\] 
\end{proposition}
This will be proved in Section~\ref{sec:applications}, but let us  give a heuristic argument here. The above parameter choice   in the FTW model translates into the DOM model via $\widehat s_i = s_m$ for all $i \leq m$ and $\widehat s_i = 0$ otherwise. As long as the proportion of the fit type is small, all individuals in a random  sample of finite size are unfit with high probability. Hence,  whenever a   fit individual under DOM selection encounters a selective arrow, the lines at the tails of all checking arrows as well as the line at the tip of the selective arrow will be unfit with high probability. Likewise, the rate at which fit individuals will be killed via selective arrows is negligible. We may therefore initially approximate the number of fit individuals (forward in time) by  a slightly supercritical branching process in which every fit individual splits into two at rate $1+\widehat s_1 + \ldots + \widehat s_m = 1+m \widehat s_1$ and dies at rate 1. The resulting offspring expectation and offspring variance are   $\mu = 1+m \widehat s_1/2 + \scO(\widehat s_1)$ and $\sigma^2 = 1+ \scO(\widehat s_1)$, respectively,  as $\widehat s_1 \to 0$. This results in the survival probability of $2(\mu-1)/\sigma^2 = m \widehat s_1+ \scO(\widehat s_1)$ \cite[Theorem~5.5]{Haccou_05}. The classical approximation of the fixation probability by the  survival probability (see  \cite{Boenkostetal21,Boenkostetal22} and references therein) then yields the claim. 

\smallskip

Specialising the first statement in Corollary \ref{cor:repr.absorpt.prob.finite} to $k=N-1$,
noting that $(N-1)^{\underline{R_\infty}}/N^{\underline{R_\infty}} = (N-R_\infty)/N$, and then combining with Proposition~\ref{prop:haldane}, we get
\begin{corollary}\label{coro:statlines}
  Under the conditions of Proposition~\ref{prop:haldane}, 
  \[
    \E[R_\infty] = m \sigma N^{1-\alpha}  \big (1+\scO(1) \big ).
  \]
\end{corollary}
For $m=1$, this reduces to the result of \cite[Sec.~2.4]{Boenkostetal21}  for single branching events at rate $\sigma/N^\alpha$. The corollary then tells us that, for $m$-fold branching, $\E[R_\infty]$ is very similar to the same quantity under single branching events at rate $m \sigma/N^\alpha$. 

\subsubsection{Stationary and ancestral type distributions under FTW: some observations}

To conclude the main-results section, we compare the behaviour of the frequency-dependent selection scheme with the simpler genic case. We do this in a somewhat informal way, which also relies on numerical observations. 

\smallskip

We focus on the mean proportion of unfit individuals and  unfit ancestors at stationarity under different parameter regimes. The first observation is that, morally speaking, the higher the order of selection, the greater its strength in the following sense. Fix a parameter $s > 0$ and $m_1,m_2\in \N$ with $m_1< m_2$. Consider the number of unfit individuals $Y^{(1)}$ and $Y^{(2)}$ in two FTW models with the same parameters $u$, $\nu_0$ (and thus $\nu_1$) and identical initial values, but with selection rates $s^{(1)}_m = \ind_{m_1}(m)s $ and $s^{(2)}_m = \ind_{m_2}(m)s$, respectively. Then the two processes can be coupled by replacing every selection event $E_{J,i}$ of $Y^{(1)}$ with $E_{\tilde{J}, i}$ of $Y^{(2)}$, where $\tilde{J}_j = J_j$ for all $j \leq m_1$ and  $\tilde{J}_{m_1+1}, \ldots, \tilde{J}_{m_2}$  are chosen uniformly at random in $[N]$. It is then clear that $Y^{(2)}$ is stochastically dominated by~$Y^{(1)}$: having more potential parents increases the probability that there is at least one fit individual in the sample, thus skewing the distribution towards the fit type. The same applies for the ancestral type. This observation is in line with Fig.~6 of \cite{Wiehe95}, which illustrates the stationary distribution in a diploid model with and without dominance (compare also item (2) at the beginning of Section~\ref{sec:models}).

\smallskip

A more meaningful comparison between different parameter regimes is obtained by fixing the  \emph{effective branching rate} $b \defeq \sum_{m > 0} ms_m$,  as analogously defined in \citep[Def.~2.18]{cordero2019general}. In Figure~\ref{fig:plots} we have chosen for comparison a model with pure genic selection (blue, solid line), one without a genic-selection component (green, dotted line) and a mixture of the two situations (orange, dashed line); all with a  small $\nu_0$ and weak selection and mutation,  a regime close to the diffusion limit. We use the ratio $u/b$  as the independent parameter in both graphs, thus generalising $u/s_1$,  the relevant quantity in the case of genic selection.  In the law-of-large-numbers regime  of the MoMo  and as $\nu_0 \to 0$,  one has the well-known  \emph{error threshold} phenomenon (see \cite{Eigenetal88,baake2018lines} and references therein). This is a threshold for the mutation rate at which the fit type goes extinct due to mutation, irrespective of  its initial frequency; a behaviour that still persists approximately for large but finite $N$ and small $\nu_0$. Figure~\ref{fig:plots} now suggests that the phenomenon also survives in some capacity in the frequency-dependent setting.
\begin{figure}[t]
        \begin{center}
		\scalebox{0.6}{\includegraphics[width=1\textwidth]{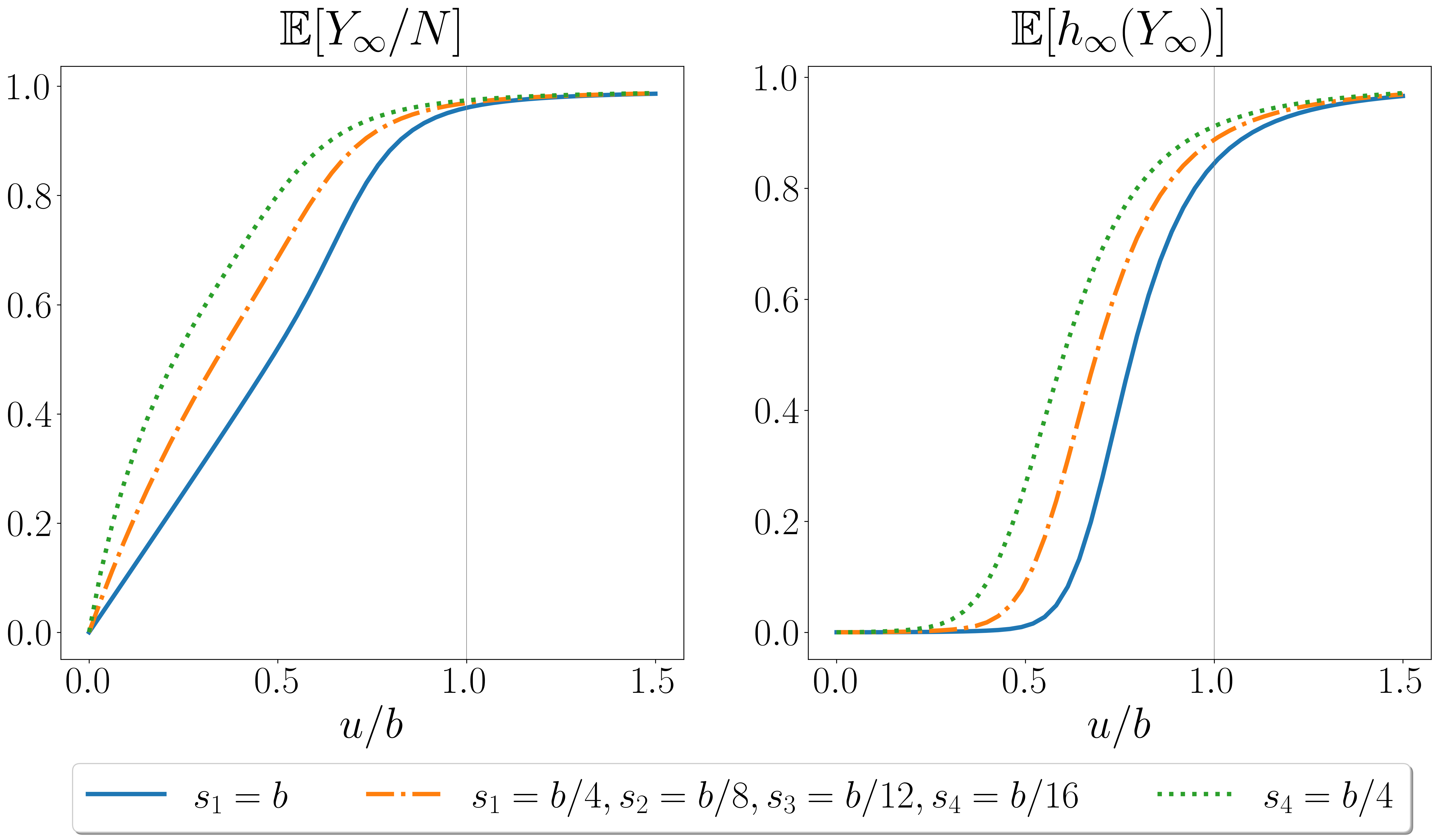}	}
	\end{center}
	\caption{The expected stationary proportions of unfit individuals (left) and of unfit ancestors (right) for $b = 0.005$ as a function of the mutation rate, for $N = 10000$ and $\nu_0 = 0.005$. The grey vertical line indicates $u=b$, which corresponds to the error threshold in the law-of-large-numbers regime in the case of genic selection as $\nu_0\to 0$~\citep[Rem.~8]{BCH18jmath}. }
	\label{fig:plots}
\end{figure}
More precisely, for fixed  $b$, the strength of selection now \emph{decreases} with the order of frequency dependence (whereas, in the moderate-selection setting of  Proposition~\ref{prop:haldane} and Corollary~\ref{coro:statlines}, the order of the frequency dependence only has a  negligible effect for fixed effective branching rate).
The diffusion approximation delivers a heuristic for this observation. Consider the selective part of the drift term in the generator of~$\Ys$ (see \eqref{eq:generatorWF}), whose modulus is a measure of the strength of selection. For $0<y<1$, we have
\begin{equation*}
\sum_{m>0} \sigma_m y(1 - y^m) = \sum_{m>0} m \sigma_m y(1-y)\frac{1}{m}\sum_{i=0}^{m-1} y^{i} \leq y(1-y)\sum_{m>0} m\sigma_m \eqdef  \beta y(1-y)
\end{equation*}
with equality if and only if $\sigma_i=0$ for all $i>1$, that is, in the case of genic selection, where the effective branching rate is $\beta=\sigma_1$. 
This small calculation shows that, for fixed $\beta$, the genic case (that is, $\sigma_m = \beta \ind_{1}(m)$) maximises the selective drift. Whether  formal results can be obtained in the diffusion limit and also in the finite setting is  territory for future explorations.

%%%%%%%%%%%%%%%%%%%%%%%%%%%%%%%%%%%%%%%%%%%%%%%%%%%%%%%%%%%%%%%%%%%%%%%%%%%%%%%%%%%%%%%%%%%%%%%%%%%%%%%%%%%%%%%%%%%%%%%%%%%%%%%%%%%%%%%%%%%%%%%%%%%%%%%%%%%%%
%%%%%%%%%%%%%%%%%%%%%%%%%%%%%%%%%%%%%%%%%%%%%%%%%%%%%%%%%%%%%%%%Section2%%%%%%%%%%%%%%%%%%%%%%%%%%%%%%%%%%%%%%%%%%%%%%%%%%%%%%%%%%%%%%%%%%%%%%%%%%%%%%%%%%%%%
%%%%%%%%%%%%%%%%%%%%%%%%%%%%%%%%%%%%%%%%%%%%%%%%%%%%%%%%%%%%%%%%%%%%%%%%%%%%%%%%%%%%%%%%%%%u%%%%%%%%%%%%%%%%%%%%%%%%%%%%%%%%%%%%%%%%%%%%%%%%%%%%%%%%%%%%%%%
\section{Details of the two selection models}\label{sec:models}
In this section, we detail the construction of the two MoMo's. To start off, we discuss the biological settings covered by the DOM model. Next, we formalise the graphical representation in terms of Poisson processes and explain the propagation of types and ancestry. Finally, we prove the equivalence in distribution of the two models under assumption~\eqref{eq:non-increasing}.

\smallskip

The DOM selection scheme covers a variety of biologically relevant situations. 
	\begin{enumerate}
		\item \label{s1}
		For $\widehat s_1>0$ and $\widehat s_m = 0$ for all $m > 1$,  the reproduction rate of a type-0 individual is  independent of the remaining population. This is the aforementioned case of \emph{genic selection}, as opposed to  \emph{frequency-dependent selection}, which applies to situations (\ref{s2}) and (\ref{sall}) below.
              \item \label{s2}
                For $\widehat s_1 \geqslant 0$, $\widehat s_2 > 0$, and $\widehat s_m = 0$ for $m>2$,   the selective advantage of a type-0 individual has a contribution determined by the type of a randomly chosen partner. Up to small terms that vanish in the diffusion limit, this has an alternative interpretation in terms of \emph{diploid selection}. Here, one identifies diploid individuals  with their genotypes, where a genotype is an unordered pair~$\{ i,j\} \in \{0,1\}^2$ of gametes, which are combined independently (by slight abuse of notation, we use  the set notation here for unordered pairs even if $i=j$). Working at the level of the gametes, one assumes that type 0 when combined with another 0 has a selective reproduction rate of $2 \widehat s_1 + \widehat s_2$;  type 0 combined with type 1, as well as type 1 combined with type 0,  has a selective reproduction rate of  $\widehat s_1 + \widehat s_2$; and type~1 combined with type~1 has no seletive reproduction rate. This means that  type~$0$ is (partially) dominant\footnote{Note that the term \emph{balancing selection} used in \cite{casanova2018duality} instead of dominance is misleading; in fact, balancing selection means that the genotype $\{ 0,1\}$ is superior to both $\{ 0,0\}$ and $\{ 1,1\}$, which is excluded by the positivity of the parameters; see, for example, \cite[p.~64]{Du08}.}, that is, it can also (partly) play out its advantage if paired with a type~$1$; but, unlike in our DOM model, this happens in a symmetric way, in that type~1 also profits from the interaction.  We also speak of \emph{linear dominance} since  the reproduction rate of any given type depends linearly on $k/N$.  In a population with $k$ type-1 gametes, we therefore have selective transitions to $k-1$ at rate $(2 \widehat s_1 + \widehat s_2) k(N-k)^2/N^2  + k^2(N-k)/N^2$ and to $k-1$ at rate $(\widehat s_1 + \widehat s_2) k(N-k)^2/N^2$.
With the methods used in Section~\ref{sec:diffusionlimit}, it is easily seen that the resulting process has the same diffusion limit as $\widehat{Y}$.  		
		\item \label{sall}
		For any other choice of the $\widehat s_m$, the selective advantage of a type-0 individual depends on  two or more randomly chosen partners to be of type 1. This  may come from  ecological interactions between individuals, as opposed to the purely genetic interactions caused by diploid genotypes.  Since this generalises case (\ref{s2}) in that the reproduction rate now contains nonlinear terms in $k/N$, we speak of \emph{nonlinear dominance}.
	\end{enumerate}

\subsection{Graphical construction with type and ancestry propagation}
The construction of the graphical representation for the two MoMos requires the following independent families of independent homogeneous Poisson point processes. Recall the parameters from the model description in Section~\ref{sec:mainresult} (in particular, $(\widehat{s}_m)_{m>0}$ need not be non-increasing here) and consider the Poisson point processes on the real line \begin{equation}\label{eq:poissonprocesses}
\Pi_{j,i}^{\narrow} \text{ at rate } \frac{1}{N},\quad  \widehat{\Pi}_{J,i}^{\sarrow} \text{ at rate } \frac{ \widehat{s}^{}_{l(J)}}{N^{l(J)}}, \quad  \Pi_{J,i}^{\sarrow} \text{ at rate } \frac{ s^{}_{l(J)}}{N^{l(J)}}, \quad   \Pi_i^{\circ} \text{ at rate } u\nu_0, \quad  \Pi_i^{\times} \text{ at rate } u\nu_1,
\end{equation}
where $i,j\in [N]$, $J\in \bigcup_{m>0} [N]^m$, and $l(J)$ denotes the length of the tuple $J$ (of course, almost surely, no point belongs to more than one family).
Collect these families into the set $\widehat \Lambda\coloneqq \{ \Pi_{j,i}^{\narrow}\cup \widehat \Pi_{J,i}^{\sarrow}\cup  \Pi_i^{\circ}\cup \Pi_i^{\times}: i,j\in [N],\, J\in \bigcup_{m>0} [N]^m  \}$; analogously, let $\Lambda$ be defined as $\widehat \Lambda$ but with $\widehat{\Pi}_{J,i}^{\sarrow}$ replaced by $\Pi_{J,i}^{\sarrow}$. These Poisson point processes  deliver the graphical elements encoding neutral arrows (from line $i$ to $j$), checking and selective arrows (DOM),  selective arrows (FTW), beneficial mutations (on line~$i$), and deleterious mutations (on line~$i$). More precisely, in the DOM model, for example, a point in $\widehat \Pi_{J,i}^{\sarrow}$ means there is a selective arrow with tip pointing to line $i$ and tail $J_1$ accompanied by checking arrows whose tails are the lines in the set  $\{J_2, \ldots, J_m\}$. Similarly, in the FTW model, a point in $\Pi_{J,i}^{\sarrow}$ means that the continuing line is~$i$ and it receives the (joint) tip of all arrows emanating from the lines in $\breve{J}\coloneqq\{J_1, \ldots, J_m\}$ (the set of incoming lines).

\medskip

The DOM model  will be proved to be equivalent in distribution to the FTW model under assumption \eqref{eq:non-increasing}. Therefore we formalise the type and ancestry propagation only in the FTW model; defining the analogue for the DOM model is straightforward. For $\tau\in \R$, we write $f_{\tau+}\defeq\lim_{h\searrow 0} f_{\tau+h}$ and $f_{\tau-}\defeq\lim_{h\searrow 0} f_{\tau-h}$ for any function $f$ on $\R$. Moreover, for $a\in \R$ and $B\subseteq \R$, $B\pm a \coloneqq\{a\pm b: b\in B\} \eqqcolon a\pm B$.

\begin{definition}[Types and ancestry in the FTW MoMo]\label{def:typingMoMo}
	Let $\Lambda$ be the family of Poisson point processes in~\eqref{eq:poissonprocesses}. We call $\alpha:[N]\to \{0,1\}, \ i\mapsto \alpha(i)$, a \emph{site colouring}. Given $\Lambda$, $\alpha$ and $t_0\in \R$, we define \begin{enumerate}
		\item[(i)] the \emph{typed MoMo} $\Ce^{t_0}=(\Ce^{t_0}_t)_{t \geqslant 0}$ with site colouring~$\alpha$ at time~$t_0$, where $\Ce^{t_0}_t = (\Ce^{t_0}_t(i))_{i \in [N]}$ and $\Ce^{t_0}_t(i) \in \{0,1\}$ is the type of site $i$ at time $t_0+t$ if $\Ce^{t_0}_{0}=\alpha$, and
		\item[(ii)] the \emph{MoMo ancestry} $\Ae^{t_0}=(\Ae^{t_0}_t)_{t \geqslant 0}$ with site colouring~$\alpha$ at time~$t_0$, where $\Ae^{t_0}_t = (\Ae^{t_0}_t(i))_{i \in [N]}$ and $\Ae^{t_0}_t(i) \in [N]$ is the ancestral site at time~$t_0$ of the individual occupying site $i$ at time $t_0+t$ if $\Ce^{t_0}_{0}=\alpha$.
	\end{enumerate}  
We construct types and ancestries as follows. Start with $\Ce^{t_0}_{0} = \alpha$ and set $\Ae^{t_0}_{0}(i)=i$ for $i\in[N]$. Proceed inductively for the arrival times  $\tau \in (\Lambda-t_0) \cap \R_{+}$ in increasing order, in the following way. In between arrival times of~$\Lambda$, nothing happens. 
	\begin{enumerate}
		\item If $\tau\in (\Pi_{j,i}^{\narrow}-t_0)$ for some $i,j\in [N]$, then $\Ce^{t_0}_{\tau+}(i)=\Ce^{t_0}_{\tau}(j)$, and for $\eta\in [N]\setminus\{i\}$, $\Ce^{t_0}_{\tau+}(\eta)=\Ce^{t_0}_{\tau}(\eta)$. Similarly, $\Ae^{t_0}_{\tau+}(i)=\Ae^{t_0}_{\tau}(j)$, and for $\eta\in [N]\setminus\{i\}$, $\Ae^{t_0}_{\tau+}(\eta)=\Ae^{t_0}_{\tau}(\eta)$.
		\item If $\tau\in (\Pi_{J,i}^{\sarrow}-t_0)$ for some $i \in [N]$ and $J=(J_{\ell})_{\ell=1}^m\in[N]^m$, $m\in \N$, then \begin{itemize}
			\item if $\exists \ell\in \breve{J}$ with $\Ce^{t_0}_{\tau}(\ell)=0$, set $\Ce^{t_0}_{\tau+}(i)=0$, $\Ae^{t_0}_{\tau+}(i)=\Ae^{t_0}_{\tau}(\min\{k\in[m]:\Ce^{t_0}_{\tau}(J_k)=0\})$, and for $\eta\in [N]\setminus\{i\}$, $\Ce^{t_0}_{\tau+}(\eta)=\Ce^{t_0}_{\tau}(\eta)$ and $\Ae^{t_0}_{\tau+}(\eta)=\Ae^{t_0}_{\tau}(\eta)$;
			\item if  $\Ce_{\tau}(\ell)=1 \forall \ell\in \breve{J}$, set $\Ce^{t_0}_{\tau+}(\eta)=\Ce^{t_0}_{\tau}(\eta)$ and $\Ae^{t_0}_{\tau+}(\eta)=\Ae^{t_0}_{\tau}(\eta)$ for all $\eta\in [N]$.
		\end{itemize}
		\item If $\tau\in (\Pi_i^{\star}-t_0)$ for some $i \in [N]$ and $\star\in \{\times,\circ\}$, then $\Ce^{t_0}_{\tau+}(i)=\1\{\star=\times\}$ and for $\eta\in [N]\setminus\{i\}$, $\Ce^{t_0}_{\tau+}(\eta)=\Ce^{t_0}_{\tau}(\eta)$. Moreover, for all $\eta\in [N]$, $\Ae^{t_0}_{\tau+}(\eta)=\Ae^{t_0}_{\tau}(\eta)$.
	\end{enumerate}

\end{definition}
We omit the superscript if it is $0$, that is, $\Ce^{}_t \defeq \Ce^{0}_t$ and $\Ae^{}_t \defeq \Ae^{0}_t$. Clearly, (1), (2), and (3)  formalise the idea of type propagation and the notion of ancestry in a neutral reproduction, selective reproduction, and mutation event, respectively. We have defined $\Ce$ and $\Ae$ as l\`{a}dc\`{a}g for consistency with the ancestral processes considered below. In the forward perspective, we will for notational convenience usually choose $t_0=0$.

\smallskip

Let $Y_t\coloneqq \sum_{i=1}^N\Ce_t(i)$ be the number of unfit individuals at time $t$. Then $Y=({Y}_t)_{t \geqslant 0}$ is the birth-death process with the generator $\cA_Y$ as defined in~\eqref{eq:generatorhatY}. In particular, the type of an individual randomly chosen at time~$t$ has Bernoulli distribution with parameter~$Y_t/N$.

\subsection{Connecting nonlinear dominance with fittest-type-wins}
For the remainder of the manuscript, we assume~\eqref{eq:non-increasing} is in place. Then Lemma~\ref{lem:DF} states that the two selection models are equivalent in distribution. We now provide the proof.
\begin{proof}[Proof of Lemma~\ref{lem:DF}]
Since the generators of $\widehat Y$ and $Y$   agree except for the terms involving selective reproduction (cf.~\eqref{eq:generatorhatY}), we only consider the latter. Fix $\widehat Y_t = Y_t=k\in[N]_0$. For $m\in \N$, let $\widehat p_m $ ($p_m$) be the probability in the DOM (FTW)  model that, given a selective event of order~$m$ occurs, an offspring is produced (which is then of type 0 by construction). Given $\widehat Y_t = Y_t=k$, these probabilities are independent of the type of the individual that will be replaced (that is, the one on the continuing line). By construction, we have
\[
\textstyle \widehat p_m = \frac{N-k}{N} \big ( \frac{k}{N} \big )^{m-1}   \quad \text{and}\quad   p_m = 1- \big ( \frac{k}{N} \big )^m.
\]
Hence, 
\begin{equation}\label{eq:decomp}
\textstyle \sum_{n=1}^m \widehat p_n = \frac{N-k}{N} \sum_{n=1}^m \big ( \frac{k}{N} \big )^{n-1} = 1- \big ( \frac{k}{N} \big )^m  = p_m
\end{equation}
and so $\widehat p_1 = p_1$ and $\widehat p_m = p_m - p_{m-1}$  for $m>1$.
As a consequence, the reproduction rate of fit individuals via selective events is
\[\textstyle
\sum_{m >0} \widehat s_m \widehat p_m = \widehat s_1 p_1 + \sum_{m >1} \widehat s_m (p_m - p_{m-1}) = \sum_{m>0} (\widehat s_m -  \widehat s_{m+1}) p_m = \sum_{m>0} s_m p_m
\]
under the stated choice of the $s_m$, which entails the identity of the selective death rates in \eqref{eq:generatorhatY}.
 \end{proof}

Lemma~\ref{lem:DF} can be read in two ways:

\begin{enumerate}
\item For any given time horizon $t$, a realisation of the DOM model may be obtained from a realisation of the FTW model by replacing every event $E_{J,i}$ by  events $\widehat E_{(J_1),i},  \widehat E_{(J_2,J_1),i}, \ldots, \widehat E_{(J_ {l(J)}, J_{l(J)-1},\ldots, J_1),i}$ (so each incoming line $J_k$ in an FTW event introduces a $k$th-order DOM event with incoming line~$J_k$),  each occurring independently at a time chosen uniformly in $[t_0, t_0+t]$. Indeed, for $J\in[N]^m$,  $\widehat{E}_{J,i}$-events then occur at rate $$\sum_{n\geq m}\sum_{\tilde{J}\in [N]^n: \atop \forall \eta\in [m]:\, \tilde{J}_\eta= J_\eta} \frac{s_n}{N^n}=\frac{1}{N^m}\sum_{n\geq m} s_n.$$
Each FTW event of order $m$ is thus decomposed into a family of DOM events of orders $1, \ldots, m$.
This yields the relation $\widehat s_m = \sum_{n \geqslant m} s_n$, in agreement with the assumption that $(\widehat s_m)_{m >0}$ is a nonincreasing sequence. 
\item On the other hand, a realisation of the FTW model may be obtained from a realisation of the DOM model satisfying \eqref{eq:non-increasing}  via thinning as follows.
Whenever an event $\widehat E_{J,i}$ occurs in DOM, replace it either by the event $E_{J,i}$ or $E_\varnothing$ (a silent event where nothing happens) with probability $\PP(E_{J,i}\mid \widehat E_{J,i}) := (\widehat s_{l(J)} - \widehat s_{l(J) +1})/\widehat s_{l(J)}$ and $\PP(E_{\varnothing} \mid \widehat E_{J,i}) = \widehat s_{l(J) +1}/\widehat s_{l(J)}$, respectively. In the so-constructed FTW model, $E_{J,i}$ events occur at rate $(\widehat s_{l(J)} - \widehat s_{l(J) +1})/N^{l(J)}$, and therefore selective events of order $m$ occur to every line $i$ at rate $s_m = \widehat s_m - \widehat s_{m+1}$. 
\end{enumerate}
 
\begin{remark}
\begin{enumerate}
	\item \citet[Cor.~2.12]{cordero2019general} describe a rather general class of selection models. In their context, a condition similar to~\eqref{eq:non-increasing} implies the existence of a moment duality, which is also ASG-based. Ancestral structures beyond this case seem to be inextricably harder to treat. 
	
	\item Let us compare our model parameters with~\cite{casanova2018duality} assuming $s_m=\widehat{s}_m-\widehat{s}_{m-1}$ for $m>0$. The counterpart to $\sum_{m>0}  s_m$  is the
	selection intensity $\kappa$ in~\citep{casanova2018duality}, $s_m/(\sum_j s_j)$ corresponds to the probability $\pi_m$ of~$m$ potential parents in a selective (FTW) event; and $\widehat s_m/(\sum_j s_j)$ is the corresponding tail probability, once more in agreement with our assumption  \eqref{eq:non-increasing}. So, $\widehat s_m/(\sum_j s_j)$ is the coefficient of $x^{m-1}$ in the power series of the selection function $s(x)$ in \cite{casanova2018duality}. The connection between the $\pi_m$ and the coefficients in the power series of $s(x)$ as tail probabilities becomes transparent via our Lemma~\ref{lem:DF}.
\end{enumerate}
\end{remark}

\section{Construction of the ancestral selection graph with multiple branching}\label{sec:ASG}
In this section, we formalise the ASG and the notions of the distributions of type, ancestral type, and common ancestor type. Because of Lemma~\ref{lem:DF}, it suffices to focus on the FTW model.

\smallskip

The ASG can be formalised in various ways (e.g. directed acyclic graphs~\citep[Sect.~4]{cordero2019general}). We encode it as a continuous-time Markov chain $G^{\Ts}=(G_r^{\Ts})_{r \geq 0}$ on $\Ps([N])$, the power set of $[N]$, where~$G_r^{\Ts}$ is the set of sites in the graph at  time~$r$ before~$\Ts$. Recall that in principle, we distinguish between the \emph{sites} (the labels of the lines in the interacting particle system) and the influencer \emph{lines}  (which may move between sites); nevertheless, for the sake of readability, we will sometimes speak of `line~$i$' instead of `the line at site~$i$' when we think there is no risk of confusion. 

\smallskip

To construct the process $G^{\Ts}$, we again rely on the Poisson point processes in \eqref{eq:poissonprocesses}, see Fig.~\ref{fig:ASG2}. 
More precisely, we now consider the arrival times backward in time. We write $\varrho$ instead of $\tau$ for the time points, i.e. $\varrho\in (\Ts-\Lambda)\cap \R_{+}$.  Since $G^{\Ts}$ is constant between jumps, it suffices to define what happens at the jump times.

\smallskip

\begin{definition}[ASG]\label{def:ASG}
Fix~$\Ts\in \R$, $g \subseteq [N]$, and  a realisation of~$\Lambda$ in~\eqref{eq:poissonprocesses}. Let $G_0^{\Ts}=g$ be the (set of sites of) the \emph{initial sample} taken at time $\Ts$. For $r> 0$, $G_r^{\Ts}\subseteq[N]$ is the set of sites occupied by the potential influencers  of the lines in $g$ at  time $r$ before $\Ts$.
To construct $G^{\Ts}=(G_r^{\Ts})_{r\geq0}$, we proceed inductively for the arrival times  $\varrho \in (\Ts-\Lambda) \cap \R_{+}$ in increasing order in the following way. In between arrival times of~$\Lambda$, $G^{\Ts}$ does not change. 

\begin{enumerate}
	\item If $\varrho\in (\Ts-\Pi_{j,i}^{\narrow})$ for some $i \in G_{\varrho-}^{\Ts}$ and $j \in [N]$, then $G_\varrho^{\Ts} = (G_{\varrho-}^{\Ts}\setminus \{i\})\cup\{j\}$.
	\item If $\varrho\in (\Ts-\Pi_{J,i}^{\sarrow})$ for some $i \in G_{\varrho-}^{\Ts}$ and $J \in [N]^m$ for some $m\in \N$, then $G_\varrho^{\Ts} = G_{\varrho-}^{\Ts}\cup \breve J$. 
	\item In all other cases (that is, if $\varrho\in (\Ts-\Pi_i^{\star})$ for $\star\in \{\times,\circ\}$ and some $i \in [N]$, or if $\varrho\in (\Ts-\Pi_{j,i}^{\narrow})\cup (\Ts-\Pi_{J,i}^{\sarrow})$ for some $j$ or $J$, but $i\notin G_{\varrho-}^{\Ts}$), $G_\varrho^{\Ts} = G_{\varrho-}^{\Ts}$.
\end{enumerate}
We refer to the triple $(G^{\Ts},g,\Lambda)$ as the ASG with initial sample $g$ taken at time~$\Ts$. In what follows, we frequently consider an ASG in a finite time horizon. We then write $(G^{\Ts}_{[0,r]},g,\Lambda)$ for the restriction of $(G^{\Ts},g,\Lambda)$ to $[\Ts-r,\Ts]$ (or equivalently, until backward time~$r$).
\end{definition} 
We again omit the upper index if it is $0$, that is, $G^{}_r \defeq G^{0}_r$. Unless specified otherwise, in what follows we assume $t_0=0$ ($\Ts=0$) when we consider the forward (backward) process individually. When we consider them jointly, we assume $t_0<\Ts$.

\begin{figure}
	\scalebox{0.75}{\begin{tikzpicture}
		
	%Frame
	\draw[dashed] (0,-0.5) --(0,4.25);
	\draw[dashed] (14.5,-0.5) --(14.5,4.25);    
	\node[above] at (0,4.25) {$\Ts-t_0$};  
	\node[above] at (14.5, 4.25) {$0$};  
	%	\node [right] at (0,-0.5) {$r$};
	%	\node [right] at (14.5,-0.5) {$0$};
	\draw[-{angle 60[scale=5]}] (9.25,4.7 ) -- (5.25,4.7) node[text=black, pos=.5, yshift=6pt]{$r$};

	%G_0
	\node[gray] at (14.5,-0.8) {\scalebox{1}{$g$}};
	\fill[gray,opacity=.6]  (14.35,3.85) rectangle (14.65,4.15);
	
	%G_1
	\node[gray] at (12.9,-0.8) {\scalebox{1}{$G_{\varrho^{}_1}$}};
	\draw[dashed] (12.9,-0.5) --(12.9,4.25);   
	\node [above] at (12.9,4.25) {$\varrho^{}_1$};
	
	\fill[gray,opacity=.6]  (12.75,3.85) rectangle (13.05,4.15);
	\fill[gray,opacity=.6]  (12.75,0.85) rectangle (13.05,1.15);

	%G_2
	\node[gray] at (8.5,-0.8) {\scalebox{1}{$G_{\varrho^{}_2}$}};
	\draw[dashed] (8.5,-0.5) --(8.5,4.25);   
	\node [above] at (8.5,4.25) {$\varrho^{}_2$};
	
	\fill[gray,opacity=.6]  (8.35,3.85) rectangle (8.65,4.15);
	\fill[gray,opacity=.6]  (8.35,2.85) rectangle (8.65,3.15);
	\fill[gray,opacity=.6]  (8.35,1.85) rectangle (8.65,2.15);
	\fill[gray,opacity=.6]  (8.35,0.85) rectangle (8.65,1.15);

	%G_3
	\node[gray] at (8,-0.8) {\scalebox{1}{$G_{\varrho^{}_3}$}};
	\draw[dashed] (8,-0.5) --(8,4.25);   
	\node [above] at (8,4.25) {$\varrho^{}_3$};
	
	\fill[gray,opacity=.6]  (7.85,3.85) rectangle (8.15,4.15);
	\fill[gray,opacity=.6]  (7.85,2.85) rectangle (8.15,3.15);
	\fill[gray,opacity=.6]  (7.85,1.85) rectangle (8.15,2.15);
	\fill[gray,opacity=.6]  (7.85,0.85) rectangle (8.15,1.15);
	
	%G_4
	\node[gray] at (7.5,-0.8) {\scalebox{1}{$G_{\varrho^{}_4}$}};
	\draw[dashed] (7.5,-0.5) --(7.5,4.25);   
	\node [above] at (7.5,4.25) {$\varrho^{}_4$};
	
	\fill[gray,opacity=.6]  (7.35,3.85) rectangle (7.65,4.15);
	\fill[gray,opacity=.6]  (7.35,2.85) rectangle (7.65,3.15);
	\fill[gray,opacity=.6]  (7.35,1.85) rectangle (7.65,2.15);
	\fill[gray,opacity=.6]  (7.35,0.85) rectangle (7.65,1.15);
	
	%G_5
	\node[gray] at (6.2,-0.8) {\scalebox{1}{$G_{\varrho^{}_5}$}};
	\draw[dashed] (6.2,-0.5) --(6.2,4.25);   
	\node [above] at (6.2,4.25) {$\varrho^{}_5$};
	
	\fill[gray,opacity=.6]  (6.05,3.85) rectangle (6.35,4.15);
	\fill[gray,opacity=.6]  (6.05,2.85) rectangle (6.35,3.15);
	\fill[gray,opacity=.6]  (6.05,1.85) rectangle (6.35,2.15);
	\fill[gray,opacity=.6]  (6.05,0.85) rectangle (6.35,1.15);
	
	%G_6
	\node[gray] at (3.6,-0.8) {\scalebox{1}{$G_{\varrho^{}_6}$}};
	\draw[dashed] (3.6,-0.5) --(3.6,4.25);   
	\node [above] at (3.6,4.25) {$\varrho^{}_6$};
	
	\fill[gray,opacity=.6]  (3.45,3.85) rectangle (3.75,4.15);
	\fill[gray,opacity=.6]  (3.45,2.85) rectangle (3.75,3.15);
	\fill[gray,opacity=.6]  (3.45,0.85) rectangle (3.75,1.15);
	
	%G_7
	\node[gray] at (2.9,-0.8) {\scalebox{1}{$G_{\varrho^{}_7}$}};
	\draw[dashed] (2.9,-0.5) --(2.9,4.25);   
	\node [above] at (2.9,4.25) {$\varrho^{}_7$};
	
	\fill[gray,opacity=.6]  (2.75,3.85) rectangle (3.05,4.15);
	\fill[gray,opacity=.6]  (2.75,2.85) rectangle (3.05,3.15);
	\fill[gray,opacity=.6]  (2.75,0.85) rectangle (3.05,1.15);
	\fill[gray,opacity=.6]  (2.75,-0.15) rectangle (3.05,0.15);
	
	%G_8
	\node[gray] at (2.2,-0.8) {\scalebox{1}{$G_{\varrho^{}_8}$}};
	\draw[dashed] (2.2,-0.5) --(2.2,4.25);   
	\node [above] at (2.2,4.25) {$\varrho^{}_8$};
	
	\fill[gray,opacity=.6]  (2.05,3.85) rectangle (2.35,4.15);
	\fill[gray,opacity=.6]  (2.05,2.85) rectangle (2.35,3.15);
	\fill[gray,opacity=.6]  (2.05,0.85) rectangle (2.35,1.15);
	\fill[gray,opacity=.6]  (2.05,-0.15) rectangle (2.35,0.15);
	
	%G_9
	\node[gray] at (1,-0.8) {\scalebox{1}{$G_{\varrho^{}_9}$}};
	\draw[dashed] (1,-0.5) --(1,4.25);   
	\node [above] at (1,4.25) {$\varrho^{}_{9}$};
	
	\fill[gray,opacity=.6]  (0.85,3.85) rectangle (1.15,4.15);
	\fill[gray,opacity=.6]  (0.85,2.85) rectangle (1.15,3.15);
	\fill[gray,opacity=.6]  (0.85,-0.15) rectangle (1.15,0.15);

	%%%%%%%%%%%%%TYPING
	starting distribution
	\fill [lightbrown] (-.15,-.15) rectangle (0.15,0.15);
	\fill [lightbrown] (-.15,2.85) rectangle (0.15,3.15);
	\fill [lightbrown] (-.15,3.85) rectangle (0.15,4.15);
	
	%5 line
	\draw[lightbrown,loosely dashdotted, opacity=1,line width=0.9mm] (0,4)--(2.9,4);
	\draw[darkgreen,solid, opacity=1,line width=0.9mm] (2.9,4)--(8,4);
	\draw[lightbrown,solid,opacity=1,line width=0.9mm] (8,4)--(12.9,4);
	\draw[darkgreen,solid, opacity=1,line width=0.9mm] (12.9,4)--(14.5,4);

	%4 line
	\draw[lightbrown,densely dotted,opacity=1,line width=0.9mm] (0,3)--(8.5,3);

	%3 line
	\draw[lightbrown,densely dotted,opacity=1,line width=0.9mm] (3.6,2)--(7.5,2);
	\draw[darkgreen,densely dotted,opacity=1,line width=0.9mm] (7.5,2)--(8.5,2);

	%2 line
	
	\draw[lightbrown,densely dotted,opacity=1,line width=0.9mm] (1,1)--(6.2,1);	\draw[darkgreen,solid, opacity=1,line width=0.9mm] (6.2,1)--(12.9,1);

	%1 line
	\draw[lightbrown,solid, opacity=1,line width=0.9mm] (0,0)--(2.2,0);
	\draw[darkgreen,solid,opacity=1,line width=0.9mm] (2.2,0)--(2.9,0);

	%neutral arrows
	\draw[-{triangle 45[scale=5]},semithick,opacity=1] (3.6,1) -- (3.6,2);
	\draw[-{triangle 45[scale=5]},semithick,opacity=1] (1,3) -- (1,1);
	
	%selective arrows
	\draw[-{open triangle 45[scale=5]},semithick,opacity=1] (12.9,1) -- (12.9,4);
	%\draw[-{open triangle 45[scale=5]},semithick,opacity=1] (2.9,0) -- (2.9,4);
	\draw[-{open triangle 45[scale=5]},semithick,opacity=1] (6.2,4) .. controls (6,2.5) and (6,2.4) .. (6.2,1);
	%\draw[-{open triangle 45[scale=5]},semithick,opacity=1] (6.2,4) -- (6.2,1);
	\draw[-{open triangle 45[scale=5]},semithick,opacity=1] (2.9,0) .. controls (2.7,2.5) and (2.7,1.5) .. (2.9,4);

	%interactive arrows 
	\draw[] (8.5,3) circle (.6mm)  [fill=black!100];
	\draw[] (8.5,2) circle (.6mm)  [fill=black!100];
	\draw[-{open triangle 45[scale=5]},opacity=1,semithick] (8.5,3) -- (8.5,4);
	\draw[opacity=1,semithick] (8.5,2) -- (8.5,3);

	%Mutations deleterious
	\node[opacity=1] at (8,4) {\scalebox{1.5}{$\times$}};
	
	%Mutations beneficial
	\draw (2.2,0)[opacity=1] circle (1.5mm)  [fill=white!100];    
	%\draw (10.4,1)[opacity=1] circle (1.5mm)  [fill=white!100];
	\draw (7.5,2)[opacity=1] circle (1.5mm)  [fill=white!100];
	\end{tikzpicture}}
	
	\caption{The ASG from Fig.~\ref{fig:ASG} typed according to $\alpha\equiv 1$ (and with $\Ts=0$). The line colour indicates the type, the line style (dash-dotted, dotted, solid) indicates the ancestral site. $\varrho^{}_1, \ldots, \varrho^{}_{9}$ are the jump times of the process, and $G_{\varrho^{}_1}, \ldots, G_{\varrho^{}_{9}}$ are the sites of the ASG at each jump (indicated by grey squares). }
	\label{fig:ASG2}
\end{figure}
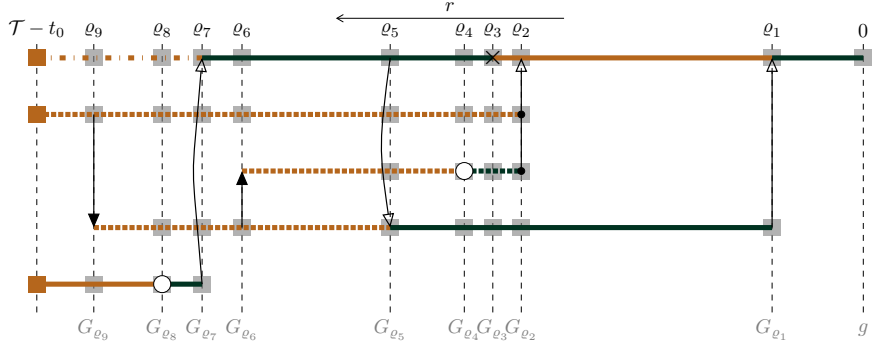

\medskip

For fixed~$r\geq 0$,  a site colouring for the MoMo leads to a \emph{typed ASG}. To this end, attach the types to each line of the ASG at backward time $r$ according to the site colouring, and then propagate the types through the graph in the forward direction of time (i.e., for decreasing $r$), see again Fig.~\ref{fig:ASG2}. The typing mechanism is the same as in the FTW MoMo, but restricted to the lines in $(G_{[0,r]},g,\Lambda)$. That is, at a neutral reproduction event, the offspring inherits the type of the parent; at a selective reproduction, the offspring is type~$1$ if and only if all potential parents are type~$1$; and at a deleterious (beneficial) mutation event the type on the line is~$1$ (resp.~$0$) after the mutation. Lines not affected by an event keep their types. Since we have now taken the backward perspective, the sites are coloured at~$t_0=-r$. In particular, for $(G_{[0,r]},g,\Lambda)$ and a site colouring~$\alpha$, $\Ce^{-r}_{0}(i)=\alpha(i)$ for all $i\in G_r$. The notion of ancestry also translates naturally  to $(G_{[0,r]},g,\Lambda)$. To this end, set $\Ae_{0}^{-r}(i)=i$ for all $i\in G_r$, and then propagate the ancestral sites as in the MoMo. We note that if we fix $g,r,\alpha$,  then for $v\in [0,r)$ and $i\in G_v$, $\Ce^{-r}_{r-v}(i)$ and $\Ae^{-r}_{r-v}(i)$ are not measurable with respect to $\sigma(\Lambda \cap [-v,0])$, but they are with respect to~$\sigma(\Lambda \cap [-r,-v])$. Moreover, only the restriction of a site colouring to the lines in $G_r$ enters the typed ASG.

\smallskip 

The type distribution of the ancestors of individuals alive at time $r=0$ will now be defined in a way amenable to an analysis via the ASG. 

\begin{definition}[Ancestral type distribution, common ancestor type distribution]\label{def:ancestraltypedistribution}
	Let $\Gamma=(\Gamma_i)_{i \in [N]}$ be a random variable that is independent of~$\Lambda$ and uniformly distributed on the site colourings, i.e. for $\alpha:[N]\to\{0,1\}$, $\P(\Gamma=\alpha)=2^{-N}$. Set $\lvert \Gamma\rvert\defeq\sum_{i\in [N]} \Gamma(i)$.
	\begin{itemize}
		\item The conditional \emph{ancestral type} at backward time~$r>0$ given $\lvert \Gamma \rvert =k\in[N]_0$ is Bernoulli distributed with parameter $$h_r(k)\defeq\P(\Ce_{0}^{-r}(\Ae_{r}^{-r}(1))=1\mid \Ce_{0}^{-r}=\Gamma, \lvert \Gamma \rvert =k).$$
		\item The (conditional) type of the \emph{common ancestor} given $\lvert \Gamma \rvert =k\in[N]_0$ has Bernoulli distribution with parameter $h_{\infty}(k)$, where  $h_{\infty}(k)\defeq\lim_{r\to\infty} h_r(k)$,	if the limit exists.
	\end{itemize}
\end{definition}
\begin{remark}
The distributions in the above definition depends on $\Gamma$ only via $\lvert \Gamma \rvert$. We choose to consider (the ancestry of) the individual occupying site~$1$; but since we work under an exchangeable type assignment, we could have chosen any other individual. Furthermore, we also have  $h_r(k)=\P(\Ce_{0}^{}(\Ae_{r}^{}(1))=1\mid \Ce_{0}^{}=\Gamma, \lvert \Gamma \rvert =k)$ due to time homogeneity.
\end{remark}

\section{Killed ASG with multiple branching and its applications}\label{sec:kASG}
In this section, we first formalise the kASG for the FTW model and derive the rates of the associated (generalised) line-counting process. We then prove the factorial moment duality and its applications. We start by making the intuition appealed to in Section~\ref{subsec:kASG} precise.
\begin{definition} \label{DefkilledASG} (killed ASG with multiple branching).
Fix $g \subseteq [N]\cup\{\Delta\}$ and a realisation of~$\Lambda$ defined in~\eqref{eq:poissonprocesses}. Let $K_0=g$ be the (set of sites of) the initial sample (taken at backward time $r=0$). The killed ASG is then the process $K=(K_r)_{r \geq 0}$ on $\Ps([N])\cup \{\Delta\}$ constructed as follows. If $K_0\in \{\varnothing,\Delta\}$, then $K_r=K_0$ for all $r>0$. If $K_0\notin\{\varnothing,\Delta\}$, we construct $K$ inductively for the arrival times  $\varrho \in (-\Lambda \cap \R_{+})$ in increasing order according to the following rules. In between arrival times of~$\Lambda$, $K$ does not change. 

\begin{enumerate}
 \item[(0)] If $\varrho\in -(\Pi_{j,i}^{\narrow}\cup \Pi_{J,i}^{\sarrow}\cup\Pi_i^{\times}\cup\Pi_i^{\circ})$ with $i\notin K_{\varrho-}$ (and some $j,J$), then $K_\varrho = K_{\varrho-}$.
	\item[(1)] If $\varrho\in -\Pi_{j,i}^{\narrow}$ for some $i \in K_{\varrho-}$ and $j \in [N]$, then $K_\varrho = (K_{\varrho-}\setminus \{i\})\cup\{j\}$.
	\item[(2)] If $\varrho\in -\Pi_{J,i}^{\sarrow}$ for some $i \in K_{\varrho-}$ and $J \in [N]^m$ for some $m\in \N$, then $K_\varrho = K_{\varrho-}\cup \breve J$. 
	\item[(3)] If $\varrho\in -\Pi_i^{\times}$ for  some $i \in [N]$, then $K_\varrho = K_{\varrho-} \setminus \{i\}$; if $\varrho\in -\Pi_i^{\circ}$ for  some $i \in [N]$, then $K_\varrho = \Delta$.
\end{enumerate}
The triple $(K,g,\Lambda)$ is the kASG with initial sample $g$.
\end{definition}

The associated line-counting process plays a major role in the analysis. Recall that $R_r\coloneqq \lvert K_r\rvert,$ where $\lvert \Delta \rvert\coloneqq \Delta$. $R$ (like the kASG) is  c\`{a}dl\`{a}g in the positive direction of $r$. Whenever we write $R_0=[N]_{0,\Delta}$,  it is assumed that the initial sample is uniformly chosen among all sets in $K_0\in\Ps([N])\cup\{\Delta\}$ such that $\lvert K_0\rvert =n$.
Next, we provide the proof for Proposition~\ref{prop:lckASG}, that is, we determine the infinitesimal generator of~$R$.

\begin{proof}[Proof of Proposition~\ref{prop:lckASG}]
If $s_m=0$ for all $m>0$, the rates are clear. Hence, it suffices to check that~\eqref{eq:generatorRselection} corresponds to the rate at which the number of lines in~$K$ increases. Assume there are currently~$n$ lines in the kASG. Each of them is hit independently by the tip of a selective
arrow of order ~$m$ at rate~$s_m$. For the number of lines to increase by~$j\in[m]$, we have to place the $m$ distinguishable marks on $N$ distinguishable sites such that exactly $j$ out of $N-n$ sites currently not in the graph receive at least one mark. There are $(N-n)^{\underline{j}} C^n_{mj}$ such possibilities. To see this, observe that there are $(N-n)^{\underline{j}}$ possibilities to choose $j$ out of $N-n$ sites not in the graph. For each of these possibilities, we must place some number $\ell\in \{j,\ldots,m\}$ of $m$ marks on these $j$ sites, and the remaining $m-\ell$ ones on the~$n$ lines in the graph. We therefore sum over all possibilities to select~$\ell$ out of~$m$ marks; for every such $\ell$, there are $\binom{m}{\ell}$ such ways. For every such possibility, in turn, there are $\stirlingii{\ell}{j}$ ways to partition~$\ell$ marks into the~$j$ selected  sites. For the remaining $m-\ell$ marks, there are $n^{m-\ell}$ ways to place them on~$n$ lines in the graph. Each configuration has probability $N^{-m}$. Hence, we obtain~\eqref{eq:generatorRselection}.
\end{proof}

Next, we prove the factorial moment duality between $Y$ and~$R$.

\begin{proof}[Proof of Theorem~\ref{thm:duality_R}]
We want to apply~\citep[Prop. 1.2]{Jaku}. Recall that $\mathcal{A}_Y$ and $\mathcal{A}_R$ denote the generators of $Y$ and $R$, respectively. Since the state space of $Y$ is finite, every function is in the domain of  $\mathcal{A}_Y$. In particular, $H(\cdot, n)(k)$ and $P_t^Y H(\cdot, n)(k)$
lie in the domain of  $\mathcal{A}_Y$, where $(P_t^Y)_{t\geqslant 0}$ is the transition semigroup corresponding to $Y$. Similarly,
$H(k, \cdot)(n)$ and $P_t^R H(k, \cdot)(n)$ lie in the domain of $\mathcal{A}_R$, where  $(P_t^R)_{t\geqslant 0}$ is the transition
semigroup corresponding to~$R$. Furthermore, $H_F$ is obviously bounded. Thus, \citep[Prop. 1.2]{Jaku} provides us with a necessary and sufficient condition for  duality, namely
\begin{equation}
\mathcal{A}_Y H_F(\cdot, n)(k) = \mathcal{A}_R H_F(k, \cdot)(n) \quad \forall k \in [N]_0, \quad n \in [N]_{0,\Delta},
\label{eq:fmdualitygen}
\end{equation}
which we now verify.

\smallskip

Recall that $\mathcal{A}_Y = \mathcal{A}_Y^{\rm{n}} + \sum_{m>0} \mathcal{A}_Y^{s_m} + \mathcal{A}_Y^{u},$ and $\mathcal{A}_R  = \mathcal{A}_R^{\rm{n}} + \sum_{m>0} \mathcal{A}_R^{s_m} + \mathcal{A}_R^{u},$
with the building blocks defined in~\eqref{eq:generatorhatY} and \eqref{eq:generatorRcoalescence}--\eqref{eq:generatorRmutation}. It will turn out that the duality relation  holds pairwise for each of these parts. The case is clear for $n=\Delta$  since $\mathcal{A}_{Y} H_F(\cdot, \Delta)(k) = 0 = \mathcal{A}_R H_F(k, \cdot)(\Delta)$ for $k \in [N]_0$. For $n \in [N]_0$ and $k \in [N]_0$, we get for the neutral part\[
\begin{split}
\cA_Y^{\rm{n}} H_F ( \cdot,n)(k) & = k \frac{N-k}{N} \big ( [H_F(k+1,n) - H_F(k,n)] + [H_F(k-1,n) - H_F(k,n)]  \big )\\
& = k \frac{N-k}{N} \Big ( \frac{(k+1)^{\underline{n}} - k^{\underline{n}}}{N^{\underline{n}}} + \frac{(k-1)^{\underline{n}} - k^{\underline{n}}}{N^{\underline{n}}} \Big ) \\
& = k \frac{N-n+1-(k-n+1)}{N^{\underline{n}}} n \frac{n-1}{N} (k-1)^{\underline{n-2}} \\
& = \frac{n(n-1)}{N} [H_F(k,n-1)-H_F(k,n)] = \cA_R^{\rm{n}} H_F(k,\cdot)(n);
\end{split}
\]
for the mutation part \[
\mathcal{A}_Y^{u} H_F  ( \cdot,n)(k)  = (N-k) nu \nu^{}_1 \frac{k^{\underline{n-1}}}{N^{\underline{n}}}  - k n u \nu^{}_0 \frac{(k-1)^{\underline{n-1}}}{N^{\underline{n}}} = \mathcal{A}_R^{u} H_F ( k, \cdot)(n),
\]
since $(k+1)^{\underline{n}}-k^{\underline{n}} = n k^{^{\underline{n-1}}}$ and $\frac{k^{\underline{n-1}}}{N^{\underline{n-1}}} - \frac{k^{\underline{n}}}{N^{\underline{n}}}=(N-k) \frac{k^{\underline{n-1}}}{N^{\underline{n}}}$;
and
for   the selective part of order $m$
\[
\mathcal{A}_{Y}^{s_m} H_F(\cdot, n)(k) 
= s_m k \frac{N^m - k^m}{N^m} \frac{1}{N^{\underline{n}}} [ (k-1)^{\underline{n}} - k^{\underline{n}} ] = s_m \frac{n}{N^m} \frac{k^{\underline{n}}}{N^{\underline{n}}} (k^m-N^m),
\]
since $ (k-1)^{\underline{n}} - k^{\underline{n}} = k^{\underline{n}} (\frac{k-n}{k}-1) = - k^{\underline{n}} \frac{n}{k}$.
On the other hand,
\[
\mathcal{A}_R^{s_m} H(k, \cdot)(n) 
= s_m \frac{n}{N^m}  \sum_{j=1}^{m}  (N-n)^{\underline{j}} C^n_{mj} \Big[ \frac{k^{\underline{n+j}}}{N^{\underline{n+j}}} - \frac{k^{\underline{n}}}{N^{\underline{n}}}\Big] 
= s_m \frac{n}{N^m}  \frac{k^{\underline{n}}}{N^{\underline{n}}} \sum_{j=1}^m C^n_{mj} [ (k-n)^{\underline{j}} - (N-n)^{\underline{j}}],
\]
where we have used in the last step that $y^{\underline{n+j}}/y^{\underline{n}} = (y-n)^{\underline{j}}$. It remains to show that the sum on the right-hand side  equals $ k^m - N^m$. Changing summation, using the identity $x^\ell = \sum_{j=0}^\ell \stirlingii{\ell}{j} x^{\underline{j}}$ \citep[Prop.~3.24]{Aigner79} and the fact that $\stirlingii{0}{j}=0$ for all $j>0$  gives \begin{align}
\sum_{j=1}^m C^n_{mj} (k-n)^{\underline{j}} &= \sum_{j=1}^m \sum_{\ell=j}^m \binom{m}{\ell} \stirlingii{\ell}{j} n^{m-\ell} (k-n)^{\underline{j}}  
= \sum_{\ell=1}^m \binom{m}{\ell} n^{m-\ell} \sum_{j=1}^\ell  \stirlingii{\ell}{j} (k-n)^{\underline{j}} \nonumber \\
&= \sum_{\ell=1}^m \binom{m}{\ell} n^{m-\ell} (k- n)^\ell = (n + k -n)^m - n^m= k^m- n^m	.\label{eq:aux2}
\end{align}

Since the same holds for $k$ replaced by $N$, the result follows suit.   \end{proof}

\begin{remark}
	Theorem~\ref{thm:duality_R} also extends to the diffusion limit and the law of large numbers (e.g. \cite[Props.~1 and~2]{baake2018lines}  and \cite[Thm.~2]{BCH18jmath}  for genic selection,  and \cite[Cor.~2.12]{cordero2019general} for FTW selection in a diffusion limit but without mutation). Predecessors of the idea go back to \cite{Shiga1986}. Recently, \citet[Thm.~3.1]{Boenkostetal21} established  a factorial moment duality for a Cannings model with selection.
\end{remark}

\section{Proofs related to Siegmund duality}\label{sec:siegmund}
This section contains the proofs of Lemma~\ref{lem:bdduality} and Theorem~\ref{thm:factsieg}. 

\begin{proof}[Proof of Lemma~\ref{lem:bdduality}]
	Because of the finite state spaces of both $X$ and $X^S$, once again \cite[Prop 1.2]{Jaku} tells us it is enough to verify a relation between the generators analogue to \eqref{eq:fmdualitygen}. Name $\mathcal{A}$ and $\mathcal{A}^{S}$ the generator of $X$ and $X^{S}$, respectively. We have
	\begin{align*}
		\mathcal{A}H_S(\cdot, x^\ast) (x) &= \lambda_x [\ind{\{x+1 \geq x^\ast\}} - \ind{\{x \geq x^\ast\}}] + \mu_x [\ind{\{x-1 \geq x^\ast\}} - \ind{\{x \geq x^\ast\}}]  \\
		&= \lambda_x \ind{\{x = x^\ast-1\}} - \mu_x \ind{\{x = x^\ast\}} 	\end{align*} 
	for all $x,x^\ast$; as for $\mathcal{A}^{S}$, consider the following cases: For $x^\ast = x$,
	\[
	\mathcal{A}^{S}H_S(x, \cdot) (x^\ast) = \mu_x [\ind{\{x \geq x+1\}} - \ind{\{x \geq x \}}] + \lambda_{x-1} [\ind{\{x \geq x-1 \}} - \ind{\{x \geq x \}}] = - \mu_x.
	\]
	For $x^\ast = x+1$,
	\[
	\mathcal{A}^{S}H_S(x, \cdot) (x^\ast) = \mu_{x+1} [\ind{\{x \geq x+2\}}- \ind{\{x \geq x+1 \}}] + \lambda_x [\ind{\{x \geq x\}} - \ind{\{x \geq x+1\}}] = \lambda_x.
	\]
	For $x^\ast < x$ or $x^\ast > x+1$, it is trivially checked that $\mathcal{A}^{S}H_S(x, \cdot) (x^\ast) = 0$. These calculations also remain true for the edge cases under the convention that $\lambda_{-1} =   \mu_{N+1} \defeq 0$. 
\end{proof}

To formalise the connection between the factorial moment and the Siegmund duality, we work on the integer numbers, and in order to do so, we identify $\Delta$ with $N+1$, therefore replacing $H_F(k, \Delta) \defeq 0$ by $H_F(k, N+1) \defeq 0$. The linear transformation $T$ also  translates naturally to this setting by using~\eqref{eq:T}, but replacing $T(\Delta,\Delta)\defeq1$ by $T(N+1,N+1)\defeq1$.

\begin{proof}[Proof of Theorem~\ref{thm:factsieg}]
	With a slight abuse of notation, we identify $H_S$, $H_F$ with their matrix representation. In particular, $H_S,H_F\in \R^{(N+1)\times (N+2)}$.  We claim \begin{eqnarray}
		H_S T^{\trans} = H_F, \label{eq:relationSiFa}
	\end{eqnarray} where superscript~$\trans$ indicates a transposition. To see \eqref{eq:relationSiFa}, note that for $k\in [N]_0$ and $\ell\in [N]$,  
	\[
	(H_ST^{\trans})(k, \ell) = \sum_{j=0}^k T(\ell, j) = \frac{1}{\binom{N}{\ell}} \sum_{j=1}^k \binom{j-1}{\ell-1} 
	= \frac{1}{\binom{N}{\ell}} \sum_{j=0}^{k-1}\binom{j}{\ell-1} = \frac{\binom{k}{\ell}}{\binom{N}{\ell}} =  \frac{k^{\underline{\ell}}}{N^{\underline{\ell}}} = H_F(k, \ell),
	\]	
	with the usual convention that the empty sum is 0. Moreover, $H_S T^{\trans}(k,0)=\1\{k\geq 0\}=1=H_F (k,0)$, and $H_S T^{\trans}(k,N+1)=0=H_F T^{\trans}(k,N+1)$.
	Since $T$ is invertible (cf.\ \eqref{Tinv}), we also have  $H_S = H_F (T^{-1})^{\trans}$. 
	
	\smallskip
	
	Identity \eqref{eq:relationSiFa} can be exploited in our setting by writing down the duality relations in their matrix form, thanks to \cite[Prop. 1.2]{Jaku}. In particular, denote by $Q\in \R^{(N+1)\times (N+1)}$ the generator matrix of $X$. If~$X$ admits a factorial moment dual $X^F$ on~$[N+1]_0$ with generator matrix $Q_F\in\R^{(N+2)\times (N+2)}$, then, using~\eqref{eq:relationSiFa} and~$(T^{\trans})^{-1}$ of \eqref{Tinv},
	\begin{equation*}
		Q H_F = H_F Q_F^{\trans} \Longleftrightarrow Q H_S T^{\trans} = H_S  T^{\trans}  Q_F^{\trans} \Longleftrightarrow Q H_S = H_S T^{\trans} Q_F^{\trans} (T^{-1})^{\trans} = H_S (T^{-1} Q_F T)^{\trans}.
	\end{equation*}
	The last set of equalities establishes the Siegmund duality, provided  $T^{-1}Q_FT$ is indeed a generator matrix; this settles part~\eqref{factsieg} of the theorem. The proof of part~\eqref{siegfact} is completely analogous. 
\end{proof}

\section{Formalisation of the pruned lookdown ASG and proof of associated results}\label{sec:ancestraltype}

In the following, we make the verbal description of the pLD-ASG rigorous, and prepare and then prove the representation of the ancestral type distribution. 

\begin{definition}[pruned lookdown ASG]\label{def:pLDASG}
	Fix a realisation of the  family $\Lambda$ of Poisson point processes~\eqref{eq:poissonprocesses} and let $\varnothing \neq G_0=g=\{i_1,\ldots,i_n\}\subseteq [N]$ such that $i_1<\ldots<i_n$. Let $(G,g,\Lambda)$ be the corresponding ASG. The \emph{level process} $\ell = (\ell_r)_{r\geq 0}$ with $\ell_r \defeq (\ell_r(\eta))_{\eta \in G_r}$ and $\ell_r(\eta) \in [\lvert G_r\rvert]\cup\{\infty\}\}$, together with the \emph{immune-line process} $\imu = (\imu_r)_{r \geq 0}$ with $\imu_r\in[\lvert G_r \rvert ]$ are constructed as follows. Start by setting $\ell_0(i_m)=m$ for $m\in [n]$ and $\imu_0=n$. Proceed inductively for the arrival times $\varrho \in (-\Lambda) \cap \R_{+}$ in increasing order in the following way. In between arrival times of~$\Lambda$, $\ell$ and $\imu$ do not change. 
Given $\ell_{\varrho-}$ and $\imu_{\varrho-}$ for some $\varrho \in (-\Lambda)\cap \R_{+}$, we first define $\ell_\varrho$. 
	\begin{enumerate}
	\item[(0)] (Outside event) If $\varrho\in \Ts-(\Pi_{j,i}^{\narrow}\cup \Pi_{J,i}^{\sarrow}\cup \Pi_i^{\times},\Pi_i^{\circ})$ for $i\notin G_{\varrho-}^{\Ts}$ (and some $j,J$), then $\ell_\varrho=\ell_{\varrho-}$.
		\item[(1)] \begin{enumerate}
			\item(Coalescence) If $\varrho\in -\Pi_{j,i}^{\narrow}$ for some $i,j\in G_{\varrho-}$, then
			 $$\ell_\varrho(\eta)=\begin{cases}
			\ell_{\varrho-}(i)\wedge \ell_{\varrho-}(j), &\text{if }\eta=j,\\
			\ell_{\varrho-}(\eta),& \text{if }\eta\in G_{\varrho-}\setminus\{i\}\text{ and } \ell_{\varrho-}(\eta)< \big ( \ell_{\varrho-}(i)\vee \ell_{\varrho-}(j) \big ),\\
			\ell_{\varrho-}(\eta) -1,& \text{if }\eta\in G_{\varrho-}\setminus\{i\}\text{ and } \ell_{\varrho-}(\eta)> \big (\ell_{\varrho-}(i)\vee \ell_{\varrho-}(j) \big ).
			\end{cases}$$ 
			\item (Relocation) If $\varrho \in -\Pi_{j,i}^{\narrow}$ for some $i\in G_{\varrho-}$ and $j\notin G_{\varrho-}$, then
			\[
			\ell_{\varrho}(\eta) = \begin{cases}
				\ell_{\varrho-}(i), &\text{if }\eta = j, \\
				\ell_{\varrho-}(\eta), &\text{if }\eta \neq j. 	
			\end{cases} \]
		\end{enumerate}
		\item[(2)] (Selection)  If $\varrho\in -\Pi_{J,i}^{\sarrow}$ for some $i\in G_{\varrho-}$ and $J\in [N]^m$, $m\in \N$,  let $\breve{J} \setminus \{\eta\in G_{\varrho-} :\ell_{\varrho-}(\eta)<\ell_{\varrho-}(i)\} = \{J_{j_1},\ldots,J_{j_{\kappa}}\}$, where $j_1<\ldots<j_\kappa$ for some $\kappa\in[m]_0$ (here $\kappa=0$ means that the set is empty). Then, for $\eta\in G_{\varrho-}$,
		$$\ell_{\varrho}(\eta)=\begin{cases}
		\ell_{\varrho-}(\eta),&\text{if }\eta \in G_{\varrho-} \text{ and }\ell_{\varrho-}(\eta)<\ell_{\varrho-}(i),\\
		\ell_{\varrho-}(i)+\beta-1,&\text{if } \eta=J_{j_\beta} \text{ for some } \beta\in[\lvert\kappa\rvert],\\
		\ell_{\varrho-}(\eta)+\kappa, &\text{if } \eta \in G_{\varrho-} \text{ and }\ell_{\varrho-}(\eta)\geq\ell_{\varrho-}(i),\, \eta\notin\breve{J}. 		\end{cases}$$
		\item[(3)] (Deleterious mutation) If $\varrho\in -\Pi_{i}^{\times}$ for some $i\in G_{\varrho-}$, then for $\eta\in G_\varrho$,
		$$\ell_\varrho(\eta)=\begin{cases}
		\ell_{\varrho-}(\eta),&\text{if } \ell_{\varrho-}(\eta)<\ell_{\varrho-}(i),\\
		\ell_{\varrho-}(\eta)-1, &\text{if } \ell_{\varrho-}(\eta)>\ell_{\varrho-}(i),\\
		\lvert\{j\in G_{\varrho-}: \ell_{\varrho-}(j)\neq \infty\}\rvert,&\text{if }\eta=i \text{ and } \ \  \ell_{\varrho-}(i)=\imu_{\varrho-},\\
		\infty,&\text{if }\eta=i,\ \  \ell_{\varrho-}(i)\neq \imu_{\varrho-}.		\end{cases}$$
		\item[(4)] (Beneficial mutation) If $\varrho\in -\Pi_{i}^{\circ}$ for some $i\in G_{\varrho-}$, then for $\eta\in G_\varrho$,
		$$\ell_\varrho(\eta)=\begin{cases}
		\ell_{\varrho-}(\eta),&\text{if } \ell_{\varrho-}(\eta)\leq\ell_{\varrho-}(i),\\
		\infty, &\text{if } \ell_{\varrho-}(\eta)>\ell_{\varrho-}(i).
		\end{cases}$$
	\end{enumerate}
	Finally, $\imu_\varrho$ is given as follows (still given $\ell_{\varrho-}$ and $\imu_{\varrho-}$). If  $\varrho\in -\Pi_{i}^{\circ}$ for some $i\in G_{\varrho-}$ with $\ell_{\varrho-}(i) < \infty$, then $\imu_\varrho =  \ell_{\varrho-}(i).$
	If  $\varrho\in -\Pi_{j,i}^{\narrow}$ for some $i\in G_{\varrho-}$ with $\ell_{\varrho-}(i),\ell_{\varrho-}(j) < \infty$ and $\imu_{\varrho-}=i$, then $\imu_\varrho =  \ell_{\varrho-}(i) \wedge \ell_{\varrho-}(j).$ 
	In all other cases,
	\begin{equation}\label{eq:immunefollows}
			\imu_\varrho = \ell_\varrho \big ( \ell_{\varrho-}^{-1}(\imu_{\varrho-}) \big ).
	\end{equation}
	 We refer to $\ell_r(i)$ as the level of line~$i$ at (backward) time~$r$, to $\imu_r$ as the (level of the) immune line at time~$r$, and to $(G,g,\Lambda,\ell)$ as the pLD-ASG, $(G_{[0,r]},g,\Lambda,\ell)$ is its restriction until time~$r$. The \emph{line-counting process} $L=(L_r)_{r\geq 0}$ of the pLD-ASG is formally defined via
	 $$L_r \defeq \lvert\{j\in G_r: \ell_r(j)\neq \infty \}\rvert.$$
\end{definition}
Each finite level is associated with a unique site. In particular, the site of the immune line is~$\ell_r^{-1}(\imu_r)$.
Recall that in all transitions except those induced by beneficial mutations and neutral reproduction events, this site remains unchanged, but it moves to a new level whenever  the site does so; this is what \eqref{eq:immunefollows} tells us.

\smallskip

 The pLD-ASG is constructed so as to facilitate to identify the ancestral line for a given site colouring. This crucial feature  is the content of the next proposition.

\begin{figure}[t]
	\begin{center}
		\scalebox{0.8}{\begin{tikzpicture}

			%%horizontal line
			\draw[ ]  (3.7,0.7) -- (4,0.7) -- (4.3,1.4)-- (4.75,1.4) -- (5.05,.7) -- (6.75,.7) -- (7.05,0)-- (9.75,0) -- (10.05,1.4) -- (10.5,1.4) -- (10.7,2.1) -- (10.9,2.1)-- (11,2.8)-- (11.75,2.8) -- (12.05,1.4) -- (12.75,1.4) -- (13.05,.7) -- (13.75,.7) -- (14.05,0) -- (14.5,0);

			\draw[]  (3.7,1.4) -- (4,1.4) -- (4.3,0)--(4.5,0);
			\draw[]  (3.7,0)-- (4,0) -- (4.3,0.7) -- (4.5,0.7);

			%%incoming line
			\draw[-{open triangle 45[scale=2.5]}] (11,0) -- (14,0);

			\draw[-{open triangle 45[scale=2.5]}] (8,2.1) -- (8.65,2.1) -- (8.95,1.4) -- (9.75,1.4) -- (10.05,.7) -- (10.5,.7) -- (10.7,1.4) -- (10.9,1.4) -- (11,2.1) -- (11.75,2.1) -- (12.05,.7) -- (13,.7);

			\draw[-{open triangle 45[scale=2.5]}] (9.5,0) -- (10,0);
			
			\draw[-{open triangle 45[scale=2.5]}]  (6,1.4) -- (6.75,1.4) -- (7.05,.7) -- (8.9,0.7);
			
			\draw[-{open triangle 45[scale=2.5]},]  (4.5,0) -- (7,0);
			
			\draw[-{open triangle 45[scale=2.5]},]  (4.5,0.7) -- (5,0.7);
			
			\draw[-{open triangle 45[scale=2.5]},]  (10.65,0.7) -- (12,.7);
			\draw  (11,1.4)-- (11.75,1.4);
			
			\draw (11.75,1.4) .. controls (11.65,1.2) and (11.65,0.8) .. (11.75,.7);
			%\draw (11.75,1.4) -- (11.75,0.7);
			\draw (11.75,0) .. controls (11.65,0.2) and (11.65,0.6) .. (11.75,.7);			\draw (8,1.4) -- (8.65,1.4) -- (8.9,0.7) -- (9.75,0.7) -- (10,0) -- (11.75,0);

			%Neutral events
			\draw[-{triangle 45[scale=2.5]},]  (6,0) -- (6,1.4);
			\draw[-{triangle 45[scale=2.5]},]  (11,0.7) -- (11,1.4);

			%\draw[dashed] (8,.7) -- (8,2.1);
			
			%Immune line
			
			\draw[line width=.5mm, lightgray ] (8,0)-- (9.75,0) -- (10.05,1.4) -- (10.5,1.4) -- (10.7,2.1) -- (10.9,2.1)-- (11,2.8)-- (11.75,2.8) -- (12.05,1.4) -- (12.75,1.4) -- (13.05,.7) -- (13.75,.7) -- (14.05,0) -- (14.5,0);
			\draw[line width=.5mm, lightgray ] (6,1.4) -- (6.75,1.4) -- (7.05,0.7)-- (8,0.7);
			\draw[line width=.5mm, lightgray ] (3.7,1.4)--(4,1.4)-- (4.3,0) -- (6,0);

			%Mutations
			\node at (10.65,.7) {\scalebox{1.6}{$\times$}} ;
			\node at (4, 1.4) {\scalebox{1.6}{$\times$}} ;
			\node at (13.4, 0.7) {\scalebox{1.6}{$\times$}} ;
			\draw (8,.7) circle (1.7mm)  [fill=white!100];
			
			\draw (9.75,0) circle (.4mm)  [fill=black!100];
			\draw (13.75,0) circle (.4mm)  [fill=black!100];
			\draw (12.75,0.7) circle (.4mm)  [fill=black!100];
			\draw (8.65,0.7) circle (.4mm)  [fill=black!100];
			\draw (6.75,0) circle (.4mm)  [fill=black!100];
			\draw[] (4.75,0.7) circle (.4mm)  [fill=black!100];

			\draw[] (11.75,1.4) circle (.4mm)  [fill=black!100];
			\draw[] (11.75,.7) circle (.4mm)  [fill=black!100];
			\draw[] (11.75,0) circle (.4mm)  [fill=black!100];

			\draw[dashed,opacity=0] (11,.7) --(11,2.45);
			
			\end{tikzpicture}}
		\end{center}
		\caption[Realisation of the pLD-ASG for finite populations]{A cutout of a realisation of the pLD-ASG. Only lines at finite levels are shown. Grey: level of immune line.}
		\label{fig:pldASGfinitepopulations}
	
\end{figure}
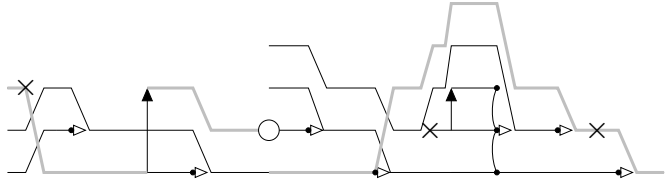

 \smallskip

\begin{proposition}\label{prop:levelancestralsite}
	Let $g=\{1\}$ and consider the pLD-ASG $(G_{[0,r]},g,\Lambda,\ell)$ for some $r>0$ together with some site colouring. Then, the level of the ancestral line at backward time $r$ is almost surely either the lowest finite level that has type $0$ at time~$r$;  or, if all finite levels are of type $1$, it is $\imu(r)$, the level of the immune line at time $r$. In particular, the ancestral site is of type~$1$ at time $r$ if and only if all  lines at finite levels are of type~$1$ at time $r$.
\end{proposition}

\begin{proof}
	The proof is similar to the proof of Prop.~2 in \cite{LKBW15}. We recall the argument and adapt the part associated with the more general form of selection. Let $0<\varrho_1<\ldots<\varrho_m$ be the arrival times of~$-\Lambda \cap (0,r]$  and set $\varrho_0=0$. We will prove by induction that for any $k\in [m]_0$, the level of the ancestral line at time $\varrho\in[\varrho_k,\varrho_{k+1})$ is either the line at time~$\varrho$ that, under a given site colouring  assigned at time~$\varrho$, is the lowest finite level that carries type~$0$; or, if all finite levels are of type~$1$, it is $\imu(\varrho)$. For $k=0$, the claim is trivially true. Next, we prove the claim for times $\varrho\in [\varrho_{k},\varrho_{k+1})$ assuming it is true at any time~$<\varrho_k$. Because the types and ancestral sites are constant  between events in $(0,r]\cap -\Lambda$, it suffices to prove the claim for $\varrho=\varrho_k$ assuming the claim is true at time $\varrho_{k-1}$. We denote by $\lambda_k$  and $\lambda_{k-1}$  the lowest finite type-$0$ level at $\varrho_k$ and $\varrho_{k-1}$, respectively, under a given site colouring assigned at time~$\varrho_k$ --- provided such levels exist. Consider (I) the case  that all finite levels at time $\varrho_k$ are assigned type~$1$. If, (A), the event at time~$\varrho_k$ is not a beneficial mutation, then $\imu_k$ is the predecessor of~$\imu_{k-1}$, which  is the ancestral line by the induction hypothesis. If, (B), the event at time~$\varrho_k$ is a beneficial mutation, then $\lambda_{k-1}=\imu_k$ and $\imu_k$ is the predecessor of $\lambda_{k-1}$, which is the ancestral line by the induction hypothesis. Hence, in (IA) and (IB), the ancestral line at time~$\varrho_k$ is $\imu_k$.  We are therefore left to consider  case (II) where  at least one finite level  is assigned type $0$ at time $\varrho_k$. We now consider the possible events at time~$\varrho_k$. For mutations, coalescence and relocations events, argue as in~\citep[Proof of Prop.~2]{LKBW15}.
	
	\smallskip

	In a selective event, the order among all finite-level lines at time~$\varrho_k$ that are not incoming at time~$\varrho_{k-1}$ carries over to the descendants at time~$\varrho_{k-1}$. Moreover, all finite-level lines at time~$\varrho_{k-1}$ are descendants of lines at finite levels at time~$\varrho_k$. If $\lambda_k$ is not a potential parent in the selective event, then $\lambda_k$ is the parent of $\lambda_{k-1}$. If $\lambda_k$ is a potential parent, it is the one with the lowest type-0 level in this event and therefore, by the  propagation rule, again the parent of $\lambda_{k-1}$. By the induction hypothesis, the claim follows. 

\end{proof}

Next, we derive the infinitesimal generator of~$L$.

\begin{proof}[Proof of Proposition~\ref{prop:eqdistL}]
	Assume that $L_r=n$ for some $r\geq0$. For $j\in [N-n]$, the next jump of~$L$ is to $n+j$ if the first arrival of $-\Lambda\cap(r,\infty)$ is in $-\Pi_{J,i}^{\sarrow}$ for some $i\in G_r$ with  $\ell_r(i) <  \infty$ and $\breve{J}\subseteq[N]$ with $\lvert \breve{J}\rvert \geq j$ such that $j$ lines in~$\breve{J}$ have level~$\infty$ or are not in $G_r$. Such an arrival occurs at rate $n\sum_{m \geq j} s_m \, \frac{\faf{(N-n)}{\,j}}{N^m}C^n_{mj}$, which agrees with the rate at which~$R$ makes such a transition. The first jump of~$L$ after backward time~$r$ is to $n-1$ if the first arrival in $-\Lambda\cap(r,\infty)$ is in $-\Pi_{j,i}^{\narrow}$ for some $i,j\in G_r$ with $i \neq j$, $\ell_r(i),\ell_r(j)\neq \infty$, or if it is in $-\Pi_i^\times$ for some $i\in G_r$ such that $\ell_r(i)\notin\{\infty, \imu_r\}$, or if it is in $-\Pi_i^\circ$ for $i\in G_r$ such that $\ell_r(i)=n-1$. The first of these events occurs at rate $n\frac{n-1}{N}+u\nu_1(n-1)+u\nu_0\1_{\{n>1\}}$. The first jump of $L$ after time~$r$ is to~$j \in [n-2]$ if the first arrival of~$-\Lambda\cap(r,\infty)$ is in $-\Pi_i^\circ$ for $i\in G_r$ with $\ell_r(i)=j$, which occurs at rate~$u\nu_0$. 
\end{proof}

We are now ready to prove the representation of the ancestral type distribution.
\begin{proof}[Proof of Theorem~\ref{thm:atyperepresentation}]
	The sample has an unfit ancestor at backward time~$r$ if and only if all the individuals at backward time $r$ in the pLD-ASG are type~$1$ by Proposition~\ref{prop:levelancestralsite}; for a given pLD-ASG and a site colouring at time~$r$ with~$k$ type-$1$ assignments the probability for this is $\faf{k}{L_r}/\faf{N}{L_r}$. Averaging over all realisations of $(L_v)_{v\in[0,r]}$ yields~\eqref{eq:atypedist}. \eqref{eq:catypedist} follows because $L$  converges to its stationary distribution.
\end{proof}

\begin{remark} 
	\begin{enumerate}
		\item Note that we work here with the  Poisson point processes $\Lambda$ that define the ASG;  so $\ell_r$ and~$\imu_r$ are functionals of the ASG and hence measurable with respect to $\sigma(\Lambda \cap [-r,0])$  for given $\varnothing\neq g\subseteq[N]$ and $r>0$. Due to the exchangeability, however, the pLD-ASG may as well be constructed as a Markov process independent of an ASG, by using an analogous family of Poisson point processes attached to  the \emph{levels of the lookdown},  rather than the \emph{lines of the ASG}. The former is in the spirit of the  original  lookdown construction by Donnelly and Kurtz \cite{DK99}; see \cite{LKBW15} for a discussion of both possibilities in the case of genic selection.
		\item It is customary (and required for the formulation of a duality) to not insist on starting the pLD-ASG from a single individual. One should keep in mind, however, that if we start the process with~$\lvert G_0\rvert>1$ 
	lines, then $L$ does not correctly describe the number of potential ancestors of a sample of~$\lvert G_0\rvert$ individuals. For example, assume that the first event is a beneficial mutation on level~$1$. This induces the pruning of all other levels, which does not properly reflect the potential ancestry of~$\lvert G_0\rvert$ individuals.
		\item If~$u=0$ and $L_0=R_0$, then~$L=R$ in distribution and hence $L_{\infty}=R_{\infty}$. In this case, one type goes to fixation. In particular, the probability of the common ancestor of the population in the distant future  to be unfit at present coincides with the fixation probability of the unfit type at present.
		\item We recover the known representation of the probability for a fit common ancestor in the MoMo with genic selection of~\citep[Prop.~4.7]{Cordero17DL} by rewriting $\P(L_\infty=n)$ in~\eqref{eq:catypedist} as $\P(L_\infty>n-1)-\P(L_\infty>n)$ and rearranging terms, which yields\begin{equation}\label{eq:catypedist2}
			1-h_\infty(k)=1-\E\Bigg[\frac{k^{\underline{L_{\infty} } }}{N^{\underline{L_{\infty} } }}\Bigg]=\frac{N-k}{N}\sum_{n=0}^{N-1}\P(L_{\infty}>n) \,\frac{\faf{k}{n} }{\faf{(N-1)}{n}}. \end{equation}
		This means that we can partition the event of a beneficial ancestor according to the first finite level occupied by a type-$0$ individual. Namely, \begin{equation}
			\P(L_\infty>n)\, \frac{N-k}{N}\,\frac{\faf{k}{n}}{\faf{(N-1)}{n}}\label{eq:ancestraltailrepresentation}
		\end{equation}  is the probability that at least~$n+1$ finite levels are occupied, the first~$n$ levels are of type~$1$, and the ~$(n+1)^{st}$ level carries type~$0$. Summing this probability over~$n$ gives the probability of a fit ancestor in a stationary pLD-ASG.
	\end{enumerate}
\end{remark}

\section{A forward approach to the ancestral type distribution}\label{sec:absorbingmarkov}
The forward approach to the ancestral type distribution is based on~$\widetilde{Y}$. First, we formally describe this process. Next, we prove the factorial moment duality between~$\widetilde{Y}$ and $L$. The alternative representation of the ancestral type distribution then follows easily. Finally, we formally define the descendant process and prove the connection to $\widetilde{Y}$.

\subsection{The process $\widetilde Y$}\label{sec:widetildeY}

Consider a site colouring with $\widetilde{Y}_0\in[N]_0$ type-$1$ individuals; time runs forward. Let~$\Ce$ be the corresponding typed FTW MoMo. Recall the family of Poisson processes in~\eqref{eq:poissonprocesses} and, for each $\tau\in \{\Pi_i^\star: i\in[N],\, \star\in \{\times,\circ\}\}$, let $B_\tau^\circ$ and $B_\tau^\times$ be random variables such that given $Y_\tau$, $B_\tau^\circ$ and $B_\tau^\times$ has Bernoulli distribution with parameter $1/(N-Y_\tau+1)$ and $1/(Y_\tau+1)$, respectively. For $\star\in\{\times,\circ\}$, define  $\widetilde{T}^\star\defeq\inf\{\tau\geq 0: Y_\tau\neq Y_{\tau+}, \tau\in \Pi_i^\star \text{ for some } i\in[N], B^\star_\tau=1\}$, and $T_{0,N}\defeq\min\{t\geq 0: Y_t\in\{0,N\}\}$ (the usual convention applies: these times are infinite if the events defining them never occur). Set $\widetilde{T}\defeq\min\{\widetilde{T}^\times,\widetilde{T}^\circ,T_{0,N}\}$. Then, for $t\leq \widetilde{T}$, set $\widetilde{Y}_t\defeq Y_t$. For $t>\widetilde{T}$, set $$\widetilde{Y}_t\defeq \begin{cases}
	0,&\text{if } \widetilde{T}=\widetilde{T}^\times,\\
	N, & \text{if } \widetilde{T}=\widetilde{T}^\circ,\\
	Y_{T_{0,N}}, & \text{if } \widetilde{T}=T_{0,N}.
\end{cases}$$
Note that the states~$0$ and~$N$ are absorbing.

\smallskip

The so-constructed process~$\widetilde{Y}\defeq(\widetilde{Y}_t)_{t\geq 0}$ is a continuous-time Markov chain on~$[N]_{0}$, with an infinitesimal generator that acts on functions~$f:[N]_0\to \R$ and is given by~$\As_{\widetilde{Y}}=\As_{Y}^{\rm{n}}+\sum_{m>0}\As_{Y}^{s_m}+\As_{\widetilde{Y}}^{\nu_0}+\As_{\widetilde{Y}}^{\nu_1},$ with~$\As_{Y}^{\rm{n}}$ and~$\As_{Y}^{s_m}$ of~\eqref{eq:generatorhatY}, respectively, and for $k\in [N]_0$,\begin{equation}\label{generatorYtilde}
\begin{split} \As_{\widetilde{Y}}^{\nu_0}f({k})&\defeq{k}\, \frac{N-{k}}{N-{k}+1}\,u\nu_0\,[f({k}-1)-f({k})]+\frac{{k}}{N-{k}+1}\,u\nu_0\,[f(N)-f({k})],\\
\As_{\widetilde{Y}}^{\nu_1}f({k})&\defeq(N-{k})\,\frac{{k}}{{k}+1}\,u\nu_1\,[f({k}+1)-f({k})]+ \frac{N-{k}}{{k}+1}\,u\nu_1\,[f(0)-f({k})] .\end{split}
\end{equation}

\begin{figure}[t]
	\begin{center}
		\hspace{-2cm}\begin{minipage}{0.4\textwidth} \begin{center}
				\scalebox{0.7}{
					\begin{tikzpicture}
					\draw (0,-.5) -- (0,3) [dashed];
					\draw (5,-.5) -- (5,3) [dashed];

					\node at (2.5,1.5) {\scalebox{1.8}{$\times$}};

					\draw[darkgreen,line width=0.6mm] (0,2.5) -- (5,2.5);
					\draw[darkgreen,line width=0.6mm] (0,2) -- (5,2);
					\draw[darkgreen,line width=0.6mm] (0,1.5) -- (2.5,1.5) ;
					\draw[lightbrown,line width=0.6mm] (2.5,1.5) -- (5,1.5) ;
					\draw[lightbrown,line width=0.6mm] (0,1) -- (5,1);
					\draw[lightbrown,line width=0.6mm] (0,0.5) -- (5,0.5);
					\draw[lightbrown,line width=0.6mm] (0,0) -- (5,0);
					\node (s1) at (5,1.5) {} ;
					\node (s2) at (5,0.5) {} ;
					\node (s3) at (5,0) {} ;
					%\fill [lightbrown] (s1) circle (2pt);
					%\fill [lightbrown] (s2) circle (2pt);
					%\fill [lightbrown] (s3) circle (2pt);
					%\node[draw, cross out] at (2.5,1.5) {};
					\draw [decorate,decoration={brace,amplitude=4pt},xshift=-3]
					(0,-0.10) -- (0,1.1) node [black,midway,left, xshift=-5pt] {\scalebox{1.3}{$k$}}; 
					\draw [decorate,decoration={brace,amplitude=4pt},xshift=-3]
					(0,1.2) -- (0,2.6) node [black,midway,left, xshift=-5pt] {\scalebox{1.3}{$N-k$}}; 
					\draw [decorate,decoration={brace,amplitude=4pt},xshift=3]
					(5,1.6) -- (5,-0.10) node [black,midway,right] {\ \scalebox{1.3}{$k+1$}};
					\draw [decorate,decoration={brace,amplitude=4pt},xshift=3]
					(5,2.6) -- (5,1.7) node [black,midway,right] {\ \scalebox{1.3}{$N-(k+1)$}};
					\end{tikzpicture}
				}
			\end{center}
		\end{minipage}\hspace{1.3cm}\begin{minipage}{0.3\textwidth}
			\begin{center}
				\scalebox{0.7}{
					\begin{tikzpicture}
					\draw (0,-.5) -- (0,3) [dashed];
					\draw (5,-.5) -- (5,3) [dashed];

					\draw[darkgreen,line width=0.6mm] (0,2.5) -- (5,2.5);
					\draw[darkgreen,line width=0.6mm] (0,2) -- (5,2);
					\draw[lightbrown,line width=0.6mm] (0,1.5) -- (2.5,1.5) ;
					\draw[darkgreen,line width=0.6mm] (2.5,1.5) -- (5,1.5) ;
					\draw[lightbrown,line width=0.6mm] (0,1) -- (5,1);
					\draw[lightbrown,line width=0.6mm] (0,0.5) -- (5,0.5);
					\draw[lightbrown,line width=0.6mm] (0,0) -- (5,0);
					\node (s1) at (5,2.5) {} ;
					\node (s2) at (5,1.5) {} ;
					\node (s3) at (5,0) {} ;
					%\draw[fill=white] (2.5,1.5) circle [radius=0.25cm];
					\draw[fill=white!100] (2.5,1.5) circle (1.8mm);
					%\node[draw, circle, fill=white, color=white] at (2.5,1.5) {};
					\draw [decorate,decoration={brace,amplitude=4pt},xshift=-3] (0,-0.10) -- (0,1.6) node [black,midway,left, xshift=-5pt] {\scalebox{1.3}{$k$}}; 
					
					\draw [decorate,decoration={brace,amplitude=4pt},xshift=-3] (0,1.7) -- (0,2.6) node [black,midway,left, xshift=-5pt] {\scalebox{1.3}{$N-k$}}; 
					
					\draw [decorate,decoration={brace,amplitude=4pt},xshift=3] (5,1.1) -- (5,-0.10) node [black,midway,right] {\ \scalebox{1.3}{$k-1$}};
					\draw [decorate,decoration={brace,amplitude=4pt},xshift=3] (5,2.6) -- (5,1.1) node [black,midway,right] {\ \scalebox{1.3}{$N-(k-1)$}};
					\end{tikzpicture}
				}
			\end{center}
		\end{minipage}
	\end{center}
	\caption[Type changing mutation in the graphical representation]{If $Y_{t-}=k$, a type-changing deleterious mutation (beneficial mutation) on a given line changes the type distribution in the MoMo  from $(N-k,k)$ to~$(N-(k+1),k+1)$ (to~$(N-(k-1),k-1)$) and is depicted on the left (right). At each type-changing deleterious mutation (beneficial mutation), flip a coin with frequency-dependent success probability $1/(k+1)$ (success probability $1/(N-k+1)$). In the case of  success, $\widetilde{Y}$ jumps to~$0$ (to~$N$). Recall that a dark green (light brown) line corresponds to an fit (unfit) individual. }
	\label{fig:typeChangeAncestral}
\end{figure}
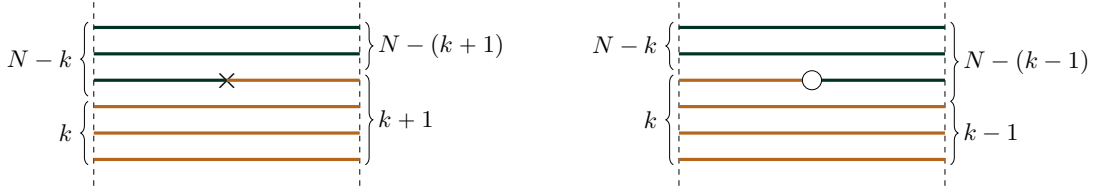

To prove the factorial moment duality between $L$ and $\widetilde{Y}$, we require the following auxiliary lemma.

\begin{lemma}[Auxiliary lemma]\label{lem.binom}
	For~$N,{k},n\in \Nb$ with $k,n \leq N$,
	\begin{equation}
	\sum_{j=1}^{n-1} \frac{\faf{k}{j}}{\faf{N}{j}} = \frac{k}{N-k+1} - \frac{N-n+1}{N-k+1} \frac{\faf{k}{n}}{\faf{N}{n}},
	\end{equation}
	with the usual convention that the empty sum is 0.
\end{lemma}
\begin{proof} 
The statement is proved via an elementary induction over $n$. For $n=1$, it is trivially true.  Using the induction hypothesis in the first step, we obtain for the induction step
\[
\begin{split}
\sum_{j=1}^{n} \frac{\faf{k}{j}}{\faf{N}{j}} &  = \frac{k}{N-k+1} + \Big ( - \frac{N-n+1}{N-k+1}+1 \Big ) \frac{\faf{k}{n}}{\faf{N}{n}}  
 = \frac{k}{N-k+1} - \frac{k-n}{N-k+1} \frac{\faf{k}{n}}{\faf{N}{n}} \\ & = \frac{k}{N-k+1} - \frac{N-n}{N-k+1} \frac{\faf{k}{n+1}}{\faf{N}{n+1}},
\end{split}
\]
which proves the claim.
\end{proof}
Next, we prove the factorial moment duality between $\widetilde{Y}$ and $L$.
{\allowdisplaybreaks
	\begin{proof}[Proof of Theorem~\ref{thm:finite.dual.mut}]
		We want to apply~\citep[Prop.~1.2]{Jaku}. Since the state space of~$\widetilde{Y}$ is finite, every $\mathbb{R}$-valued function is in the domain of~$\As_{\widetilde{Y}}$. In particular, for $k\in[N]_0$, $n\in[N]$, $H_F(\cdot,n)(k)$ and $P^{\widetilde{Y}}_tH_F(\cdot,n)(k)$ lie in the domain of~$\As_{\widetilde{Y}}$, where $(P^{\widetilde{Y}}_t)_{t\geq 0}$ is the transition semigroup corresponding to~$\widetilde{Y}$. Similarly, $H_F(k,\cdot)(n)$ and $P^L_tH_F(k,\cdot)(n)$ lie in the domain of~$\As_L$, where $(P_t^L)_{t\geq 0}$ is the transition semigroup corresponding to~$L$. In the proof of Theorem~\ref{thm:duality_R}, we already showed that $\As_{Y}^{\rm{n}} H_F(\cdot,n)({k})=\As_{R}^{\rm{n}} H_F({k},\cdot)(n)$ and $\As_{Y}^{s_m} H_F(\cdot,n)({k})=\As_{R}^{s_m} H_F({k},\cdot)(n)$ for all $m>0$. Hence, it suffices to check that for all $k\in[N]_0$, $n\in[N]$, $$\As_{\widetilde{Y}}^{\nu_0}H_F(\cdot,n)({k})=\As_{L}^{\nu_0}H_F({k},\cdot)(n)\quad \text{and}\quad \As_{\widetilde{Y}}^{\nu_1}H_F(\cdot,n)({k})=\As_{L}^{\nu_1}H_F({k},\cdot)(n),$$ which then implies $\As_{\widetilde{Y}}H_F(\cdot,n)({k})=\As_{L}H_F({k},\cdot)(n)$ for all ${k}\in[N]_0,\, n\in [N]$. First note that for $k=0$ or $n >k$ the result is trivial. It is then enough to prove for $k \in [N], n \in [k]$. For the part corresponding to the type-$1$ mutation we obtain 
		\begin{align*}\allowdisplaybreaks
		\As_{L}^{\nu_1}H_F({k},\cdot)(n)&= (n-1)u\nu_1 \big(N-n+1-({k}-n+1)\big)\frac{\faf{{k\,}}{n-1}}{\faf{N}{n}}
		=u\nu_1 (N-{k})\,\frac{{k}-({k}-n+1)}{{k}+1} \, \frac{\faf{({k}+1)}{n}}{\faf{N}{n}}  \\
		&=u\nu_1 (N-{k})\,\frac{{k}}{{k}+1}\,\Bigg[\frac{\faf{({k}+1)}{n}}{\faf{N}{n}}-\frac{\faf{{k\,}}{n}}{\faf{N}{n}}\Bigg]+u\nu_1\,\frac{N-{k}}{{k}+1}\,\Bigg[-\frac{\faf{{k\,}}{n}}{\faf{N}{n}}\Bigg]
		=\As_{\widetilde{Y}}^{\nu_1}H_F(\cdot,n)({k}).
		\end{align*}
		For the part associated to mutation to type~$0$, we get,  with the help of Lemma~\ref{lem.binom} in the second step,
\[
\begin{split}
\As_{L}^{\nu_0}H_F({k},\cdot)(n)&= u\nu_0 \sum_{j=1}^{n-1}\Bigg [ \frac{{k\,}^{\underline{j}}}{N^{\underline{j}}}-\frac{\faf{k\,}{n}}{\faf{N}{n}}\Bigg ] 
= u\nu_0 \Bigg [ \frac{k}{N-k+1} - \frac{N-n+1}{N-k+1}  \frac{\faf{{k\,}}{n}}{\faf{N}{n}} - (n-1) \frac{\faf{{k\,}}{n}}{\faf{N}{n}} \Bigg ] \\
& = u\nu_0 \frac{k}{N-k+1} \Bigg [ \Big ( 1- \frac{\faf{{k\,}}{n}}{\faf{N}{n}} \Big ) - (N-k) \frac{n}{k} \frac{\faf{{k\,}}{n}}{\faf{N}{n}} \Bigg ] \\
& = u\nu_0 \frac{k}{N-k+1} \Bigg [ \Big ( 1- \frac{\faf{{k\,}}{n}}{\faf{N}{n}} \Big ) - (N-k)  \frac{\faf{(k-1)}{n} - \faf{k}{n}}{\faf{N}{n}}  \Bigg ] = \As_{\widetilde Y}^{\nu_0}H_F(\cdot,n)(k),
\end{split}
\]
where the second-last step is true since $- \frac{n}{k} \faf{k}{n} = \faf{(k-1)}{n} - \faf{k}{n} $. 
	\end{proof}
	
	This factorial moment duality leads to a representation of~$h_{\infty}$ as an absorption probability that does not depend on~$L$. It transpires in Section~\ref{mr:pldASGdifflimit} that this representation is the natural analogue to~\citep[Prop.~2.5]{Taylor2007} in the finite-population setting. We now derive this representation.

	\begin{proof}[Proof of Corollary~\ref{coro:jumpprocess}]
	 Theorem~\ref{thm:atyperepresentation} and \eqref{eq:dualitymarkovjumpchain} yield the representation of $h_t(k)$. Taking~$t\to\infty$ then leads to~\eqref{eq:characerizationhinffinite}. The boundary conditions follow also from Theorem~\ref{thm:atyperepresentation}. A first-step decomposition of the absorption probability of~$\widetilde{Y}_t$ in~$N$ leads to
		\begin{align*} &\Bigg[{k}\bigg\{2\,\frac{N-{k}}{N}+\sum_{m=1}^\infty s_m \bigg(1-\bigg(\frac{k}{N}\bigg)^m\bigg)\bigg\}+(N-{k})u\nu_1+{k}u\nu_0\Bigg]h_{\infty}({k}) \\
		&= k\Bigg[\frac{N-{k}}{N}+\frac{N-{k}}{{k}+1}\,u\nu_1\Bigg]h_{\infty}({k}+1)+k\Bigg[\frac{N-{k}}{N}+\sum_{m=1}^\infty s_m\bigg( 1-\bigg(\frac{k}{N}\bigg)^m\bigg)+\frac{N-{k}}{N-({k}-1)}\,u\nu_0\Bigg]h_{\infty}({k}-1)\\
		&\ +\frac{{k}}{N-{k}+1}\,u\nu_0,  \qquad \qquad ({k}\in [N-1]),\end{align*}
		where we used the boundary conditions. Dividing by ${k}$ leads to~\eqref{eq:boundaryvalueproblem}. Note that the system of equations~\eqref{eq:boundaryvalueproblem} can be written in matrix form. More precisely, define $D=(d_{jk})\in \R^{(N-1)\times (N-1)}$ with
		$$d_{jk}=\begin{cases}
		\frac{N-k}{N} +\sum_{m=1}^\infty s_m\big(1-\big(\frac{k}{N}\big)^m\big) + \frac{N-k}{N-{k}+1}u\nu_0, &\text{if } k = j-1,
		\\
		-\Big(2\,\frac{N-k}{N}+\sum_{m=1}^\infty s_m \big(1-\big(\frac{k}{N}\big)^m\big)+\frac{N-k}{{k}}u\nu_1+u\nu_0\Big), &\text{if } k=j,\\
		\frac{N-k}{N}+\frac{N-k}{{k}+1}u\nu_1 , &\text{if } k=j+1,\\
		0,&\text{otherwise},
		\end{cases}$$
		where $j,k\in[N-1]$. Writing $h_{\infty}\coloneqq (h_{\infty}(k))_{k=1}^{N-1}$, \eqref{eq:boundaryvalueproblem} is equivalent to $Dh_{\infty}=b$ for some $b\in \R^{n-1}$. $D$ is a strictly diagonally dominant matrix, i.e. $\lvert d_{ii}\lvert>\sum_{j\neq i} \lvert d_{ij}\rvert$. It follows from the L\'{e}vy--Desplanques theorem (e.g. \citep[Cor.~5.6.17]{johnson1985matrix}) that $D$ is nonsingular. In particular, the solution of the recursion with the boundary conditions is unique.
	\end{proof}
	
}

\subsection{Descendant process}\label{sec:descendantprocess}
We begin by recalling the definition of the descendant process of~\citep{Kluth2013}.

\begin{definition}[Descendant process] 
	Consider the setup and notation of Definition~\ref{def:typingMoMo}, i.e. fix a site colouring $\alpha\in \{0,1\}^{[N]}$ and a realisation of the family~$\Lambda$ of Poisson processes. For $t\geq 0$ and a \emph{starting set} $A\subseteq [N]$, define $D_t \defeq \lvert \{i \in [N]: \Ce_t(i) = 1,\, \Ae_t(i)\in A\} \rvert\in[N]_0$ as the number of unfit descendants at time~$t$ of individuals in $A$ at time $0$, and analogously, $B_t\defeq \lvert \{i \in [N]: \Ce_t(i) = 0, \, \Ae_t(i)\in A\} \rvert\in[N]_0$ those of the fit type. We refer to $(Y,D,B)=(Y_t, D_t, B_t)_{t \geq 0}$ as the descendant process started from $A$ (with site colouring~$\alpha$). 
\end{definition}
$(Y, D, B)$ is a time-homogenous continuous-time Markov chain with values in $\Theta=\{(k, d, b)\in [N]_0^3: d\leq k,\, b\leq N-k\}$. The initial distribution of the descendant processes will often be prescribed by some deterministic $(Y_0,D_0,B_0)$ without explicitly stating a starting set or site colouring. It is meant that they are then uniformly chosen among all starting sets and site colourings compatible with~$(Y_0,D_0,B_0)$. The (lengthy) generator of $(Y,D,B)$ can be derived in a straightforward  calculation; it acts on $f:\Theta\to\R$ and is given by
\begin{equation}\label{eq:genDescendant}
	\begin{aligned}
		\As_{(Y,D,B)}f(k,d,b) = &d\frac{k-d}{N}\,[f(k,d+1,b)+f(k,d-1,b)] \\
	&+b\frac{N-k-b}{N}(1+s(k))[f(k,d,b+1)+f(k,d,b-1)] \\
	&+(k-d)\frac{N-k-b}{N}[f(k+1,d,b) +(1+s(k))f(k-1,d,b)]\\ &+d\frac{b}{N}[f(k+1,d+1,b-1)+(1+s(k))f(k-1,d-1,b+1)]\\
	&+d\frac{N-k-b}{N}\,[f(k+1,d+1,b)+(1+s(k))f(k-1,d-1,b)]\\
	&+b\frac{k-d}{N}\,[f(k+1,d,b-1)+(1+s(k))f(k-1,d,b+1)]\\
	&+u\nu_0[d f(k-1,d-1,b+1)+ (k-d)f(k-1,b,d)]\\
	&+u\nu_1[bf(k+1,d+1,b-1)+(N-k-b)f(k+1,d,b)]\\
	&-c(k,d,b) f(k,d,b),
	\end{aligned}
\end{equation}
where $s(k) \defeq \sum_{m>0} s_m (1 - (k/N)^m )$ and $c(k,d,b)$ is such that $\As_{(Y,D,B)}1=0$. Note that $(Y,D,B)$ enters, in finite time,  either $\{(k,d,b)\in \Theta: d+b=N\}$ or  $\{(k,d,b)\in \Theta: d+b=0\}$, where it is trapped. In the first case, $b=N-k, d=k$, and in the second case $b=d=0$. Moreover, the construction via the graphical representation implies additivity in the initial condition, i.e. if the starting set is of the form $A_1\dot{\cup}A_2$, then almost surely 
\[
\{i \in [N]: \Ce_t(i) = 1,\, \Ae_t(i)\in A_1\dot{\cup}A_2\} =  \{i \in [N]: \Ce_t(i) = 1,\, \Ae_t(i)\in A_1\}  \dot{\cup}  \{i \in [N]: \Ce_t(i) = 1,\, \Ae_t(i)\in A_2\}
\]
and analogously for the fit types.

\smallskip

Before proving the connection between $(Y,D,B)$ and $\widetilde{Y}$ stated in Proposition~\ref{prop:ytildedescendantprocess}, we require the following lemma.  For $t\geq 0$ and $(k,d,b)\in\Theta$, define $$\exdesc{t}{k}{d}{b} \defeq \E [D_{t} + B_{t} \mid Y_0 = k, D_0 = d, B_0 = b].$$ 

\begin{lemma}\label{lem:lemmadesc}
	For $n\in \N_0$, $(k,d,b)\in\Theta$ and $t\geq 0$, we have  $\exdesc{t}{k}{d}{0} = \frac{d}{k}\exdesc{t}{k}{k}{0}$ for $k\in[N]$ and $\exdesc{t}{k}{0}{b} = \frac{b}{N-k}\exdesc{t}{k}{0}{N-k}$ for $k\in[N-1]_0$. Moreover, for $k\notin\{0,N\}$, 
	\[
	\exdesc{t}{k}{d}{b} = \frac{d}{k}\exdesc{t}{k}{k}{0} + \frac{b}{N-k}\exdesc{t}{k}{0}{N-k}. 
	\]
	In particular, \begin{equation}
	\exdesc{t}{k}{d}{b} = \Big(\frac{d}{k}-\frac{b}{N-k}\Big)\exdesc{t}{k}{k}{0} + \frac{Nb}{N-k}.\label{eq:lemmadesc}
\end{equation}
\end{lemma}

\begin{proof}
The proof follows from the exchangeability of starting set and site colouring, together with the additivity in the starting set. First assume $k>0$, $d\leq k$ and $b=0$. Define the specific site colouring $\alpha^{(k)}$ by setting $\alpha^{(k)}(i)=\1_{[k]}(i)$, $i\in[N]$. Using that starting set and site colouring are uniformly distributed among all  compatible ones in the first step, additivity in the starting set in the second, and exchangeability in the third, we get (here the first and second entries in the conditioning argument are the underlying site colouring and starting set, respectively)
\begin{align*}
	\exdesc{t}{k}{d}{0}=& \frac{1}{\binom{N}{d}} \sum_{\substack{A\subseteq [N] \\ \lvert A\rvert =d}}  \frac{1}{\binom{N-d}{k-d}} \sum_{\substack{\alpha:[N]\to \{0,1\},\, \lvert \alpha\rvert =k \\ \alpha(i)=1 \forall  i \in A}}  \E[D_t+B_t\mid \alpha, A]\\
	=&\, \frac{1}{\binom{N}{d}} \sum_{\substack{A\subseteq [N] \\ \lvert A\rvert =d}}  \frac{1}{\binom{N-d}{k-d}} \sum_{\substack{\alpha:[N]\to \{0,1\},\, \lvert \alpha\rvert =k \\ \alpha(i)=1 \forall i \in A}}  \ \ \sum_{i\in A}\E[D_t+B_t\mid \alpha, \{i\}] \\
	=&\, \frac{1}{\binom{N}{d}} \sum_{\substack{A\subseteq [N] \\ \lvert A\rvert =d}}  \sum_{i\in A}\E[D_t+B_t\mid \alpha^{(k)}, \{1\}] \\
	=&\, d\,\E[D_t+B_t\mid \alpha^{(k)}, \{1\}]=\frac{d}{k}\sum_{i\in[k]}\E[D_t+B_t\mid \alpha^{(k)}, \{i\}]=\frac{d}{k}\exdesc{t}{k}{k}{0}.
\end{align*}
An analogous calculation leads to $\exdesc{t}{k}{0}{b} = \frac{b}{N-k}\exdesc{t}{k}{0}{N-k}$. Next, let $k\notin\{0,N\}$, and $b,d\in[N]_0$ such that $(k,d,b)\in \Theta$. Using again that starting set and site colouring are uniformly distributed among all the compatible ones, we have
\begin{align*}
&\exdesc{t}{k}{d}{b}\\
	&=\, \frac{b!d!(N-b-d)!}{N!} \sum_{\substack{A_0\dot{\cup}A_1\subseteq [N] \\ \lvert A_0\rvert =b, \lvert A_1\rvert =d}} \frac{1}{\binom{N-(b+d)}{k-d}} \sum_{\substack{\alpha:[N]\to \{0,1\},\, \lvert \alpha\rvert =k \\ \alpha(i)=\ell \forall  i \in A_\ell,\, \ell\in \{0,1\}}}   \E[D_t+B_t\mid \alpha, A_0\dot{\cup}A_1]\\
	&=\, \frac{b!d!(N-b-d)!}{N!} \sum_{\substack{A_0\dot{\cup}A_1\subseteq [N] \\ \lvert A_0\rvert =b, \lvert A_1\rvert =d}}  \! \! \!  \frac{1}{\binom{N-(b+d)}{k-d}} \sum_{\substack{\alpha:[N]\to \{0,1\},\, \lvert \alpha\rvert =k \\ \alpha(i)=\ell \forall i \in A_\ell,\, \ell\in \{0,1\}}}   \Big(\sum_{i\in A_0}\E[D_t+B_t\mid \alpha, \{i\}]+\sum_{i\in A_1}\E[D_t+B_t\mid \alpha, \{i\}]\Big)\\
	& =\,  \big(d\, \E[D_t+B_t\mid \alpha^{(k)}, \{1\}]+b\, \E[D_t+B_t\mid \alpha^{(k)}, \{N\}]\big)\\
	& =\frac{d}{k}\exdesc{t}{k}{k}{0} + \frac{b}{N-k}\exdesc{t}{k}{0}{N-k}.
\end{align*}

Finally, \eqref{eq:lemmadesc} follows after noting that $\exdesc{t}{k}{0}{N-k}=N-\exdesc{t}{k}{k}{0}$.
\end{proof}

\begin{proof}[Proof of Proposition~\ref{prop:ytildedescendantprocess}] 
	Recall that $\widetilde{Y}$ is defined via the graphical representation. In particular, we can use the arrival times $(\tau_i)_{i \in \mathbb{N}}$ of the underlying Poisson process (see~\eqref{eq:poissonprocesses}) as the (common) transition times for both $Y$ and $\widetilde Y$. Note that some events may be silent for either one or both  processes. We claim that  for all $k\in [N]_0$, $n \geq 0$, $M_{k,k,0}^{\tau_n}=\E[\widetilde{Y}_{\tau_n}\mid \widetilde{Y}_0=k]$. If this  is true, then 
	\begin{align*}
		&\E[D_t+B_t\mid Y_0=k,\, D_0=k, B_0=0]\\
		&=\sum_{n\geq 0 } \P(\tau_{n}< t, \tau_{n+1}\geq t) \E[D_t+B_t\mid Y_0=k,\, D_0=k, B_0=0,\tau_{n}< t,\tau_{n+1}\geq t]\\
		&=\sum_{n\geq0 } \P(\tau_{n}< t,\tau_{n+1}\geq t) \E[\widetilde{Y}_{\tau_{n}} \mid \widetilde{Y}_0=k, \tau_{n}< t, \tau_{n+1}\geq t]=\E[\widetilde{Y}_{t} \mid \widetilde{Y}_0=k],
	\end{align*}
thus proving the proposition. It remains to prove the claim. Note first that the claim is true if $k\in \{0,N\}$, because if $\widetilde{Y}_r=D_r+B_r=0$  (or $=N$) for some $r\geq 0$, then $\widetilde{Y}_t=D_t+B_t=0$ (or $=N$) for all $t\geq r$. So it remains to prove the claim for $k\in[N-1]$. We proceed by induction on the number of arrivals~$n$ in the underlying Poisson process. To ease notation, we write $\tau\defeq \tau_1$. Using the Poissonian families, we introduce the  events
	\begin{align*}
		E^\times &\defeq \{Y_{\tau+} \neq Y_{\tau}, \,  \tau\in \Pi_i^{\times} \text{ for some } i \}, && E^\circ \defeq \{Y_{\tau+} \neq Y_{\tau}, \,  \tau\in \Pi_i^{\circ}\text{ for some } i\} ,\\
		E^{+} &\defeq \{Y_{\tau+} = Y_{\tau} +1, \,  \tau\in \Pi_{i,j}^{\narrow}  \text{ for some } i,j \},&&   E^{-} \defeq \{Y_{\tau+} = Y_{\tau}-1, \,  \tau\in \Pi_{i,j}^{\narrow}\cup \Pi_{J,i}^{\sarrow}\text{ for some } i,j,J\}, 
	\end{align*}  i.e. respectively the event of a type-changing  deleterious or beneficial mutation; or the reproduction of an unfit or fit individual that  changes~$Y$. For $k\in [N]_0$, write $\P_k$ for the probability measure under $Y_0=k$. It is straightforward from the arrival rates of the Poisson process that 
\begin{align*}\label{eq:transitionsdescproc}
	&\P_k(E^\times)=\frac{1}{R}u\nu_1(N-k),  \qquad \P_k(E^\circ)=\frac{1}{R}u\nu_0k,   \\
	&\P_k(E^{+})=\frac{1}{R} k\frac{N-k}{N} ,\qquad \P_k(E^{-})=\frac{1}{R} k\Big ( \frac{N-k}{N}+s(k) \Big ),\\
	&\P_k((E^\times \cup E^\circ \cup E^{-} \cup E^{+})^C)=\frac{1}{R}\Big( u\nu_1k+u\nu_0(N-k)+\frac{k^2}{N}+\frac{(N-k)^2}{N}+ s(k)\frac{N-k}{N} \Big)
\end{align*} 
	where $R \defeq N(1 + u + \sum_{m\geq 1} s_m)$ is the total rate of $\Lambda$ and superscript~$C$ indicates the complement. Then,
	\begin{align*}
	&\E[D_\tau + B_\tau \mid Y_0 = k = D_0, B_0 = 0] \\
	&= (1-\P_k(E^+) - \P_k(E^-))\, k+ \P_k(E^{+}) \, (k+1) + \P_k(E^{-})\, (k-1) \\
	&= \frac{1}{R}u\nu_1(N-k) \Big( \frac{k}{k+1}(k+1)+ \frac{1}{k+1} 0\Big) + \frac{1}{R} u\nu_0k \Big( \frac{N-k}{N-k+1}(k-1) + \frac{1}{N-k+1} \ N \Big) \\
	&\textcolor{white}{=}  + \P_k((E^\times \cup E^\circ \cup E^{+} \cup E^{-})^C)\ k + \P_k(E^{+}) \, (k+1) + \P_k(E^{-})\, (k-1) \\
	&= \E [\widetilde{Y}_\tau \mid \widetilde{Y}_0 = k],
\end{align*}
where we have used in the last step that  $\frac{N-k}{N-k+1}(k-1) + \frac{1}{N-k+1} \ N = k$ (in line with the fact that $B_\tau+D_\tau$ do not change upon mutation), together with \eqref{generatorYtilde}.
 
For the inductive step we first use a first-step-decomposition and the Markov property of $(Y,D,B)$; second, Lemma~\ref{lem:lemmadesc}; and third, the inductive hypothesis, to obtain
\begin{align*}
	&\E [D_{\tau_{n+1}} + B_{\tau_{n+1}} \mid Y_0 = k = D_0, B_0 = 0] \\
	%&=\ \E \big[ \E[D_{\tau_n} + B_{\tau_n} \mid Y_\tau, D_\tau, B_\tau] \mid Y_0 = k = D_0, B_0 = 0  \big] \\
	&=\ \P_k (E^{\times}) \exdesc{\tau_{n}}{k+1}{k}{0} + \P_k (E^{\circ}) \exdesc{\tau_{n}}{k-1}{k-1}{1}  + \P_k (E^{+}) \exdesc{\tau_{n}}{k+1}{k+1}{0} + \P_k (E^{-}) \exdesc{\tau_{n}}{k-1}{k-1}{0} \\
	&\ \  + \P_k((E^\times \cup E^\circ \cup E^{+} \cup E^{-})^C) \exdesc{\tau_{n}}{k}{k}{0}\\
	&=\ \P_k (E^{\times}) \frac{k}{k+1}\exdesc{\tau_{n}}{k+1}{k+1}{0} + \P_k (E^{\circ}) \Big( \frac{N-k}{N-k+1} \exdesc{\tau_{n}}{k-1}{k-1}{0} + \frac{N}{N-k+1} \Big)  + \P_k (E^{+}) \exdesc{\tau_{n}}{k+1}{k+1}{0} \\
	&\ + \P_k (E^{-}) \exdesc{\tau_{n}}{k-1}{k-1}{0} + \P_k((E^\times \cup E^\circ \cup E^{+} \cup E^{-})^C) \exdesc{\tau_{n}}{k}{k}{0}\\
	&=\ \P_k (E^{\times}) \frac{k}{k+1} \E[\widetilde{Y}_{\tau_n}\mid \widetilde{Y}_0=k+1] + \P_k (E^{\circ}) \Big(\frac{N-k}{N-k+1}\E[\widetilde{Y}_{\tau_n}\mid \widetilde{Y}_0=k-1]+ \frac{N}{N-k+1}\Big) \\
	& \ + \P_k (E^{+}) \E[\widetilde{Y}_{\tau_n}\mid \widetilde{Y}_0=k+1] + \P_k (E^{-}) \E[\widetilde{Y}_{\tau_n}\mid \widetilde{Y}_0=k-1]+ \P_k((E^\times \cup E^\circ \cup E^{+} \cup E^{-})^C) \E[\widetilde{Y}_{\tau_n}\mid \widetilde{Y}_0=k]\\
	&= \E[ \widetilde{Y}_{\tau_{n+1}}\mid \widetilde{Y}_0=k].
\end{align*}
and where we used in the last step the Markov property and a first-step-decomposition of $\widetilde{Y}$.
\end{proof}

\section{Details of the results on the diffusion limit}\label{sec:diffusionlimit}
In this section, we prove the convergence results as $N\to\infty$, and how the dualities translate to this limit. Throughout the section, we indicate that processes or parameters depend on the population size~$N\in \N$ via a superscript, i.e. we write $Y= Y^{(N)}$, $R = R^{(N)}$, $L=L^{(N)}$, $u=u^{(N)}$, and for all $m\in \N$, $s_m=s_m^{(N)}$; while we assume $\nu_0$ and $\nu_1$ to be independent of $N$. Recall the assumptions on the parameters in~\eqref{eq:assweaklimits}. 

\smallskip

The following corresponds to the (mostly standard) proof of the convergence of the MoMo to the Wright--Fisher diffusion.
\begin{proof}[Proof of Proposition~\ref{prop:convergenceWF}]
	 We first prove uniform convergence of the generators applied to polynomials on $[0,1]$. Because such polynomials form a core of $\mathcal{A}_{\Ys}$, and $\Ys$ is Feller~\citep[Thm. 8.2.8]{Ku86}, the result follows from~\citep[Thm 1.6.1 and Thm. 4.2.11]{Ku86}.
	
	\smallskip

	For all $N \in \mathbb{N}$, denote by ${\mathcal{A}}_{\bar{Y}^{(N)}}$ the generator of $\bar{Y}^{(N)}$ (more precisely, its c{\`a}dl{\`a}g version, which exists because $\bar{Y}^{(N)}$ is Feller). It follows from the definition of $\bar{Y}$ that for $\bar{f}:[N]_0/N\to \R$, one has ${\mathcal{A}}_{\bar{Y}^{(N)}} \bar{f}(\frac{k}{N})= N(\cA_{Y^{(N)}}^{\mathrm{n}} + \sum_{m > 0} \cA_{Y^{(N)}}^{s_m} + \cA_{Y^{(N)}}^{\nu_0} + \cA_{Y^{(N)}}^{\nu_1})\bar{f}_N(k)$, where $\bar{f}_N(k)=\bar{f}(k/N)$.
	Let $\varphi$ be a polynomial on~$[0,1]$. Proving uniform convergence of the generators applied to~$\varphi$ means we show
	\begin{equation}
	\sup_{k \in [N]_0} \Big \lvert  {\mathcal{A}}_{\bar{Y}^{(N)}} \varphi \Big (\frac{k}{N}\Big ) - \cA_{\Ys} \varphi\Big (\frac{k}{N}\Big ) \Big \rvert \xrightarrow[]{N \to \infty} 0 \quad \text{for} \; k \in [N].
	\label{eq:limconvergence}
	\end{equation}
	We estimate the difference by splitting ${\mathcal{A}}_{\bar{Y}^{(N)}}$ and $\mathcal{A}_{\Ys}$ into their components: starting with the neutral part, a Taylor approximation of $\varphi$ gives for all $k\in [N]_0$ that	\begin{align*}
	\left\lvert N\cA^{\mathrm{n}}_{Y^{(N)}} \varphi\left(\frac{k}{N}\right) - \cA^{\mathrm{n}}_{\Ys} \varphi \left(\frac{k}{N}\right) \right\rvert & \leq \frac{\lVert \varphi^{(3)}\rVert_{\infty}}{3N}\xrightarrow{N\to\infty}0. 
	\end{align*}
	Proceeding in a similar fashion for the selective term, for every $k \in [N]_0$ we have \begin{align*}
	\left\lvert \sum_{m >0} N\cA^{s_m^{(N)}}_{Y^{(N)}} \varphi\left(\frac{k}{N}\right) - \sum_{m >0} \cA^{\sigma_m}_{\Ys} \varphi\left(\frac{k}{N}\right) \right\vert & \leq \lVert \varphi' \rVert_{\infty} \left\lvert \sum_{m>0} \big (  Ns_m^{(N)} - \sigma_m  \big ) \right\rvert + \frac{\lVert \varphi''\rVert_{\infty}}{2N} \sum_{m>0} N s_m^{(N)},  
	\end{align*}
	and both summands on the last line tend to $0$ as $N \to \infty$ under~\eqref{eq:assweaklimits}
	Calculations for the two remaining terms are completely analogous. \eqref{eq:limconvergence} then comes by using the triangle inequality, concluding the proof.

\end{proof}

\subsection{Proofs connected to the kASG in the diffusion limit}\label{sec:kASGdiff}
First, we show that~$\mathcal{R}$ indeed arises from $R^N$ as $N\to\infty$. Next, we prove that the factorial moment duality translates into a moment duality in the limit.

\smallskip 

Recall the definition of $\mathcal{R} = (\mathcal{R}_r)_{r \geq 0}$ in \eqref{eq:R}. $\mathcal{R}$ is well defined. Indeed, let us initially understand~\eqref{eq:R} as nothing but a convenient way of stating the rates. It is then clear that $\Rs$ can be coupled to a branching process where each particle branches into $m+1$ independently at rate $\sigma_m$ ($m\in \N$) so that the branching process almost surely dominates $\Rs$. It is well known that $\sum_{m=1}^\infty \sigma_m m<\infty$ implies that this branching process is non-explosive~\citep[Ch. V, Thm. 9.1]{harris1964theory}. In particular, \begin{equation}
	\Rs\text{ is non-explosive}, \label{eq:kASGnon-explosive}
\end{equation}
so~\eqref{eq:R} indeed gives rise to a unique Markov process.
We now prepare for the proof of $\bar{R}^{(N)}\xRightarrow[]{N\to\infty} \Rs$  in distribution. The proof is carried out  via a localisation argument. To this end we require convergence of the processes stopped when they escape some finite level, or absorb. The following lemma plays a key role for this.

\begin{lem}\label{lem:kASGratesconvergence}
	Let $f:\N_{0,\Delta} \to \R$ be bounded 
	and, for $N\in \N$, let $f^{\lvert N}$ be the restriction of $f$ to $[N]_{0,\Delta}$. Then, for all~$k\in \N$, $$\sup_{n\in [k]}\lvert N\As_{R^{(N)}}f^{\lvert N}(n)-\As_{\Rs}f(n)\lvert\xrightarrow{N\to\infty} 0\qquad (N\geq k).$$
\end{lem}

\begin{proof}
	Recall the definition of $\As_{R^{(N)}}^{\rm{n}}$, $\As_{R^{(N)}}^{s_m}$, $\As_{R^{(N)}}^{u}$ in \eqref{eq:generatorRcoalescence}, \eqref{eq:generatorRselection}, and \eqref{eq:generatorRmutation}, respectively.
	 For a bounded function $f:\N_{0,\Delta} \to\R$ and $k\in \N$, it is straightforward to show that for all $n\in [k]$, $$\lvert N(\As_{R^{(N)}}^{\rm{n}}+\As_{R^{(N)}}^{u})f^{\lvert N}(n) -(\As_{\Rs}^{\rm{n}}+\As_{\Rs}^{u})f(n)\lvert\xrightarrow{N\to\infty}0 \qquad (N\geq k).$$
	It remains to show the statement for the parts corresponding to selection. For $n\in[k]$,
	\begin{align*}
	&\Big\lvert \sum_{m>0}N\As_{R^{(N)}}^{s_m}f^{\lvert N}(n)-\sum_{m>0}\As_{\Rs}^{\sigma_m}f(n)\Big\rvert\\
	&\leq \Big\lvert n\sum_{m>0} \Big(Ns_m^{(N)} \frac{\faf{(N-n)}{m}}{N^m}-\sigma_m\Big)[f(n+m)-f(n)]\Big\rvert + \Big\lvert n\sum_{m> 0} Ns_m^{(N)} \sum_{j=1}^{m-1} \frac{\faf{(N-n)}{\,j}}{N^m} C_{mj}^n [f(n+j)-f(n)] \Big\rvert.
	\end{align*}
	Denote the first and second absolute value by $E^1_{N,n}$ and $E^2_{N,n}$, respectively. We first deal with $E^1_{N,n}$. Fix $\varepsilon>0$. Since $\sum_{m>0}\sigma_m<\infty$, there is $N_1=N_1(\varepsilon)$ s.t. $\sum_{m>N_1}\sigma_m<\varepsilon/2$. 
	Hence,
	$$E^1_{N,n}\leq 2k\lVert f\rVert_{\infty}^{} \bigg(N_1\max_{m\in[N_1]}\Big\{ \Big\lvert Ns_m^{(N)} \frac{\faf{(N-n)}{m}}{N^m}-\sigma_m\Big\rvert\Big\}+\sum_{m>N_1}Ns_m^{(N)}+ \frac{\varepsilon}{2}\bigg).$$
	Under \eqref{eq:assweaklimits}, we have $\lvert \sum_{m> N_1} Ns_m^{(N)}-\sum_{m> N_1}\sigma_m\rvert \xrightarrow{N\to\infty}0$. In particular, $ \lim_{N\to\infty}\sum_{m>N_1}Ns_m^{(N)}<\varepsilon/2$. 
	Moreover, for every $n\leq k$, $\lvert Ns_m^{(N)} \frac{\faf{(N-n)}{m}}{N^m}-\sigma_m\rvert\xrightarrow{N\to\infty}0$. Hence, $\lim\sup_{N\to\infty} E_{N,n}^1\leq 2k\lVert f\rVert_{\infty}^{}\varepsilon.$ Since~$\varepsilon$ was arbitrary, we deduce $\lim_{N\to\infty} E^1_{N,n}=0$. Next, we consider $E^2_{N,n}$. Fix again $\varepsilon>0$ and let $N_1(\varepsilon)$ as before. Recall the calculation from~\eqref{eq:aux2} to see that $\sum_{j=1}^{m-1}C_{mj}^n \faf{(N-n)}{\,j}= N^{m}-n^{m}-\faf{(N-n)}{m}.$ Altogether,  
	\begin{align*}
	E^2_{N,n}&\leq 2k\lVert f\rVert_{\infty}^{} \sum_{m>0}^{}N s_m^{(N)} \Big\lvert 1-\frac{n^m}{N^m}-\frac{\faf{(N-n)}{\,m}}{N^m} \Big\rvert\\
	&\leq 2k\lVert f\rVert_{\infty}^{} \Big( N_1 \max_{m\in [N_1]} \Big\{N s_m^{(N)} \big\lvert 1-\frac{n^m}{N^m}-\frac{\faf{(N-n)}{\,m}}{N^m} \big\rvert\Big\} +2\sum_{m>N_1}Ns_m^{(N)}\Big).
	\end{align*}
	Since $\lvert 1-\frac{n^m}{N^m}-\frac{\faf{(N-n)}{\,m}}{N^m}\rvert\xrightarrow{N\to\infty}0$, also $\lim\sup_{N\to\infty}E^2_{N,n}\leq 2k\lVert f\rVert_{\infty}^{} \varepsilon$. Since $\varepsilon$ was again arbitrary, $E^2_{N,n}\xrightarrow{N\to\infty} 0$. Altogether we showed that for all $n\in[k]$, $\lim_{N\to\infty}(E^1_{N,n}+E^2_{N,n})=0$.
\end{proof}
Before we prove~Proposition~\ref{prop:convergencekASG}, we recall the Skorokhod topology in our setting, see also~\citep[App.~A]{Cordero17DL} or, more generally, \citep[Ch.~$3$]{Billingsley1999}. For each $t\in (0,\infty)$, let $\Db_{\R}[0,t]\coloneqq \{\omega:[0,t]\to\R,\ \omega\ \text{c{\`a}dl{\`a}g} \}$ and $\Cb_t^{\uparrow}$ denote the class of strictly increasing, continuous mappings of $[0,t]$ onto itself, and set $\lVert \lambda\rVert^{\circ}\coloneqq \sup_{0\leq r<s\leq t} \lvert \log(\frac{\lambda(s)-\lambda(r)}{s-r})\rvert.$ Then for $f,g\in \Db_{\R}[0,t]$, $$d_t^{\circ}(f,g)\coloneqq \inf_{\lambda\in \Cb_t^{\uparrow}}\{\lVert \lambda\rVert^{\circ} \vee \sup_{s\in [0,t]}\rvert f(s)- g\circ \lambda(s) \rvert\}$$ defines a metric that induces the Skorokhod topology on~$\Db_{\R}[0,t]$. Let $\Db_{\R}[0,\infty)\coloneqq\{\omega:[0,\infty)\to\R,\ \omega\ \text{c{\`a}dl{\`a}g} \}$ and $\Cb_\infty^{\uparrow}$ denote the class of strictly increasing, continuous mappings of $[0,\infty)$ onto itself. For $m\in \N$, set $\eta_m(s)=\1_{[0,m-1)}(s)+\1_{[m-1,m]}(s)\,(m-s)$ and define $$d_\infty^{\circ}(f,g)\coloneqq \sum_{m=1}^\infty \frac{1}{2^m} \big (1\wedge d_m^\circ(\eta_m\,f, \eta_m\,g) \big ).$$ Then $d_\infty^\circ$ defines a metric that induces the Skorokhod topology on~$\Db_{\R}[0,\infty)$. Consider the Euclidean metric on $\N_{0,\Delta}$ by letting $\Delta$ take the role of $-1$. Let $\Db_{\N_{0,\Delta}}[0,\infty)\coloneqq \{\omega\in \Db_{\R}[0,\infty): \omega_t\in \N_{0,\Delta}, \forall t\geq 0\}$ with metric $d_{\infty}^\circ$. Note that $\Db_{\N_{0,\Delta}}[0,\infty)$ is a closed subspace of~$\Db_{\R}[0,\infty)$.
	
\begin{proof}[Proof of Proposition~\ref{prop:convergencekASG}]
	We follow the line of argument of~\citet[Prop.~5.3]{Cordero17DL}. 
	By the Portmanteau theorem, it suffices to show that for any uniformly continuous, bounded function~$F:\Db_{\N_{0,\Delta}}[0,\infty) \to \R$, we have $\E[F(\bar{R}^{(N)})]\to \E[F(\Rs)]$ as $N\to\infty$.
	
	\smallskip
	
	For $k\in \N$ and $\omega\in \Db_{\N_{0,\Delta}}[0,\infty)$, define $T_k(\omega)\coloneqq \inf\{t\geq 0: \omega_t \in \{0, \Delta\} \cup [k, \infty) \}$, with the usual convention $\inf \varnothing =\infty$. 
	At first, we assume $\bar{R}^{(N)}_0$ and $\Rs_0$ share the same bounded support, i.e. there is $k \in \N$ such that $\P(\bar{R}^{(N)}_0 > k) = \P(\Rs_0 > k) = 0$. 
	
	Define $\bar{R}^{(N)}(k)\coloneqq (\bar{R}^{(N)}_{t\wedge T_k(\bar{R}^{(N)})})_{t\geq 0}$ and $\Rs(k)\coloneqq (\Rs_{t\wedge T_k(\Rs)})_{t\geq 0}$, i.e. the processes stopped the first time they hit  $\{0, \Delta\} \cup [k,\infty)$. Note that since the processes can reach the boundary before any finite time~$t$ with positive probability,  we can, by virtue of the Borel-Cantelli lemma and the Markov property, deduce that $T_k(\Rs(k))$ and $T_k(\bar{R}^{(N)}(k))$ are almost surely finite. Moreover, $T_k(\Rs(k))=T_k(\Rs)$ and $T_k(\bar{R}^{(N)}(k))=T_k(\bar{R}^{(N)})$. 
	\smallskip 
	
	Fix any uniformly continuous bounded function~$F$. Then, 
	$$\lvert \E[F(\bar{R}^{(N)})]-\E[F(\Rs)]\rvert\leq \lvert \E[F(\bar{R}^{(N)})-F(\bar{R}^{(N)}(k))]\rvert +\lvert \E[F(\bar{R}^{(N)}(k))]-\E[F(\Rs(k))]\rvert+\lvert \E[F(\Rs)-F(\Rs(k))]\rvert.$$ 
	Denote the first, second, and third summand on the right-hand side  by $D_1^N$, $D_2^N$, and $D_3^N$, respectively. We first deal with $D_2^N$. By Lemma~\ref{lem:kASGratesconvergence}, we have for any bounded function~$f:\N_{0,\Delta} \to\R$ that $\lim_{N\to\infty} \sup_{n\in[k]}\lvert N\As_{R^{(N)}(k)}f^{\lvert N}(n)- \As_{\Rs(k)}f(n)\rvert=0$, where $N\As_{R^{(N)}(k)}$ and $\As_{\Rs(k)}$ are the infinitesimal generators of $\bar{R}^{(N)}(k)$ and $\Rs(k)$, respectively. Since the closure of $\As_{\Rs(k)}$ generates a Feller semigroup and the core of $\As_{\Rs(k)}$ is contained in the bounded functions (e.g.~\citep[Cor.~8.3.2]{Ku86}), using \citep[Thms. 1.6.1 and 4.2.11]{Ku86}, we deduce that $\bar{R}^{(N)}(k)$ converges to $\Rs(k)$. In particular, $\lim_{N\to\infty}D_2^N=0$.
	
	\smallskip
	
	Next, we deal with $D_1^N$ and $D_3^N$. Because $F$ is uniformly continuous, for all $\varepsilon>0$ there is $n_F\in \N_0$ such that for all $\omega,\omega'\in \Db_{\N_{0,\Delta}}[0,\infty)$ with $d_{\infty}^\circ(\omega,\omega')\leq 2^{-n_F}$, we have $\lvert F(\omega)-F(\omega')\rvert \leq\varepsilon.$ One can show that for $\omega\in \Db_{\N_{0,\Delta}}[0,\infty)$ and $t\geq 0$, $d_\infty^\circ(\omega,\omega(\cdot\wedge t))\leq 2^{-\lfloor t\rfloor}$ (see \citep[Lem. A.1]{Cordero17DL}). Because of this and because $0$ and $\Delta$ are absorbing states for $\bar{R}^{(N)}$, 	\begin{align*}
	D_1^N&=\E[\lvert F(\bar{R}^{(N)})-F(\bar{R}^{(N)}(k))\rvert]\\
	&\leq 2\, \lVert F\rVert_\infty \P \big ( T_k(\bar{R}^{(N)}(k) \big )\leq n_F, \bar{R}^{(N)}_{n_F}(k) \in [k,\infty)) +\varepsilon\, \P \big (T_k(\bar{R}^{(N)}(k))>n_F \big ).
	\end{align*}
	A similar argument leads to $D_3^N\leq 2\, \lVert F\rVert_\infty\P(T_k(\Rs(k))\leq n_F, \Rs_{n_F}(k) \in [k,\infty) )+ \varepsilon\,\P(T_k(\Rs)>n_F).$ 
	
	\smallskip 
	
	By adapting the proof of~\citep[Lem. A.2]{Cordero17DL}, it can be shown that $T_k$ is continuous on the set $\{\omega\in \Db_{\N_{0,\Delta}}[0,\infty):T_k(\omega)<\infty\}$. Because $\bar{R}^{(N)}(k)$ and $\Rs(k)$ belong to this set almost surely, and  $\bar{R}^{(N)} (k)\overset{N\to\infty}{\longrightarrow}\Rs(k)$ in distribution, it follows from the continuous mapping theorem that also $T_k ( \bar{R}^{(N)}(k)) \overset{N\to\infty}{\longrightarrow} T_k(\Rs(k))$ and $(T_k(\bar{R}^N(k)), \bar{R}^{N}(k))\overset{N\to\infty}{\longrightarrow}(T_k(\Rs(k)), \Rs(k))$ in distribution.
	
	\smallskip
	
	Moreover, $\{\omega:\ T_k(\omega)\leq n_F, \omega^{}_{n_F}\in [k,\infty)\}$ is a continuity set. To see this, note that \begin{equation*}
		\partial\{\omega:\ T_k (\omega) \leq n_F,\ \omega_{n_F} \in[k,\infty)\} \subseteq \partial\{\omega :\ T_k(\omega) \leq n_F \}\cup \bigcup_{j\in[k,\infty)} \partial \{\omega :\ \omega_{n_F} = j\}.
	\end{equation*} The projection $\omega\mapsto \omega_{n_F}$ is continuous at $\omega$ if and only if $\omega$ is continuous at $n_F$ (see \citep[Thm. 16.6]{Billingsley1999}). This implies that $\partial\{\omega:\ \omega_{n_F} = j\} \subseteq \{\omega:\ \omega \text{ discontinuous at }n_F \}$. Therefore, by subadditivity
	\begin{equation*}
		\P (\Rs (k) \in \partial\{\omega:\ T_k (\omega) \leq n_F,\ \omega_{n_F}\in[k,\infty) \}) \leq \P ( \Rs(k) \text{ is discontinuous in } n_F) + \P ( T_k (\Rs (k) ) = n_F) = 0.
	\end{equation*}
	We can now use again the Portmanteau lemma to conclude that $$\P(T_k(\bar{R}^{(N)}(k))\leq n_F, \bar{R}^{(N)}_{n_F}(k) \in [k,\infty)) \overset{N \to \infty}{\longrightarrow} \P(T_k(\Rs(k))\leq n_F, \Rs_{n_F}(k) \in [k,\infty) ).$$
	
	\smallskip 
	
	Combining everything, we have
	$$\limsup_{N\to\infty}(D_1^N+D_3^N)\leq 2\varepsilon+ 4\,\lVert F\rVert_{\infty}^{}\P(T_k(\Rs(k))\leq n_F, \Rs_{n_F}(k) \in [k,\infty) ).$$
	$\Rs$ is non-explosive by~\eqref{eq:pLDnon-explosive} so that $\lim_{k\to\infty}\P(T_k(\Rs(k))\leq n_F, \Rs_{n_F}(k) \in [k,\infty) )=0$. Altogether, for all $\varepsilon>0$, $\limsup_{N\to\infty}\lvert \E[F(\bar{R}^{(N)})]-\E[F(\Rs)]\rvert\leq 2\varepsilon$ which proves the result.
	
	\smallskip 
	
	It remains to address the case in which $\Rs_0$ does not almost surely have bounded support. In this case, for all $\epsilon>0$ there exists $M \in \N$ such that $\P(\Rs_0 > M) < \epsilon$. Moreover, since $\bar{R}^{(N)}_0\xrightarrow[N\to\infty]{(d)} \Rs_0$, there exists $N_{\epsilon}\in \N$ such that for $N > N_{\epsilon}$, $\P(\bar{R}^{(N)}_0 > M) < \epsilon$. It follows that, for~$N > N_{\epsilon}$,
	\[
	\lvert \E [F(\bar{R}^{(N)})] - \E[F(\Rs)]\rvert  \leq \lvert \check{\E}[F(\bar{R}^{(N)})]  - \check{\E}[ F(\Rs) ]\rvert  + 2 \epsilon \lVert F \rVert_\infty
	\]	
	where $\check{\E}$ denotes the expectation under the original measure conditional on $\{\bar{R}_0^{(N)}\leq M\}$ and $\{\Rs_0 \leq M \}$, respectively. Applying the previous argument to the first summand yields $\limsup_{N\to\infty} \lvert \E [F(\bar{R}^{(N)})] - \E[F(\Rs)]\rvert \leq 2 \epsilon (1+\lVert F \rVert_\infty)$. The result follows.	
	
\end{proof}

Next, we prove that the factorial moment duality between the $R^{(N)}$ and $Y^{(N)}$ in Theorem~\ref{thm:duality_R} carries over to the diffusion limit in the form of a moment duality.

\begin{proof}[Proof of Theorem~\ref{thm:momentduality}]
	While a generator-based proof like for Theorems~\ref{thm:duality_R} and \ref{lem:bdduality} (or~\ref{thm:tildemomentduality} below) is possible, for the sake of variety we prove the theorem using the duality between $Y^{(N)}$ and $R^{(N)}$, and the convergence of both processes in the diffusion limit.
	Using the factorial moment duality~\eqref{eq:factmomentdualityYR} in the finite setting, we have for any $k \in [N]_0$ and $n \in [N]_{0,\Delta}$	\begin{equation}\label{duality_rescaled}
		\mathbb{E} \Bigg[ \frac{\faf{(Y_{Nt}^{(N)})}{n}}{N^{\underline{n}}} \mid  Y^{(N)}_0 = k \Bigg ] = 	\E \Bigg [\frac{{ k}^{\underline{ {R^{(N)}_{Nt}} } }}{N^{\underline{ {R^{(N)}_{Nt}} } }} \mid R^{(N)}_0 = n \Bigg]. 
	\end{equation}
	
	We show that there exists a sequence $(k_N)_{N\in \N}$ such that $k_N/N \to y$ as $N\to\infty$, and each  side of \eqref{duality_rescaled} converges to the respective side of \eqref{eq:momentdualityYR} as $N \to \infty$. 
	Starting from the left-hand side,  note first that for $n = 0, \Delta$, the convergence is trivial. Fix then $n \in \mathbb{N}$, and set for $y\in [0,1]$, $$F_n(y) \defeq y^n\quad \text{and for }N\geq n,\quad F^N_n(y) \defeq \frac{\faf{(Ny)}{n}}{\faf{N}{n}} = \prod_{j=0}^{n-1} \frac{N y- j}{N-j}.$$ Recall that $\bar{Y}^{(N)}_t \defeq Y^{(N)}_{Nt}/N$. We have
	\begin{align*}
		\Big \lvert \E&[F_n(\Ys_t) \mid \Ys_0 = y] - \E \Big[ F_n^N\Big (\frac{Y_{Nt}^{(N)}}{N} \Big ) \mid Y^{(N)}_0 = k_N \Big] \Big \rvert \\
		&= \Big\lvert \E[F_n(\Ys_t) \mid \Ys_0 = y] - \E [F_n^N (\bar{Y}_t^{(N)}) \mid \bar{Y}^{(N)}_0 = \frac{k_N}{N}] \Big \rvert \\
		&\leq  \Big \lvert \E[F_n(\Ys_t) \mid \Ys_0 = y] - \E \Big[ F_n(\bar{Y}^{(N)}_t) \mid \bar{Y}^{(N)}_0 = \frac{k_N}{N} \Big] \Big \rvert \\
		& \ \  + \Big \lvert \E \Big[ F_n(\bar{Y}^{(N)}_t) \mid \bar{Y}^{(N)}_0 = \frac{k_N}{N} \Big] - \mathbb{E} \Big[ F_n^N(\bar{Y}^{(N)}_t) \mid  {\bar{Y}^{(N)}}_0 = \frac{k_N}{N} \Big] \Big \rvert.
	\end{align*}
	Convergence of the first summand as $N\to\infty$ comes directly from Proposition~\ref{prop:convergenceWF} (Portmanteau lemma, $F_n$ is bounded and continuous), for any $(k_N)_{N\in \N}$ such that $k_N\to y$. For the second summand, it suffices to show that $F_n^N \to F_n$ uniformly in $[0,1]$. Because $F_n^N(y) \xrightarrow{N \to \infty} F_n(y)$ for all $y \in [0,1]$, the uniform convergence follows from Dini's theorem if we can show $F_n^{N+1} (y) \geq F_n^N (y)$ for all $y \in [0,1]$. To this end, note that for any $y \in [0,1]$, $N \in \mathbb{N}$, and $j\in[N-1]_0$, one has $((N+1)y - j)/((N+1)-j)\geq (Ny-j)/(N-j)$; in particular, for $N\geq n$ and $j\in [n-1]_0$. Hence, $F_n^N(y)\to F_n(y)$ uniformly.
	
	\smallskip

	As for the right-hand side of \eqref{duality_rescaled}, name now $G_y(n) \defeq y^n$, and for $N\geq n$, $G_{k_N}^N (n) \defeq \faf{(k_N)}{n}/\faf{N}{n}$ for $n \in \mathbb{N}_{0, \Delta}$. We have
	\begin{align*}
\big\lvert \E [G_y(\Rs_t) \mid \Rs_0 = n] & - \E [G_{k_N}^N (R^{(N)}_{Nt}) \mid R^{(N)}_0 = n] \big\lvert  \\
&\leq \big\lvert \E [G_y(\Rs_t) \mid \Rs_0 = n] - \E [G_y(R_{Nt}^{(N)}) \mid R_0^{(N)} = n ] \big\rvert \\
		&\textcolor{white}{\leq} + \big\lvert  \E [G_y(R_{Nt}^{(N)}) \mid R_0^{(N)} = n ] - \E [G_{k_N}^N(R^{(N)}_{Nt}) \mid R^{(N)}_0 = n] \big\lvert. 
	\end{align*}
	
	Again, for every $y \in [0,1]$, $G_y$ is bounded in $n$, so the first summand converges by the Portmanteau lemma and because the line-counting process of the ASG converges. As for the second, it is enough to show that $G_{k_N}^N$ converges uniformly to $G_y$. If $y = 1$, taking $k_N = N$ makes this trivial.
	If $y < 1$, take any sequence $(k_N)_{N\in \N}$ such that $k_N/N \to y$. Fix $\epsilon>0$ small. Note that (i) $G_y(n)\to 0$ as $n\to\infty$. Hence, there is $n'=n'(y)\in \N$ such that for all $n\geq n'$, $G_y(n)\leq \varepsilon/3$. Next (ii), because for fixed $\tilde{n}$, $G_{k_N}^N(\widetilde{n})\to G_y(\tilde{n})$ as $N\to\infty$, there is $N'(\tilde{n})\in \N$ such that for all $N\geq N'(\tilde{n})$, $\lvert G_{k_N}^N(\tilde{n})-G_y(\tilde{n})\rvert \leq \varepsilon/3$. In particular, combining (i) and (ii) yields (iii) that for $n'$, there is $N'(n')$ such that for all $N\geq N'(n')$, $\lvert G_{k_N}^N(n')\rvert \leq \lvert G_{k_N}^N(n')-G_y(n')\rvert+\lvert G_y(n')\rvert\leq 2\varepsilon/3$. Note that (iv), $G_{k_N}^N(n)\geq G_{k_N}^N(n+j)$ for all $j\in \N_0$, and likewise for $G_y$. Using (iii), (iv), and (i) for all $N\geq N'(n')$ and for all $n\geq n'$, $\lvert  G_{k_N}^N(n)-G_y(n)\rvert\leq  \lvert G_{k_N}^N(n)\rvert +\lvert G_y(n)\rvert\leq \lvert G_{k_N}^N(n')\rvert +\lvert G_y(n')\rvert\leq \varepsilon,$ giving in particular a uniform bound for $n\geq n'$. In addition, also by (ii), there is $N''\in\N$ such that for all $n\in [n']$ and $N\geq N''$, $\lvert G_{k_N}^N(n)-G_y(n)\rvert\leq\varepsilon$. In particular for all $N\geq \max\{N',N''\}$, $\sup_{n\in \N}\lvert G_{k_N}^N(n)-G_y(n)\rvert \leq \varepsilon$.
\end{proof} 

\begin{proof}[Proof of Corollary~\ref{cor:repr.absorpt.prob.diff}]
	Assume $\theta=0$. It is well known that in this case the Wright--Fisher diffusion absorbs in $0$ or~$1$ in finite time almost surely (see e.g.~\citep[Prop.~2.29]{cordero2019general}). Moreover, $\Rs$ is irreducible, non-explosive, and positive recurrent (see~\citep[Thm.~2.20]{cordero2019general}), and therefore converges to the unique stationary distribution $\pi_{\Rs}$ as $t\to\infty$~\citep[Thm. 3.5.3, Thm. 3.6.2]{norris1998markov}. Set~$n=1$ in the duality in Theorem~\ref{thm:momentduality} and let $t\to\infty$ to obtain the first result. For the second result, note that the time to absorption of~$\Rs$ in~$0$ or~$\Delta$ is dominated by an exponential random variable with parameter $\theta\nu_0$ (a single beneficial mutation already leads to absorption). Moreover, $\Ys$ has the stationary distribution $\pi_{\Ys}$ \eqref{piY}. Fix $y \in (0,1)$ in Theorem~\ref{thm:momentduality} and take $t\to\infty$ to obtain the second result.
\end{proof}

\subsection{Proofs connected to the pLD-ASG in the diffusion limit}\label{sec:pldASGdifflimit}
Recall the infinitesimal generator of $\Ls$ in Section~\ref{mr:pldASGdifflimit}. $\Ls$ is dominated by a non-explosive branching process in the same way as~$\Rs$ is dominated and therefore \begin{equation}
\Ls\text{ is non-explosive}.\label{eq:pLDnon-explosive}
\end{equation}

First, we sketch how to prove $\bar{L}^{(N)}\xRightarrow[]{N\to\infty} \Ls.$
\begin{proof}[Proof of Proposition~\ref{prop:convergencepLD}]
First one establishes the  result analogous to Lemma~\ref{lem:kASGratesconvergence} for $\As_{\Ls}$. Since the only delicate part concerns selection, the proof works analogously. Next, the same line of argument used to prove Proposition~\ref{prop:convergencekASG} also completes the proof of the current proposition.
\end{proof}

Recall the definition of the operator $\As_{\widetilde \Ys}$ around \eqref{eq:tildeYlimitgen}. The proof of existence of an associated process~$\widetilde{\Ys}$ and convergence of $\widetilde{Y}^{(N)}\xRightarrow[]{N\to\infty}\widetilde{\Ys}$ is essentially contained in~\citep[Eq. (11) ff.]{Taylor2007} and~\citep[Sects.~4 and 5]{barton2004}. We recall it here.
\begin{proof}[Proof (sketch) of existence of $\widetilde{\Ys}$ and convergence] For the proof it is convenient to embed~$[0,1]$ into the product space~$E=[0,1]\times \{0,1,2\}$. Our aim is to define a "finer" process $\widetilde{\Ys}^{(\mathrm{f})}$ via an auxiliary generator acting on functions $f:E\to\R$, twice continuously differentiable in the first component, defined as  $$\As_{\widetilde{\Ys}^{(\mathrm{f})}}f(y,2)=\As_{\Ys}f(y,2)+\frac{1-y}{y}\theta\nu_1 [f(y,0)-f(y,2)]+\frac{y}{1-y} \theta \nu_0[f(y,1)-f(y,0)]$$ for $y\in (0,1)$ and $\As_{\Ys}$ taking $f(y,2)$ as a function of just the first component; complemented by $\As_{\widetilde{\Ys}^{(\mathrm{f})}}f(y,i)=0$ for $(y,i)\in E\setminus\{(y,2):y\in (0,1)\}$. This process is finer because the absorbing states indicate how the process got there, as that the second component of~$E$ becomes~$0$ or~$1$ once a jump occurs. $\widetilde{Y}^{(N)}/N$ can also be embedded into the same state space, and we denote the embedded version by $\widetilde{Y}^{(N,\mathrm{f})}$. We first sketch how to prove existence of $\widetilde{\Ys}^{(\mathrm{f})}$ using a localisation argument. For $k\geq 3$, define $U_k\coloneqq (1/k,1-1/k)\times \{0,1,2\}$. As a first step one proves that the closure of $\As_{\widetilde{\Ys}^{(\mathrm{f})}}$ when restricted to $U_k$ generates a Feller semigroup (analogous to \citep[Lem.~4.1]{barton2004}). In a second step one verifies that the associated stopped (when exiting $U_k$) martingale problem is well-posed. It is proved in~\citep[Lem.~2.1]{Taylor2007} that the time the (stopped) processes are actually stopped is almost surely unbounded as $k\to\infty$. This result together with the fact that for each~$k$, $U_k$ is an open set in the extended space, is used to deduce that a unique Markov process $\widetilde{\Ys}^{(\mathrm{f})}$ with generator $\As_{\widetilde{\Ys}^{(\mathrm{f})}}$ exists (see~\citep[Thm.~4.2]{barton2004} for details). To prove convergence, one first shows that for each~$k$, $\widetilde{Y}^{(N,\mathrm{f})}$ stopped when leaving~$U_k$ converges to $\widetilde{\Ys}^{(\mathrm{f})}$ stopped when leaving $U_k$ as $N\to\infty$. This is true because the generator of $\widetilde{Y}^{(N,\mathrm{f})}$ converges uniformly to the generator of $\widetilde{\Ys}^{(\mathrm{f})}$ when restricted to $U_k$ as $N\to\infty$ (see also \citep[Thm.~5.2]{barton2004}). Moreover, $\widetilde{Y}^{(N, \mathrm{f})}$ stopped when exiting $U_k$ converges, as $N\to\infty$, to $\widetilde{\Ys}^{(\mathrm{f})}$ on events that are measurable with respect to the $\sigma$-algebra of the first-exit time out of~$U_k$. Finally, because the exit time of $\widetilde{\Ys}^{(\mathrm{f})}$ out of $U_k$ increases to $\infty$ as $k\to\infty$ almost surely, one deduces using~\citep[Lem.~5.3]{barton2004} that $\widetilde{Y}^{(N,\mathrm{f})}$ converges to $\widetilde{\Ys}^{(\mathrm{f})}$ as $N\to\infty$. The corresponding statement for $\widetilde{Y}^{(N)}$ and $\widetilde{\Ys}$ follows after identifying $(y,1)$ (resp. $(y,0)$) with~$1$ (resp. $0$) and considering the projection of the fine processes to the first component of~$E$.
\end{proof}

We close this section with the proof of the  moment duality between $\widetilde{\Ys}$ and $\Ls$, and its corollary.
\begin{proof}[Proof of Theorem~\ref{thm:tildemomentduality}]
	It is straightforward to show that for all $y\in[0,1]$ and $n\in \N$, $\As_{\widetilde{\Ys}}\mathcal{H}(\cdot,n)(y)=\As_{\Ls}\mathcal{H}(y,\cdot)(n)$. Moreover, $\mathcal{H}(y,n)\leq 1$. The result follows from~\citep[Cor.~4.4.13]{Ku86}. 
\end{proof}
\begin{proof}[Proof of Corollary~\ref{coro:hittingprobab}]
	The first assertion follows from Theorem~\ref{thm:tildemomentduality} and the definition of $\mathfrak{h}_r$ in \eqref{eq:diffancestraltyp}. It follows from~\citep[Lem. 2.1]{Taylor2007} that $\widetilde{T}_0\wedge \widetilde{T}_1$ is almost surely finite which yields the second part.
\end{proof}

\section{Proof of the fixation probability under moderate selection} \label{sec:applications}
\begin{proof}[Proof of Proposition~\ref{prop:haldane}]
The proof is a brute-force calculation. Denote the birth and death rates of $Y$ in the setting of the proposition by $\lambda_k = k \frac{N-k}{N}$ and $\mu_k = k(\frac{N-k}{N} + \frac{\sigma m}{N^\alpha} (1-(\frac{k}{N})^m))$, respectively. Set $q_k\defeq \lambda_k/\mu_k=(1+\frac{\sigma}{N^\alpha} \frac{1-(k/N)^m}{1-k/N})^{-1}$ and recall from Section~\ref{sec:mainresukt:subsec:asg} the definition of $h_{\infty}$. Clearly, $h_{\infty}(0)=0=1-h_{\infty}(N)$. A classic first-step analysis of $h_{\infty}$ yields for $k\in[N]$ that $$h_{\infty}(k)=\frac{1+\sum_{\ell=1}^{k-1}\prod_{j=1}^{\ell} q_j^{-1}}{1+\sum_{\ell=1}^{N-1}\prod_{j=1}^{\ell} q_j^{-1}}.$$
This  leads to \begin{equation} \label{eq:9}
1-h_{\infty}(N-1)=\frac{\prod_{k=1}^{N-1} q_k^{-1}}{1+\sum_{\ell=1}^{N-1}\prod_{k=1}^{\ell} q_k^{-1}}=\frac{1}{\sum_{\ell=1}^{N}\prod_{k=\ell}^{N-1} q_k}=\frac{1-q_{N-1}}{(1-q_{N-1})\sum_{\ell=1}^{N}\prod_{k=\ell}^{N-1} q_k}.\end{equation} 
We will now show that, as $N\to\infty$, \begin{enumerate}[label=(\roman*)]
	\item $1-q_{N-1}=\frac{\sigma}{N^{\alpha}}(m+\mathcal{O}(1/N))$,
	\item the denominator in the right-hand side of~\eqref{eq:9} is $1+\scO(1)$.
\end{enumerate}  
From (i) and (ii) the result follows. For (i), note that by the binomial theorem and properties of the geometric series, we have 
\begin{equation}\label{eq:1}
	 \sum^{m-1}_{j=0} \binom{m}{j} \Big( \frac{k}{N}\Big)^j \Big( \frac{N-k}{N}\Big)^{m-1-j}=\frac{1- \big(\frac{k}{N}\big)^m}{1-\frac{k}{N}}=\sum^{m-1}_{j=0} \left(\frac{k}{N}\right)^j.
\end{equation}
Specialising the first identity of~\eqref{eq:1} to $k=N-1$, we deduce \begin{equation}\label{eq:6}
	\frac{1-\big(\frac{N-1}{N}\big)^m}{1-\frac{N-1}{N}}  =  m \left( \frac{N-1}{N}\right)^{m-1} + \mathcal{O} \left(\frac{1}{N}\right) = m + \mathcal{O} \left( \frac{1}{N}\right)\, . \nonumber
\end{equation}
Thus, $q_{N-1}=(1+\frac{m \sigma }{N^\alpha} (1+\mathcal{O}(1/N)))^{-1}$ and (i) follows after some straightforward algebra.

\smallskip

For (ii), write the denominator as $1+A_N-B_N$, where
 \begin{equation}\label{eq:10}
      A_N \defeq \sum\limits^N_{\ell = 2} \prod\limits^{N-1}_{k=N-\ell +1} q^{}_k - q^{}_{N-1} \sum\limits^{N-1}_{\ell =1}\prod\limits^{N-1}_{k=N-\ell + 1} q^{}_k \quad \text{and } 
      B_N \defeq q^{}_{N-1} \prod\limits^{N-1}_{k=1} q^{}_k\, .
\end{equation}
We will show that $A_N,B_N\to 0$ as $N\to\infty$. First, we deal with $(B_N)_{N>0}$. Choose $\beta$ such that  $\alpha < \beta <1$. Define $n_N\defeq \lceil N^\beta\rceil$ and $\bar{q} \defeq  q^{}_{N-n_N}.$ The second identity of~\eqref{eq:1} implies that the denominator of $q_k$ is increasing in~$k$ (so $q_k$ is decreasing in $k$). Since $q_k<1$ for all~$k$, and by the definition of~$\bar{q}$, we have
\[
B_N  =  q^{}_{N-1} \Big ( \prod^{N-n_N}_{j=1} q^{}_j \Big )
\Big ( \prod^{N-1}_{k=N-n_N+1} q^{}_k \Big ) < q^{}_{N-1} \prod^{N-1}_{k=N-n_N+1} q^{}_k < \bar{q}^{n_N}.
\]
We now show that $\bar{q}^{n_N}\xrightarrow{N\to\infty} 0$, which implies the claim for $(B_N)_{N>0}$. By the first identity in~\eqref{eq:1}, 
\begin{equation*}\label{eq:11}
    \frac{1-\big(\frac{N-n_N}{N}\big)^m}{1- \frac{N-n_N}{N}} =  m \Big( 1-\frac{n_N}{N}\Big)^{m-1} + \mathcal{O} \left(\frac{n_N}{N}\right) =  m \big( 1+ \scO (1)\big)\, .
\end{equation*}
This leads to
\begin{equation}\label{eq:12}
\bar{q}= \big(1+m \frac{\sigma}{N^\alpha} \big( 1+ \scO(1)\big)\big)^{-1}.
\end{equation}
But
\[
 \bar{q}^{n_N}  =  \Big [  \Big ( \frac{1}{1+m\frac{\sigma}{N^\alpha}\big( 1+ \scO(1)\big )}  \Big )^{N^\alpha}\Big ]^{\frac{n_N}{N^\alpha}}  = \scO(1) \quad \text{as } \; N\to\infty.
\]
To see this, note that
\begin{equation}\label{eq:econv}
 \Big ( 1+m\frac{\sigma}{N^\alpha}\big( 1+ \scO(1)\big )  \Big )^{N^\alpha} = \ee^{m \sigma }(1+\sscO(1)) \quad \text{as } \;  N \to \infty, \quad \text{and } \;  \lim_{N \to \infty} \frac{n_N}{N^\alpha} = \infty.
\end{equation}

\smallskip

Next, we show that $A_N\xrightarrow{N\to\infty}0$. Shifting index in the first sum of \eqref{eq:10}, we rewrite
\begin{equation}\label{eq:14}
  A_N  =   \sum^{N-1}_{\ell =1} (q^{}_{N-\ell} - q^{}_{N-1}) \prod^{N-1}_{k=N-\ell +1} q^{}_k\, .
\end{equation}
We first bound $q^{}_{N-\ell} - q^{}_{N-1}$. Since $x\mapsto x + 1/x$ is monotone increasing for $x>0$ and the $q^{}_{k}$ are decreasing in~$k$, we can write
\[
  q^{}_{N-\ell} - q^{}_{N-1}  < \frac{1}{q^{}_{N-1}} - \frac{1}{q^{}_{N-\ell}} = \frac{\sigma}{N^\alpha}
 \Big (\frac{1-x_1^m}{1-x^{}_1} - \frac{1-x_2^m}{1-x^{}_2} \Big )  \leq \frac{\sigma d (\ell-1)}{N^{\alpha+1}}\, ,
\]
where $x^{}_1 \defeq (N-1)/N$, $x^{}_2 \defeq (N-\ell)/N$, and
$d:= \max_{x\in [0,1]} f'(x)$  with  $f(x) \defeq  (1-x^m)/(1-x)$ 
($\lim_{x\to 1} f(x)$ and $\lim_{x\to 1} f'(x)$ exist and are finite by l'Hopital's rule).
Next, let us bound the product in \eqref{eq:14}. Since the $q_k$ are decreasing, all $q_k<1$, and using the definition of~$\bar{q}$ (and $n_N$) from the previous step, we get $\prod^{N-1}_{k=N-\ell +1} q^{}_k \leqslant \bar{q}^{\min \{ \ell -1,n_N\}}$.
Thus,
\[
  A_N  \leq  \frac{\sigma d}{N^{1+\alpha}}  \sum^{N-1}_{\ell =1} (\ell -1) \, \bar{q}^{\min \{ \ell -1,n_N\}} \nonumber = \frac{\sigma d}{N^{1+\alpha}}  \sum^{N-2}_{\ell =0} \ell \, \bar{q}^{\min \{ \ell ,n_N\}} 
  = \sigma d ( A_N^{(1)} + A_N^{(2)} ),
  \]
where 
\[
  A_N^{(1)}  \defeq \frac{1}{N^{1+\alpha}} \sum^{n_N}_{\ell =0} \ell \, \bar{q}^\ell \quad \text{and } \; A_N^{(2)} \defeq \frac{1}{N^{1+\alpha}} \sum^{N-2}_{\ell = n_N+1} \ell \, \bar{q}^{n_N}.
\]
Use that $\sum_{\ell\geq 0}\ell x^\ell=x/(1-x)^2$ and~\eqref{eq:12} to bound $A_N^{(1)}$, 
\[
A_N^{(1)}  \leq \frac{1}{N^{1+\alpha}}\frac{\bar{q}}{(1-\bar{q})^2 }= \frac{1}{N^{1+\alpha}}  \frac{1+m\frac{\sigma}{N^\alpha}\big(1+\scO(1)\big)}{\big(m \frac{\sigma}{N^\alpha}(1+\scO(1))\big)^2} 
   = \frac{1}{m^2\sigma^2} N^{\alpha-1} \big( 1+\scO(1)\big) 
   = \scO(1) \quad \text{as } \; N \to \infty \, .
\]
For $A_N^{(2)}$, use $\sum^{N-2}_{\ell =1} \ \ell \leqslant N^2/2$ to obtain
\[
A_N^{(2)} \leq  \frac{\bar{q}^{n_N}}{N^{1+\alpha}}   \sum^{N-2}_{\ell =1} \ \ell \leq \frac{\bar{q}^{n_N}}{2} N^{1-\alpha}   = \frac{N^{1-\alpha}}{2}  \big ( \bar{q}^{N^\alpha}\big )^{\frac{n_N}{N^\alpha}}  = \frac{N^{1-\alpha}}{2}  \big [ e^{-m\sigma} (1+\sscO(1)) \big ]^{\frac{n_N}{N^\alpha}} = \scO(1),
\]
where in the last two steps the same argument as in~\eqref{eq:econv} is used. The statement is thus proved.
\end{proof}

\section*{Acknowledgements}
We are grateful to Fernando Cordero, Martin M\"ohle, Cornelia Pokalyuk, and Anton Wakolbinger for enlightening discussions and to two anonymous referees for insightful comments. We acknowledge funding by the German Research Foundation (DFG, Deutsche
Forschungsgemeinschaft) --- SFB 1283, project C1. 
\addtocontents{toc}{\protect\setcounter{tocdepth}{2}}
\bibliographystyle{abbrvnat}
%\bibliography{Literature}

\end{document}